\documentclass[letter,12pt]{article}
\usepackage{extsizes}
\usepackage{blindtext}
\usepackage[utf8]{inputenc}
\usepackage{amsmath}
\usepackage{amssymb}
\usepackage{amsthm}
\usepackage{pst-node}
\usepackage{blindtext}
\usepackage{geometry}
 \usepackage{float}
 \usepackage{amssymb}
\usepackage{tikz-cd}
\usepackage{amscd}
\usepackage[all,color]{xy}
\usepackage{sseq}
\usepackage{amsfonts}
\usepackage{enumerate}
\usepackage{graphicx}
\usepackage{fancyhdr}
\usepackage{tikz}
\usepackage{graphics}
\usepackage{hyperref}
\usepackage[title]{appendix}
\usepackage[numbers]{natbib}

\theoremstyle{plain} \numberwithin{equation}{section}

\newtheorem{theorem}{Theorem}[section]

\newtheorem{corollary}[theorem]{Corollary}

\newtheorem{lemma}[theorem]{Lemma}

\newtheorem{proposition}[theorem]{Proposition}
\theoremstyle{definition}
\newtheorem{definition}[theorem]{Definition}

\newtheorem{remark}[theorem]{Remark}
\newtheorem{example}[theorem]{Example}

\def \f-{f^{-1}}
\def \fp-{f^{-1}_\partial}

\usepackage{hyperref}
\usepackage{caption}
\usepackage{subcaption}

\title{Half grid diagrams and Thompson links}
\author{Yangxiao Luo, Shunyu Wan}
\date{}
\usepackage{natbib}
\usepackage{graphicx}

\begin{document}

\maketitle

\begin{abstract}
We define half grid diagrams and prove every link is half grid presentable by constructing a canonical half grid pair (which gives rise to a grid diagram of some special type) associated with an element in the oriented Thompson group. We show that this half grid construction is equivalent to Jones' construction of oriented Thompson links in \cite{MR3589908}. Using this equivalence, we relate the (oriented) Thompson index to several classical topological link invariants, and give both the lower and upper bounds of the maximal Thurston-Bennequin number of a knot in terms of the oriented Thompson index. Moreover, we give a one-to-one correspondence between half grid diagrams and elements in symmetric groups and give a new description of link group using two elements in a symmetric group.
\end{abstract}

\section{Introduction} 
Grid diagrams have many interesting applications to knot theory and low dimensional topology (for example, relation to Legendrian and transverse knots \cite{MR2890458}), and they provide a combinatorial way to compute some knot invariants such as knot Floer homology \cite{MR2480614} and Khovanov homology \cite{MR2520400}. In recent years, there have been plenty of studies about slicing a ``closed" object into two or more ``bordered" objects to obtain information from pieces, including bordered Heegaard Floer homology \cite{MR3827056}, bordered knot Floer homology \cite{szabo2019algebras}, \cite{ozsvath2022pong}, and bordered Khovanov homology \cite{MR3584271}, \cite{MR3474320}. It is well-known that every link has a grid diagram representing it, so slicing a link can be realized by slicing a grid diagram.

It is natural to ask what a ``good" bordered grid diagram could be and whether we can split grid diagrams and get information from pieces. To answer these questions, let's first recall some definitions related to grid diagrams.

\begin{definition} \label{def: grid diagram}
	A grid diagram $\mathbb{G}$ is an $m \times m$ grid of squares, each of which is either empty or contains an $X$ or an $O$. We require that in each row and column there are exactly two nonempty squares, one containing an $X$ and the other containing an $O$.
\end{definition}

Any grid diagram $\mathbb{G}$ has an associated link $\mathcal{L}(\mathbb{G})$ defined as follows.

\begin{definition}
	Let $\mathbb{G}$ be an $m \times m$ grid diagram. The oriented link $\mathcal{L}(\mathbb{G})$ associated to $\mathbb{G}$ is obtained by connecting $X$ to $O$ in each row and connecting $O$ to $X$ in each column such that horizontal segments are always above vertical segments.
\end{definition}

See Figure \ref{fig: grid diagram example} for an example of a grid diagram and its associated link.

\begin{remark}
	The orientation and crossing conventions we are using on paper are slightly different from the usual conventions in \cite{MR2480614}. The usual orientation convention is to connect $O$ to $X$ in each row and connect $X$ to $O$ in each column, with the crossing convention that vertical segments are always above horizontal segments. However, if we rotate a grid diagram with an associated link in our conventions by 90 degrees, it will become a grid diagram with an associated link in the usual conventions. The reason for using our conventions is to make the over and under crossing compatible with the ones from Thompson links (Section \ref{sec: jones construction}).
\end{remark}

We now define half grid diagrams, which give rise to oriented tangles.

\begin{definition} \label{def: half grid diagram}
	A half grid diagram $\mathbb{H}$ is an $n \times 2n$ grid of squares, each of which is either empty or contains an $X$ or an $O$. We require that in each row there are exactly two nonempty squares, one containing an $X$ and the other containing an $O$, and each column has exactly one nonempty square (which may contain either an $X$ or an $O$). 
\end{definition}

\begin{definition}
	Let $\mathbb{H}$ be an $n \times 2n$ half grid diagram. The $(0,2n)$-oriented tangle $\mathcal{T}(\mathbb{H})$ associated to $\mathbb{H}$ is defined by first connecting $X$ to $O$ in each row, then connecting each $X$ or $O$ to the bottom of the grid by vertical arcs, oriented upward in the case of an $X$ and downward in the case of an $O$. We again use the crossing convention that horizontal segments are always above vertical segments.
\end{definition}

See Figure \ref{fig: half grid diagram example} for an example of a half grid diagram and its associated tangle.

\begin{figure} 
    \centering
    \begin{subfigure}{0.35\textwidth}
        \centering
        \includegraphics[width=\textwidth]{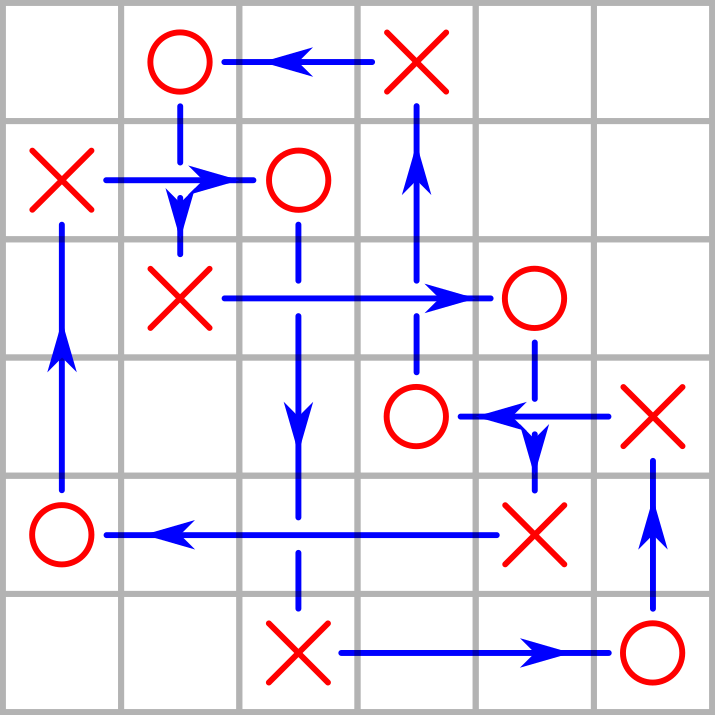}
        \caption{}
        \label{fig: grid diagram example}
    \end{subfigure}
    \hfill
    \begin{subfigure}{0.45\textwidth}
        \centering
        \includegraphics[width=\textwidth]{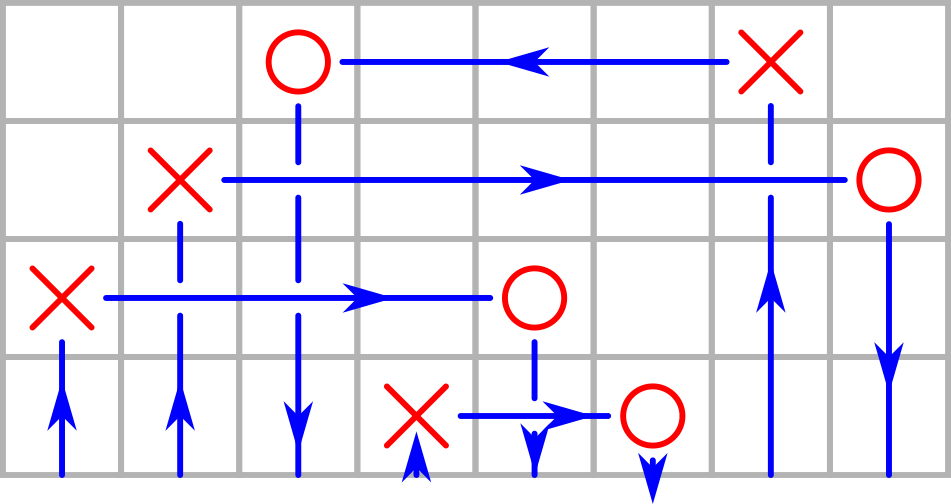}
        \caption{}
        \label{fig: half grid diagram example}
    \end{subfigure}
    \caption{(a) A $6 \times 6$ grid diagram and its associated link. (b) A $4 \times 8$ half grid diagram and its associated tangle.}   
\end{figure}

\begin{remark}
	Throughout the paper, we will view a grid or a half grid diagram as a subset of lattice points $\mathbb{Z}^2 \in \mathbb{R}^2$ where the grid at the lower left corner corresponds to point $(1,1)$. To describe a grid diagram or half grid diagram we just need to specify the coordinates of $X$'s and $O$'s on lattice points.
\end{remark}

Next, we give a compatible condition describing when two half grid diagrams can reproduce a grid diagram.

\begin{definition}
	Let $\mathbb{H}_+$ and $\mathbb{H}_-$ be a pair of $n \times 2n$ half grid diagrams. We say they are compatible, or say $(\mathbb{H}_+,\mathbb{H}_-)$ is a compatible half grid pair, if the $i^{th}$ column of $\mathbb{H}_+$ and the $i^{th}$ column of $\mathbb{H}_-$ contain the same symbol for each $i$.
\end{definition}

Given two compatible $n \times 2n$ half grid diagrams $\mathbb{H}_+$ and $\mathbb{H}_-$, we can obtain a $2n \times 2n$ grid diagram as follows: We first flip $\mathbb{H}_-$ upside down and change each $X$ to $O$ and $O$ to $X$ to get $-\overline{\mathbb{H}_-}$, then we stack $\mathbb{H}_+$ over $-\overline{\mathbb{H}_-}$ (see Figure \ref{fig: compatible half grid pair example} for an example). The resulting diagram, still denoted as $(\mathbb{H}_+, \mathbb{H}_-)$, is a $2n \times 2n$ grid diagram, because the the $i^{th}$ column of $\mathbb{H}_+$ and the $i^{th}$ column of $-\overline{\mathbb{H}_-}$ contain opposite markings for each $i$. It is easy to observe that $\mathcal{L}(\mathbb{H}_+, \mathbb{H}_-) = (-\overline{\mathcal{T}({\mathbb{H}_-})}) \mathcal{T}(\mathbb{H}_+)$, where $-\overline{\mathcal{T}({\mathbb{H}_-})}$ is the vertical mirror of $\mathcal{T}({\mathbb{H}_-})$ with reversed orientation.

\begin{figure}
	\centering
	\includegraphics[width = 0.5\textwidth]{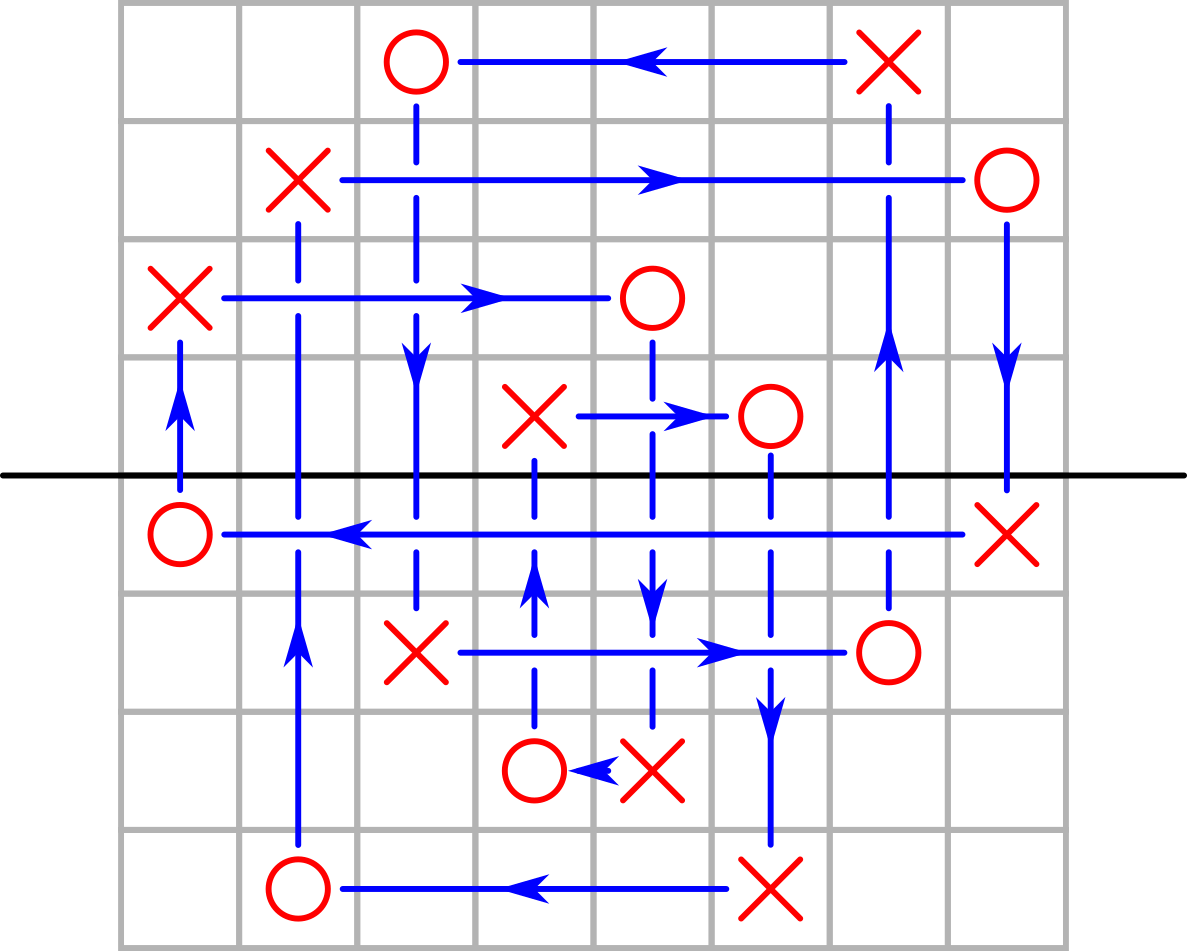}
    \put(0, 40){$-\overline{\mathbb{H}_-}$}
    \put(0, 130){$\mathbb{H}_+$}
	\caption{A pair of compatible half grid diagrams $\mathbb{H}_+$ and $\mathbb{H}_-$ can reproduce a grid diagram $(\mathbb{H}_+, \mathbb{H}_-)$.}
	\label{fig: compatible half grid pair example}
\end{figure}

\begin{remark} \label{rmk: reverse orientation of mirror}
	Notice that $(\mathbb{H}_+,\mathbb{H}_-)$ and $(\mathbb{H}_-,\mathbb{H}_+)$ are not the same grid diagram, so the order of $\mathbb{H}_+$ and $\mathbb{H}_-$ matters. The relation between $\mathcal{L}(\mathbb{H}_+,\mathbb{H}_-)$ and $\mathcal{L}(\mathbb{H}_-,\mathbb{H}_+)$ is $\mathcal{L}(\mathbb{H}_+,\mathbb{H}_-) = -\overline{\mathcal{L}(\mathbb{H}_-,\mathbb{H}_+)}$.
\end{remark}

We just saw that any compatible half grid pair $(\mathbb{H}_+,\mathbb{H}_-)$ generates an oriented link $\mathcal{L}(\mathbb{H}_+,\mathbb{H}_-)$. One can ask about the converse:  can every oriented link be generated by a compatible half grid pair? We give an affirmative answer to it.

\begin{definition} \label{def: half grid presentable}
	An oriented link $L$ is half grid presentable if there exists compatible half grid diagrams $\mathbb{H}_+$ and $\mathbb{H}_-$ such that $\mathcal{L}(\mathbb{H}_+,\mathbb{H}_-)=L$, and we call $(\mathbb{H}_+,\mathbb{H}_-)$ a half grid representative of $L$.
\end{definition}

\begin{theorem} \label{thm: link is half grid presentable}
	Every oriented link is half grid presentable.
\end{theorem}

We will prove Theorem \ref{thm: link is half grid presentable} by referring to the result in \cite{MR4071378} that every oriented link is isomorphic to $L_g$ (called an oriented Thompson link) corresponding to some $g$ in the oriented Thompson group $\vec{F}$, then showing that every oriented Thompson link is half grid presentable (Section \ref{sec: grid construction}).

\bigskip

Half grid construction of links expands the bridge between classical link theory and Thompson link theory established by Jones. For example, our first application relates the number of link components to leaf number of Thompson group element. We denote the number of link components of oriented link $L$ as $|L|$.

\begin{theorem} \label{thm: parity between link component and leaf number}
    If $g \in \vec{F}$ has leaf number $n$, then we have
	$$
	|L_g|=n \pmod{2}
	$$
\end{theorem}

Grid diagrams are closely related to front projection of the Legendrian links in $S^3$ with the standard tight contact structure. See section \ref{sec: Legendrian knot and half grid diagram}. In the second application, we provide a lower bound for both the maximum Thurston-Bennequin number $\overline{tb}(K)$ and the maximum self-linking number $\overline{sl}(K)$ of oriented knots $K$ in $(S^3,\xi_{std})$ using the oriented Thompson index $ind_{\vec{F}}(K)$.

\begin{theorem} \label{thm: inequality between Thompson index and max tb number}
	For any oriented knot $K$ in $(S^3,\xi_std)$, we have
	\begin{equation} \label{eq: short inequality}
		-ind_{\vec{F}}(K) \leq \overline{tb}(K).
	\end{equation}
	More generally,
	\begin{equation} \label{eq: generalized inequality}
		\max{(-ind_{\vec{F}}(K), -ind_{\vec{F}}(-K))} \leq  \min{(\overline{tb}(K), \overline{tb}(\overline{K}))}
	\end{equation}
	where $-K$ has reversed orientation of $K$, and $\overline{K}$ is the mirror of $K$. Similarly, we have \[-ind_{\vec{F}}(K) \leq \overline{sl}(K),\] and \[\max{(-ind_{\vec{F}}(K), -ind_{\vec{F}}(-K))} \leq  \min{(\overline{sl}(K), \overline{sl}(\overline{K}))}.\]
\end{theorem}

Apart from the relation between contact geometry invariants and the oriented Thompson index, we also have purely topological relationships between some knot invariants and the (oriented) Thompson index.

\begin{theorem} \label{thm: Euler charactersitic of Thompson link}
    For any oriented link $L$, we have  $$\chi(L)\geq -ind_{\Vec{F}}+2,$$ where $\chi(L)$ is the maximal Euler characteristic number along all Seifert surfaces of $L$. As in the special case when $L$ is a knot, we have $$g(L) \leq \frac{ind_{\Vec{F}}-1}{2},$$ where $g(L)$ is the Seifert genus of $L$.
\end{theorem}

Using the Thurston-Bennequin inequality and Theorem \ref{thm: inequality between Thompson index and max tb number}, we immediate have the following corollary which gives both upper and lower bound for $\overline{tb}(K)$ and $\overline{sl}(K)$.

\begin{corollary}
    For any oriented knot $K$, we have \[-ind_{\vec{F}}(K) \leq \overline{tb}(K)\leq ind_{\vec{F}}(K)-2,\] and \[-ind_{\vec{F}}(K) \leq \overline{sl}(K)\leq ind_{\vec{F}}(K)-2.\] 
\end{corollary}

After we see the unoriented version of half grid construction in Section \ref{sec: sym rep and unoriented grid}, we can give an upper bound of the minimal grid number $grid(L)$ using unoriented Thompson index $ind_{F}(L)$, as our third application.

\begin{theorem} \label{thm: grid number inequality}
	For any link $L$, we have
	$$    
	grid(L) \leq 2 \cdot ind_{F}(L).
	$$ As a consequence of the relation between grid number $grid(L)$, braid index $braid(L)$, and bridge number $bridge(L)$, we have
 \[bridge(L)\leq braid(L)\leq grid(L)/2 \leq ind_F(L).\]

\end{theorem}

Apart from its natural connection to Thompson links, the half grid presentation itself gives a ``symmetric" way to describe links. The first set is to encode any $n \times 2n$ half grid diagram by an element in symmetric group $Sym(2n)$.

\begin{theorem} \label{thm: half grid and element of symmetric group}
	There is a one-to-one correspondence between $n \times 2n$ half grid diagrams and elements in $Sym(2n)$. 
\end{theorem}

The correspondence is very simple, we just record where the $X$ and $O$ are in each row from bottom to top. For example, the half grid diagram in Figure \ref{fig: half grid diagram example} corresponds to
$$
\sigma =
\begin{pmatrix}
	1 & 2 & 3 & 4 & 5 & 6 & 7 & 8 \\
	4 & 6 & 1 & 5 & 2 & 8 & 7 & 3
\end{pmatrix}
$$

For a more detailed description , see Section \ref{sec: sym rep and unoriented grid}.

\bigskip

Theorem \ref{thm: link is half grid presentable} tells us that every oriented link can be generated by a pair of half grid diagrams. It turns out that we have a parallel result for link groups, stating that every link group has a ``half grid presentation" encoded by a pair of symmetric group elements $\sigma_+$ and $\sigma_-$.

\begin{theorem} \label{thm: half grid presentation of link group}
    A group $G$ is a link group ($G \cong \pi_1(S^3 \backslash L)$ for some link $L$ in $S^3$) if and only if for some positive integer $n$ and two elements  $\sigma_+$, $\sigma_-$ in $Sym(2n)$, $G$ has the following presentation, called a half grid presentation.
    \begin{equation} \label{eq: half grid presentation}
		\left< \scalebox{0.8}{$x_1,x_2,...,x_{2n}$} \middle\vert
        \begin{array}{l}
            \scalebox{0.8}{$X=x_1x_2...x_{2n},$}\\
            \scalebox{0.8}{$X(x_{\sigma_+(1)},x_{\sigma_+(2)}), X(x_{\sigma_+(1)}, x_{\sigma_+(2)},x_{\sigma_+(3)},x_{\sigma_+(4)}), ... ,X(x_{\sigma_+(1)},...,x_{\sigma_+(2n-2)}),$}\\
            \scalebox{0.8}{$X(x_{\sigma_-(1)},x_{\sigma_-(2)}), X(x_{\sigma_-(1)},x_{\sigma_-(2)},x_{\sigma_-(3)},x_{\sigma_-(4)}), ... ,X(x_{\sigma_-(1)},...,x_{\sigma_-(2n-2)})$}
        \end{array}\right> 
	\end{equation}
    where $X(x_{i_1}, x_{i_2}, ..., x_{i_j}) = x_1...\hat{x}_{j_1}...\hat{x}_{j_2}...\hat{x}_{j_k}...x_n$, with $\{j_1 < j_2 < ...< j_k\} = \{i_1, i_2, ..., i_k\}$. In other words, we take away $x_{i_1}, x_{i_2}, ..., x_{i_k}$ one by one from $X = x_1 x_2 ... x_{2n}$.
\end{theorem}

\subsection*{Acknowledgement}
The authors would like to thank Slava Krushkal and Tom Mark for useful suggestions. 
Yangxiao Luo was supported in part by NSF grant DMS-2105467 to Slava Krushkal. Shunyu Wan was supported in part by grants from the NSF (RTG grant DMS-1839968) and the Simons Foundation (grants 523795 and 961391 to Thomas Mark).

\section{Thompson group and standard dyadic partition}

\subsection{Background of Thompson group}
\label{sec: thompson_group_and_sdp}

We first give a brief introduction to Thompson group $F$ and related concepts.

\begin{definition} \label{s.d. interval and partition}
	A standard dyadic interval (s.d. interval) is an interval of the form $[\frac{k}{2^m}, \frac{k+1}{2^m}]$ for some non-negative integer $m$ and some non-negative integer $k \leq 2^m - 1$. A standard dyadic $n$-partition (s.d. $n$-partition) on $[0,1]$ is a partition $0 = a_0 < a_1 < ... < a_n = 1$ such that $[a_i, a_{i+1}]$ is an s.d. interval for any $i$.

	We call $a_i$ a breakpoint of the s.d. partition, and call $[a_i, a_{i+1}]$ a subinterval of the s.d. partition. Suppose $\mathcal{K}$ and $\mathcal{I}$ are s.d. partitions, we denote $\mathcal{K} \leq \mathcal{I}$ if $\mathcal{I}$ is a refinement of $\mathcal{K}$ as partitions. In other words, the breakpoints of $\mathcal{K}$ is a subset of the breakpoints of $\mathcal{I}$.
\end{definition}

We denote the collection of all s.d. intervals as $\mathcal{S}$, and denote $\delta$ to be the trivial partition $0=a_0 < a_1 = 1$.

\begin{definition} \label{def: Thompson group}
	Thompson group $F$ is the group (under composition) of homeomorphism $g$'s from $[0,1]$ to itself satisfying the following conditions:	
	\begin{enumerate}
		\item $g$ is piecewise linear and orientation-preserving.
		\item In the pieces where $g$ linear, the slope is always a power of $2$.
		\item The breakpoints are dyadic rational numbers, i.e. $\frac{k}{2^m}$ for some $m \in \{0,1,2,...\}$ and $k \in \{0,1,2,...,2^m\}$.
	\end{enumerate}
\end{definition}

\begin{definition} \label{def: s.d. partition pair representation}
	Suppose that $\mathcal{I}$ is a s.d. $n$-partition with breakpoints $0 = a_0 < a_1 < ... < a_n = 1$, and $\mathcal{J}$ is a s.d. $n$-partition with breakpoints $0 = b_0 < b_1 < ... < b_n = 1$. We define $\gamma^{\mathcal{I}}_{\mathcal{J}}: [0,1] \to [0,1]$ to be the piecewise linear function determined by letting $\gamma^{\mathcal{I}}_{\mathcal{J}}(a_i)=b_i$ for all $i \in \{0,1,...,n\}$ and letting $\gamma^{\mathcal{I}}_{\mathcal{J}}$ be linear on $[a_i,a_{i+1}]$ for all $i$.
	
	Note that $\gamma^\mathcal{I}_\mathcal{J}$ is an element in $F$. We say $(\mathcal{I}, \mathcal{J})$ is a pair of s.d. partitions representing $\gamma^\mathcal{I}_\mathcal{J}$.
\end{definition}

There is a bijection between the set of s.d. partitions and the set of binary trees. To illustrate that, first we label edges of the infinite binary tree $F_{\infty}$ (with an extra edge attached to the root) by standard dyadic intervals, such that the uppermost edge is labeled by $[0, 1]$ and each trivalent vertex represents a cutting at the midpoint (see Figure \ref{fig: infinite tree}). Note that this labeling gives a bijection $\epsilon: \mathcal{S} \to \mathcal{E}(F_{\infty})$, where $\mathcal{E}(F_{\infty})$ is the collection of edges of $F_{\infty}$.

\begin{figure}[ht]
	\centering
	\includegraphics[width = 0.5\textwidth]{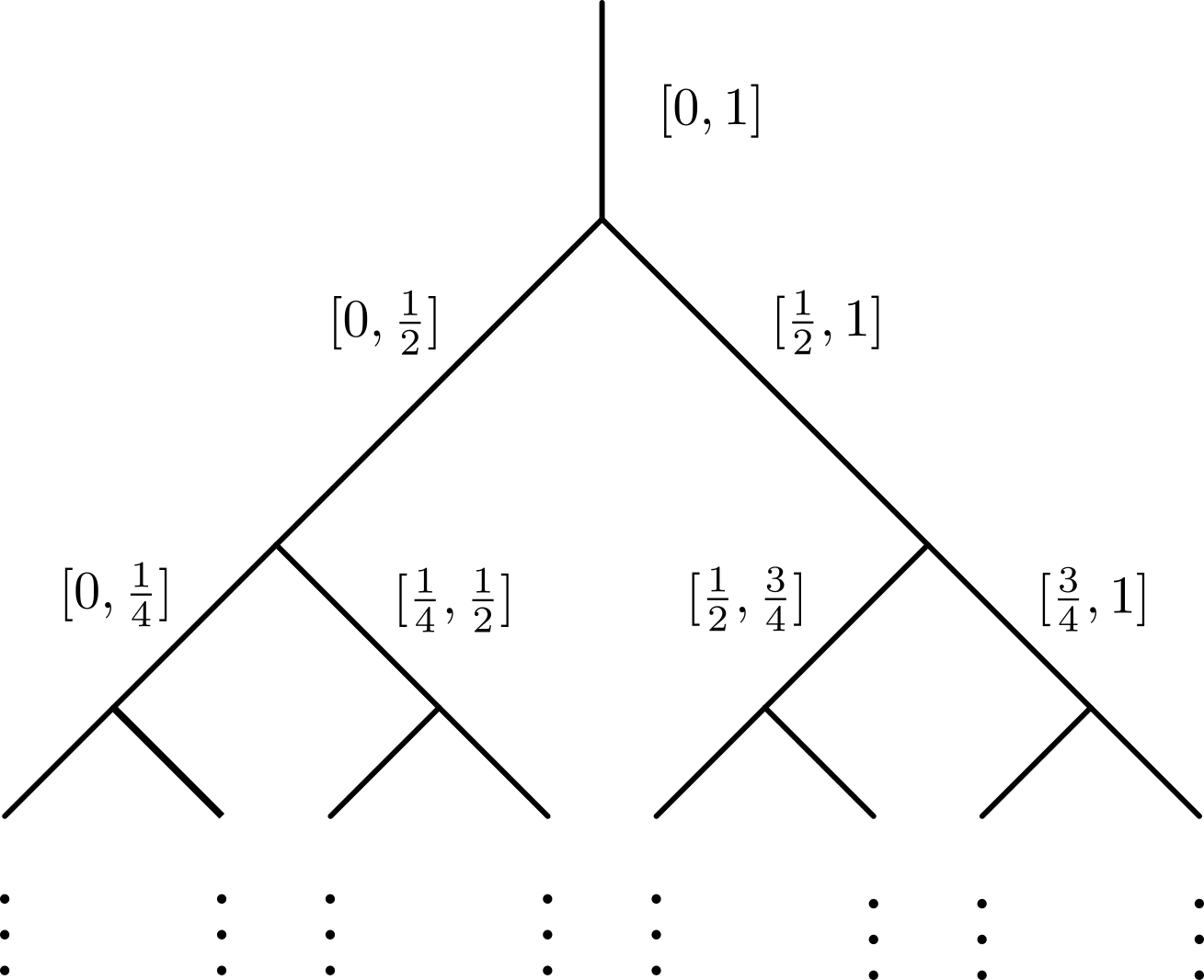}
	\put(-210, 150){$F_{\infty}$}
	\caption{The infinite tree $F_{\infty}$ with edges labeled by standard dyadic intervals.}
	\label{fig: infinite tree}
\end{figure}

Given an s.d. partition $\mathcal{I}$, we can find a unique binary tree $F_{\mathcal{I}}$ as a labeled subtree of $F_{\infty}$ whose uppermost edges are labeled by $[0,1]$, and lowermost edges are labeled by subintervals of $\mathcal{I}$ (see Figure \ref{fig: sdp and tree example} for an example). More generally, given a pair of s.d. partitions $\mathcal{K} \leq \mathcal{I}$, we can find a unique forest $F^{\mathcal{K}}_{\mathcal{I}}$ as a labeled subgraph of $F_{\infty}$ whose uppermost edges are labeled by subintervals of $\mathcal{K}$, and lowermost edges are labeled by subintervals of $\mathcal{I}$ (see Figure \ref{fig: refinement and forest example} for an example).

\begin{figure}
    \centering
    \begin{subfigure}{0.35\textwidth}
        \centering
        \includegraphics[width = \textwidth]{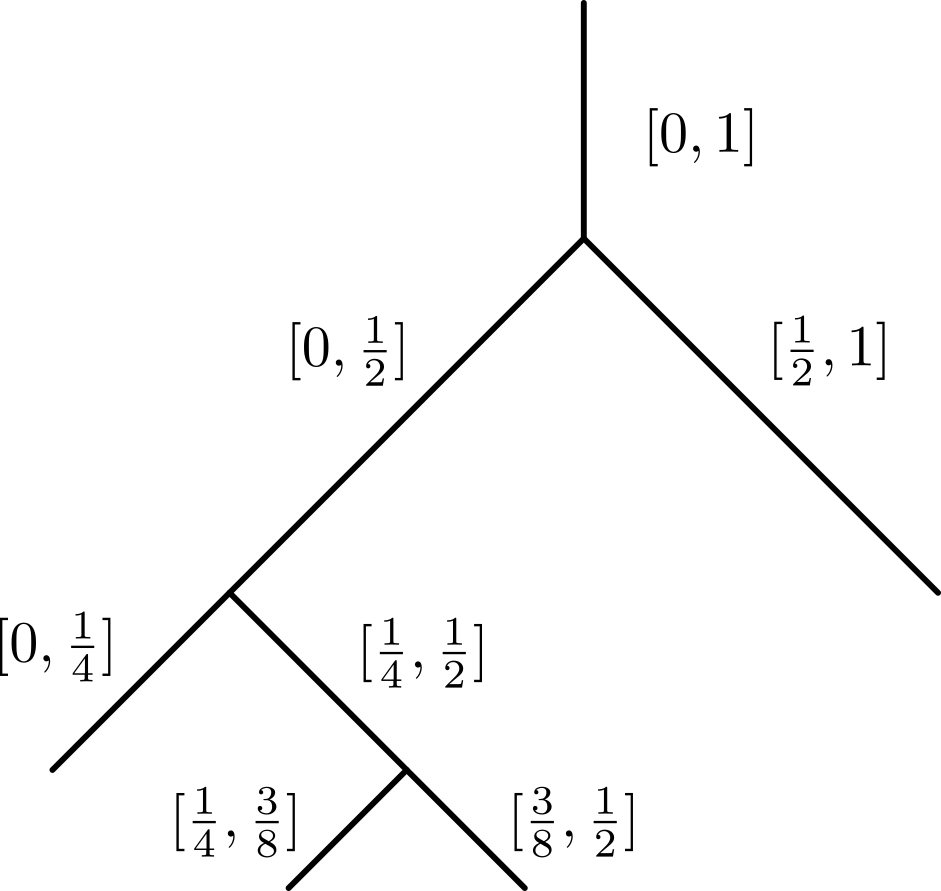}
		\put(-150, 120){$F_{\mathcal{I}}$}
		\caption{}
		\label{fig: sdp and tree example}
    \end{subfigure}
    \hfill
    \begin{subfigure}{0.4\textwidth}
        \centering
        \includegraphics[width=\textwidth]{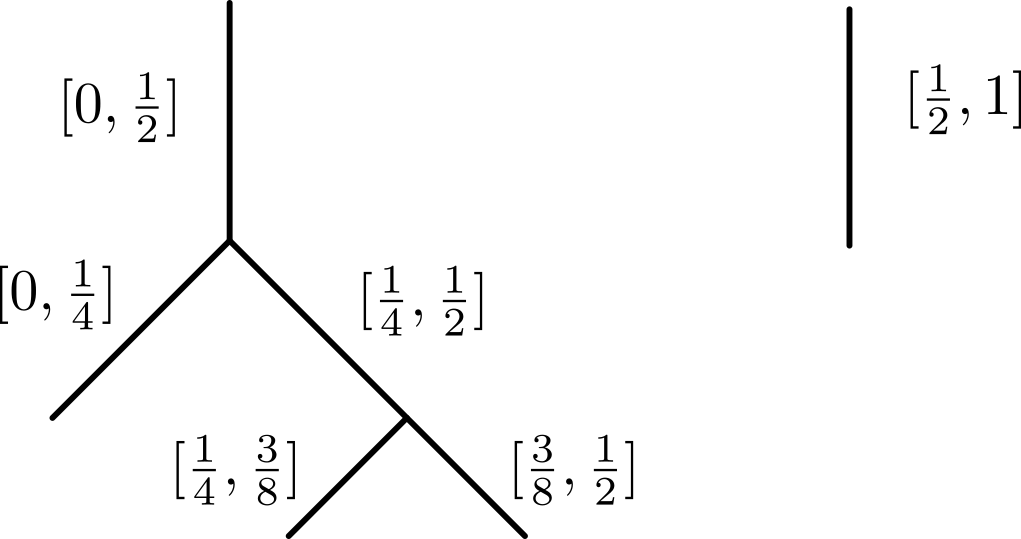}
		\put(-90, 120){$F^{\mathcal{K}}_{\mathcal{I}}$}
        \caption{}
        \label{fig: refinement and forest example}
    \end{subfigure}
    \caption{(a) $F_{\mathcal{I}}$ is the tree representing s.d. partition $\mathcal{I}$ with breakpoints $0 < \frac{1}{4} < \frac{3}{8} < \frac{1}{2} < 1$. (b) $F^{\mathcal{K}}_{\mathcal{I}}$ is the forest representing a pair of s.d. partitions $\mathcal{K} \leq \mathcal{I}$, where $\mathcal{K}$ has breakpoints $0 < \frac{1}{2} < 1$.}
\end{figure}

Using the notion of binary trees, we can visualize any pair of s.d. $n$-partitions $(\mathcal{I}, \mathcal{J})$ as a pair of binary trees $(F_{\mathcal{I}}, F_{\mathcal{J}})$ with $n$ leaves. We often take the vertical mirror of $F_{\mathcal{J}}$ and put it below $F_{\mathcal{I}}$, and call the resulting diagram a tree diagram of $(\mathcal{I}, \mathcal{J})$, also denoted as $(F_{\mathcal{I}}, F_{\mathcal{J}})$. Figure \ref{fig: tree diagram example} is an example of a tree diagram.

Given two pairs of s.d. partitions $(\mathcal{I}, \mathcal{J})$ and $(\mathcal{I}', \mathcal{J}')$, we say that $(\mathcal{I}', \mathcal{J}')$ is a refinement of $(\mathcal{I}, \mathcal{J})$ if $(F_{\mathcal{I}'}, F_{\mathcal{J}'})$ can be obtained by adding some ``carets" to $(F_{\mathcal{I}}, F_{\mathcal{J}})$, as Figure \ref{fig: carets example} shows. Equivalently, $\mathcal{I}', \mathcal{J}'$ are refinements of $\mathcal{I}, \mathcal{J}$ respectively, and these two refinements are linearly compatible, i.e. $\gamma^{\mathcal{I}'}_{\mathcal{J}'}=\gamma^{\mathcal{I}}_{\mathcal{J}}$.

We say a pair of s.d. partitions $(\mathcal{I}, \mathcal{J})$ is reduced if it is not a refinement of any other pairs. In other words, $(F_\mathcal{I}, F_{\mathcal{J}})$ contains no carets.

\begin{figure}[ht]
	\centering
	\begin{subfigure}{0.3\textwidth}
		\centering
		\includegraphics[width=\textwidth]{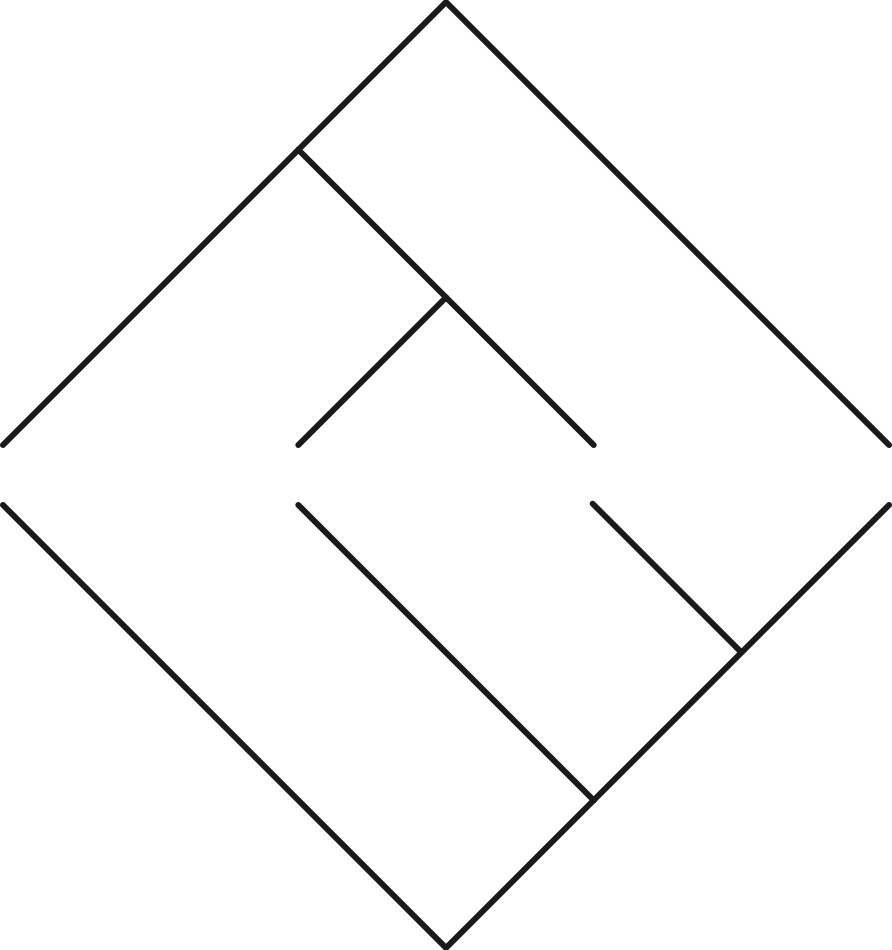}
		\put(0, 30){$F_{\mathcal{J}}$}
        \put(0, 100){$F_{\mathcal{I}}$}
		\caption{}
		\label{fig: tree diagram example}
	\end{subfigure}
	\hspace{4cm}
	\begin{subfigure}{0.3\textwidth}
		\centering
		\includegraphics[width=\textwidth]{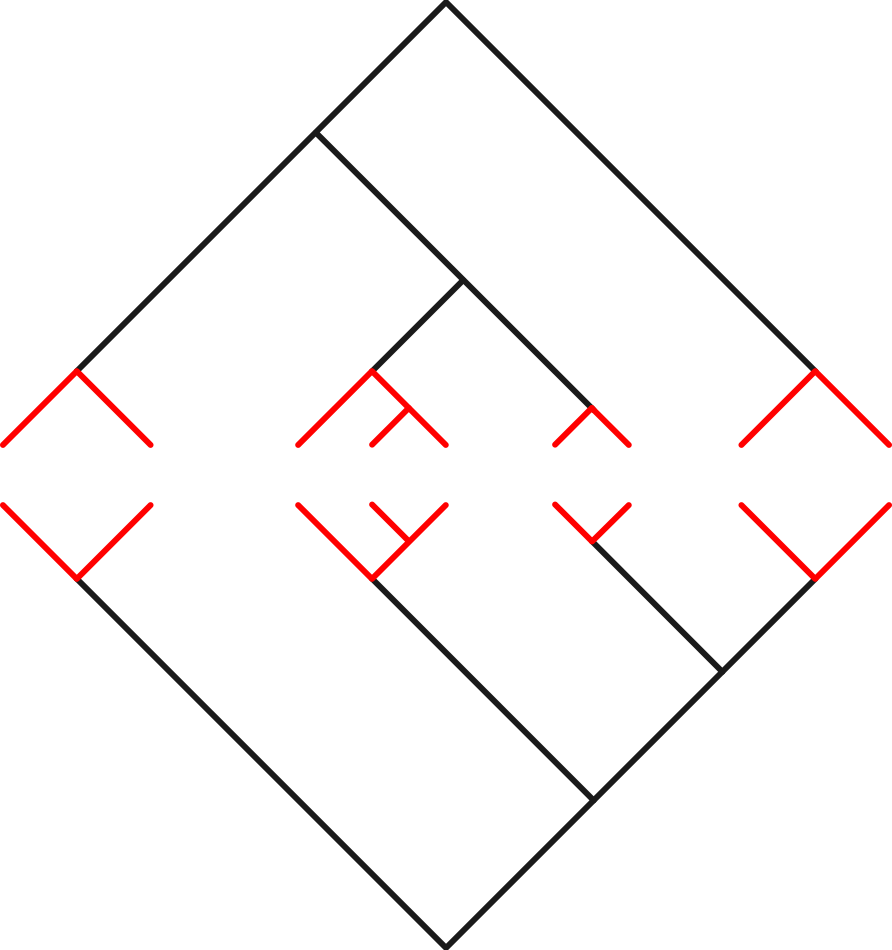}
		\put(-140, 30){$F_{\mathcal{J}'}$}
        \put(-140, 100){$F_{\mathcal{I}'}$}
		\caption{}
		\label{fig: carets example}
	\end{subfigure}
	\caption{(a) A tree diagram $(F_\mathcal{I}, F_{\mathcal{J}})$. As a convention, we don't draw extra edges attached to roots in a tree diagram. (b) The tree diagram $(F_{\mathcal{I}'}, F_{\mathcal{J}'})$ is obtained by adding some carets (marked in red) to $(F_\mathcal{I}, F_{\mathcal{J}})$.}
\end{figure}

\begin{lemma} \cite[Lemma 2.2]{MR1426438}
	Any $g\in F$ has a unique reduced pair of s.d. partitions $(\mathcal{I}, \mathcal{J})$ representing it. Furthermore, any pair $(\mathcal{I}', \mathcal{J}')$ representing $g$ is a refinement of $(\mathcal{I}, \mathcal{J})$.
\end{lemma}

If $(\mathcal{I}, \mathcal{J})$ is a pair of s.d. partitions representing $g$, we say $(F_{\mathcal{I}}, F_{\mathcal{J}})$ is a tree diagram of $g$. If $(\mathcal{I}, \mathcal{J})$ is the reduced pair of s.d. partitions representing $g$, we say $(F_{\mathcal{I}}, F_{\mathcal{J}})$ is the reduced tree diagram of $g$. In the latter case, the leaf number of $F_{\mathcal{I}}$ or $F_{\mathcal{J}}$ is called the leaf number of $g$.

\subsection{Some general theories of standard dyadic partition}
\label{sec: some general theories of sdp}

In later sections, we will construct half grid diagrams based on s.d. partitions. First we need to introduce some useful notations and properties related to s.d. partitions.

\begin{definition}
	A left (resp. right) s.d. interval is an s.d. interval of the form $[\frac{a}{2^n}, \frac{a+1}{2^n}]$ for some even (resp. odd) number $a$ and some integer $n \geq 1$. We denote the collection of left (resp. right) s.d. intervals as $\mathcal{S}_L$ (resp. $\mathcal{S}_R$).
\end{definition}

According to the above definition, $[0, 1]$ is neither a left s.d. interval nor a right s.d. interval. However, we will regard $[0, 1]$ as a right s.d. interval. With this convention, any s.d. interval is either a left s.d. interval or a right s.d. interval, but cannot be both simultaneously. As a consequence, $\mathcal{S}$ is the disjoint union of $\mathcal{S}_L$ and $\mathcal{S}_R$.

\begin{definition}
	Let $A=[\frac{a}{2^n}, \frac{a+1}{2^n}] \in \mathcal{S} \setminus \{[0, 1]\}$, we define the conjugate of $A$ to be
	$$
	\overline{A}=
	\begin{cases}
		[\frac{a+1}{2^n}, \frac{a+2}{2^n}] ,  & \text{if } A \in    \mathcal{S}_L \\
		[\frac{a-1}{2^n}, \frac{a}{2^n}], & \text{if } A \in \mathcal{S}_R
	\end{cases}
	$$
	
	$\overline{A}$ is called the conjugate of $A$.
\end{definition}

In the previous section we described a bijection $\epsilon: \mathcal{S} \to \mathcal{E}(F_{\infty})$. We can easily see that for any $A \in \mathcal{S} \setminus \{[0, 1]\}$, $A$ is a left (resp. right) interval if and only if $\epsilon(A)$ is a left (resp. right) edge. Furthermore, $\epsilon(A \cup \overline{A})$ splits into  $\epsilon(A)$ and $\epsilon (\overline{A})$ at a trivalent vertex.

\bigskip

Next definition will play the key role in the construction of half grid diagrams.

\begin{definition}
	Let $\mathcal{I}$ be an s.d. partition with breakpoints $0 = a_0 <  a_1 <... < a_n = 1$. We denote $\mathcal{S}(\mathcal{I})$ to be the collection of $[a_i, a_j]$'s such that $[a_i, a_j]$ is a s.d. interval. Denote $\mathcal{S}_L(\mathcal{I}) = \mathcal{S}(\mathcal{I}) \cap \mathcal{S}_L$ and $\mathcal{S}_R(\mathcal{I}) = \mathcal{S}(\mathcal{I}) \cap \mathcal{S}_R$.
\end{definition}

Recall that $\mathcal{I}$ can be represented by a binary tree $F_{\mathcal{I}}$ as a subtree of $F_{\infty}$. If we restrict the bijection $\epsilon: \mathcal{S} \to \mathcal{E}(F_{\infty})$ on $\mathcal{S}(\mathcal{I})$, we have a bijection $\epsilon_{\mathcal{I}}: \mathcal{S}(\mathcal{I}) \to \mathcal{E}(F_{\mathcal{I}})$, where $\mathcal{E}(F_{\mathcal{I}})$ is the collection of edges of $F_{\mathcal{I}}$. This bijection sends subintervals of $\mathcal{I}$ to the lowermost edges of $F_{\mathcal{I}}$, and other s.d. intervals to non-lowermost edges.

Observe that each non-uppermost edge $\epsilon(A)$ must be a split of $\epsilon(A \cup \overline{A})$, and each non-lowermost edge $\epsilon(B)$ can split into $\epsilon(A)$ and $\epsilon(\overline{A})$ for some $A \in \mathcal{S}(\mathcal{I})$, then we have the following lemma.

\begin{lemma} \label{lem: conjugate in SDI}
	Suppose $A \in \mathcal{S}(\mathcal{I}) \setminus \{[0,1]\}$, then $A \cup \overline{A} \in \mathcal{S}(\mathcal{I})$. Furthermore, if $B \in \mathcal{S}(\mathcal{I})$ is not a subinterval of $\mathcal{I}$, then $B = A \cup \overline{A}$ for some $A \in \mathcal{S}(\mathcal{I})$.
\end{lemma}

By counting the number of edges of $F_{\mathcal{I}}$, we get the following lemma.

\begin{lemma} {\label{lem: cardinality of SDI}}
	Let $\mathcal{I}$ be an s.d. $n$-partition, then
	\begin{enumerate}
		\item $|\mathcal{S}(\mathcal{I})| = 2n - 1$
		\item $|\mathcal{S}_L(\mathcal{I})| = |\mathcal{S}_R(\mathcal{I})|-1 = n - 1$
	\end{enumerate}
\end{lemma}

Next, we turn to dyadic rational numbers and their connection to s.d. intervals and the Thompson group.

\begin{definition}
	We denote $E = \{\frac{k}{2^m}: m \geq 0 \text{ and } 1 \leq k \leq 2^m - 1 \}$, the set of dyadic rational numbers except $0$ and $1$. Given an s.d. partition $\mathcal{I}$, denote $M(\mathcal{I})$ to be the set of midpoints of subintervals of $\mathcal{I}$. Denote $E(\mathcal{I})$ to be the union of $M(\mathcal{I})$ and all of breakpoints of $\mathcal{I}$ except $0$ and $1$.
\end{definition}

Note that $M(\mathcal{I})$ and $E(\mathcal{I})$ are both subsets of $E$.

\begin{lemma} \label{lem: taking midpoint is a bijection from S to E}
	The map $m: \mathcal{S} \to E$, sending s.d. intervals to their midpoints, is a bijection. Furthermore, $m|_{\mathcal{S}(\mathcal{I})}: \mathcal{S}(\mathcal{I}) \to E(\mathcal{I})$ is a bijection for any s.d. partition $\mathcal{I}$.
\end{lemma}

\begin{proof}
	Note that $m$ has inverse $m^{-1}: E \to \mathcal{S}$ sending any dyadic number $\frac{a}{2^n}$ with some $a$ not divided by 2, to an s.d. interval  $[\frac{a-1}{2^n}, \frac{a+1}{2^n}] = [\frac{\frac{a-1}{2}}{2^{n-1}}, \frac{\frac{a-1}{2}+1}{2^{n-1}}]$. So $m$ is a bijection.

	To show $m|_{\mathcal{S}(\mathcal{I})}: \mathcal{S}(\mathcal{I}) \to E(\mathcal{I})$ is a bijection, we just need to show $m$ sends $\mathcal{S}(\mathcal{I})$ onto $E(\mathcal{I})$, then the statement follows since $|\mathcal{S}(\mathcal{I})| = |E(\mathcal{I})| = 2n-1$ if $\mathcal{I}$ is an $n$-partition. Indeed, for any $B \in \mathcal{S}(\mathcal{I})$, if $B$ is a subinterval of $\mathcal{I}$, then $m(B) \in M(\mathcal{I}) \subseteq E(\mathcal{I})$. If $B$ is not a subinterval of $\mathcal{I}$, then $B = A \cup \overline{A}$ for some $A \in \mathcal{S}(\mathcal{I})$ by the last part of Lemma \ref{lem: conjugate in SDI}, so $m(B)$ is a breakpoint of $\mathcal{I}$.
\end{proof}

In the next lemma, we will consider elements in the Thompson group $F$ as homeomorphisms from $[0, 1]$ to itself.

\begin{lemma} \label{lem: g restricted on E is bijection}
	For any $g\in F$, we have that $g|_E: E \to E$ is a bijection. Furthermore, if $g$ is represented by a pair of s.d. partitions $(\mathcal{I}, \mathcal{J})$, then $g|_{E(\mathcal{I})}: E(\mathcal{I}) \to E(\mathcal{J})$ is a bijection.
\end{lemma}

\begin{proof}
	First of all, we have $g(E) \subseteq E$. To show it, we take the reduced s.d. partition pair $(\mathcal{I}, \mathcal{J})$ representing $g$. For any $a \in E$, we can always choose refinement $(\mathcal{I}', \mathcal{J}') \geq (\mathcal{I}, \mathcal{J})$ such that $a$ is a breakpoint of $\mathcal{I}'$. Notice that $(\mathcal{I}', \mathcal{J}')$ also represents $g$, so $g$ sends $a$ to a breakpoint of $\mathcal{J}'$, which is in $E$.
	
	Notice that $g^{-1}|_E : E \to E$ is the inverse of $g|_E$, so $g|_E$ is a bijection.

	To show $g|_{E(\mathcal{I})}: E(\mathcal{I}) \to E(\mathcal{J})$ is a bijection, we just need to show $g$ sends $E(\mathcal{I})$ into $E(\mathcal{J})$, then the statement follows since $|E(\mathcal{I})| = |E(\mathcal{J})|$. Indeed, by Definition \ref{def: s.d. partition pair representation}, we know that $g$ sends breakpoints of $\mathcal{I}$ to breakpoints of $\mathcal{J}$, sends $M(\mathcal{I})$ to $M(\mathcal{J})$ due to linearality of $g$ on each subinterval.
\end{proof}

\section{Thompson link}
\label{sec: thompson link}

\subsection{Jones' construction of links from Thompson group}
\label{sec: jones construction}

Let $C = (C_n)$ be the planar algebra of Conway tangles, where $C_n$ has basis Conway tangles with $2n$ boundary points. For Conway tangles, we identify two tangles if they differ by a family of distant unlinked unknots. For example, $C_0$ is spanned by links up to distant unlinked unknots, $C_1$ is spanned by tangles with two boundary points, up to distant unlinked unknots. We denote the class of trivial tangle in $C_1$ as $\xi$.

Moreover, $C_n$ has a $C_0$-valued inner product structure $\langle x, y \rangle$, given by connecting $2n$ boundary points of $x$ to $2n$ boundary points of $y$ one by one (see Figure \ref{fig: inner product example}).

\begin{figure}
	\centering
	\includegraphics[width = 0.4\textwidth]{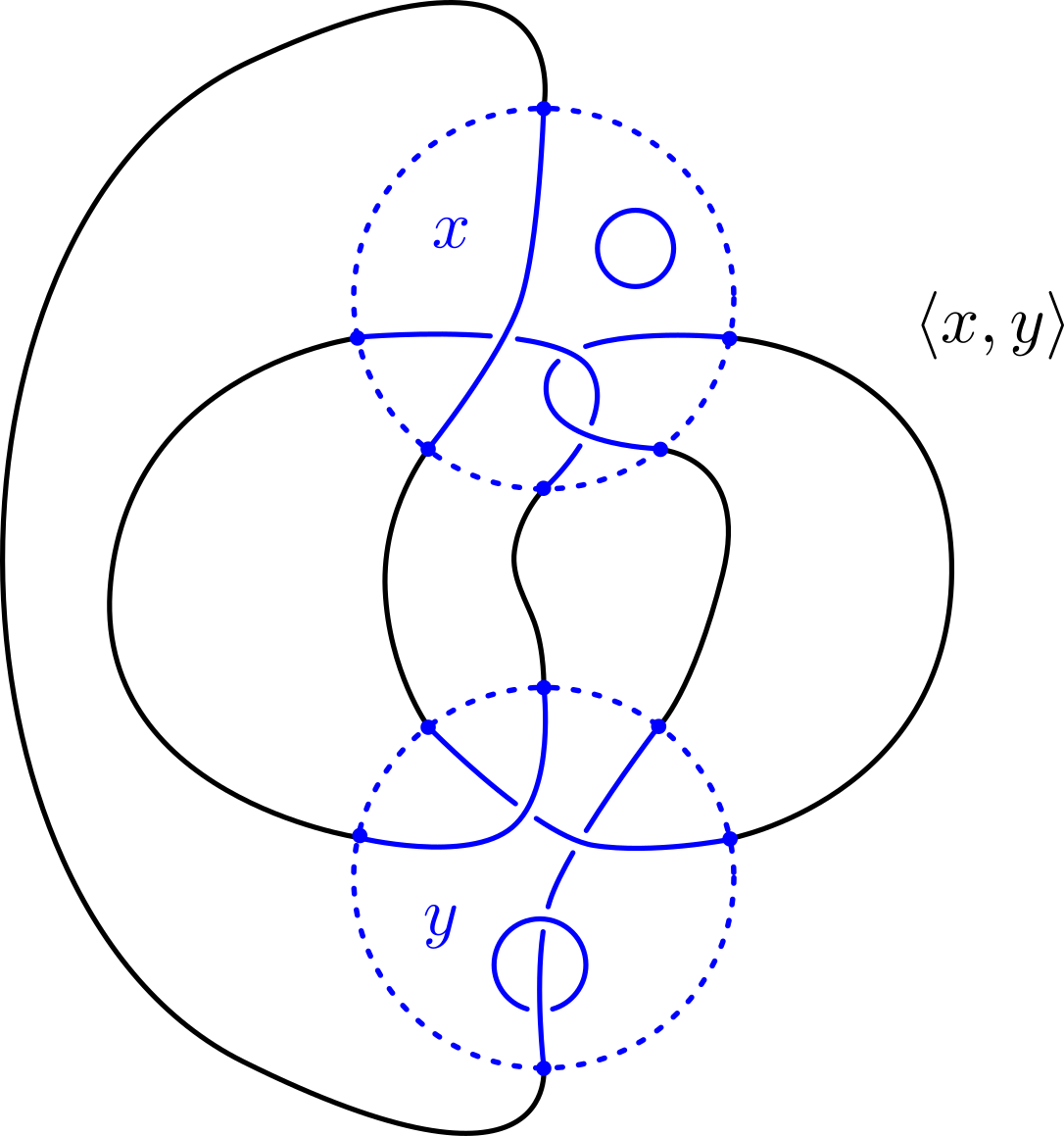}
	\caption{$x, y \in C_3$ are tangles with $6$ boundary points. Their inner product $\langle x, y \rangle \in C_0$ is obtained by connecting boundary points of $x$ to boundary points of $y$ one by one.}
	\label{fig: inner product example}
\end{figure}

\begin{definition}
	Let $\mathfrak{T}$ be the category of tangles, with objects s.d. partitions, and morphisms $Hom(\mathcal{I}, \mathcal{J})$ for $\mathcal{I} \leq \mathcal{J}$ being tangles with marked boundary points $E(\mathcal{I})$ on the top and $E({\mathcal{J}})$ on the bottom. For a tangle $T\in Hom(\mathcal{I}, \mathcal{J})$, we denote $\overline{T}\in Hom(\mathcal{J}, \mathcal{I})$ to be the vertical mirror of $T$.
\end{definition}

Suppose that we have a pair of s.d. partitions $\mathcal{K} \leq \mathcal{I}$ in $\mathfrak{T}$. Recall that we have a forest $F^{\mathcal{K}}_{\mathcal{I}}$ such that each lowermost edge is labeled by subintervals of $\mathcal{I}$ and each uppermost edge is labeled by subintervals of $\mathcal{K}$. Now we embed $F^{\mathcal{K}}_{\mathcal{I}}$ into $[0, 1] \times [0, 1]$ such that each lowermost edge $\epsilon(A)$ has lower endpoint $(m(A), 0)$, each uppermost edge $\epsilon(A')$ has upper endpoint $(m(A'), 1)$. By this embedding, we can regard $F^{\mathcal{K}}_{\mathcal{I}}$ as a forest whose roots are $M(\mathcal{K}) \times \{1\}$ and whose leaves are $M(\mathcal{I}) \times \{0\}$.

For each edge $e$ of $F^{\mathcal{K}}_{\mathcal{I}}$, we denote the vertical projection of $e$'s lower vertex as $\phi^{\mathcal{K}}_{\mathcal{I}}(e)$. Now we further isotope $F^{\mathcal{K}}_{\mathcal{I}}$ in $[0, 1] \times [0, 1]$ such that $\phi^{\mathcal{K}}_{\mathcal{I}}(\epsilon(B)) = m(B)$ for any edge $\epsilon(B)$, then we can regard $\phi^{\mathcal{K}}_{\mathcal{I}}$ as a map from $\mathcal{E}(F^{\mathcal{K}}_{\mathcal{I}})$ to $E(\mathcal{I})$. Specifically, $\phi^{\delta}_{\mathcal{I}}: \mathcal{E}(F^{\delta}_{\mathcal{I}}) \to E(\mathcal{I})$ is a bijection, because $\phi^{\delta}_{\mathcal{I}} \circ \epsilon_{\mathcal{I}} = m|_{\mathcal{S}(\mathcal{I})}$ and both of $\epsilon_{I}$, $m|_{\mathcal{S}(\mathcal{I})}$ are bijections. See Figure \ref{fig: embedded forest example} and \ref{fig: embedded tree example} for examples of an embedded forest and an embedded tree.

\begin{figure}
    \centering
    \begin{subfigure}{0.4\textwidth}
        \centering
        \includegraphics[width=\textwidth]{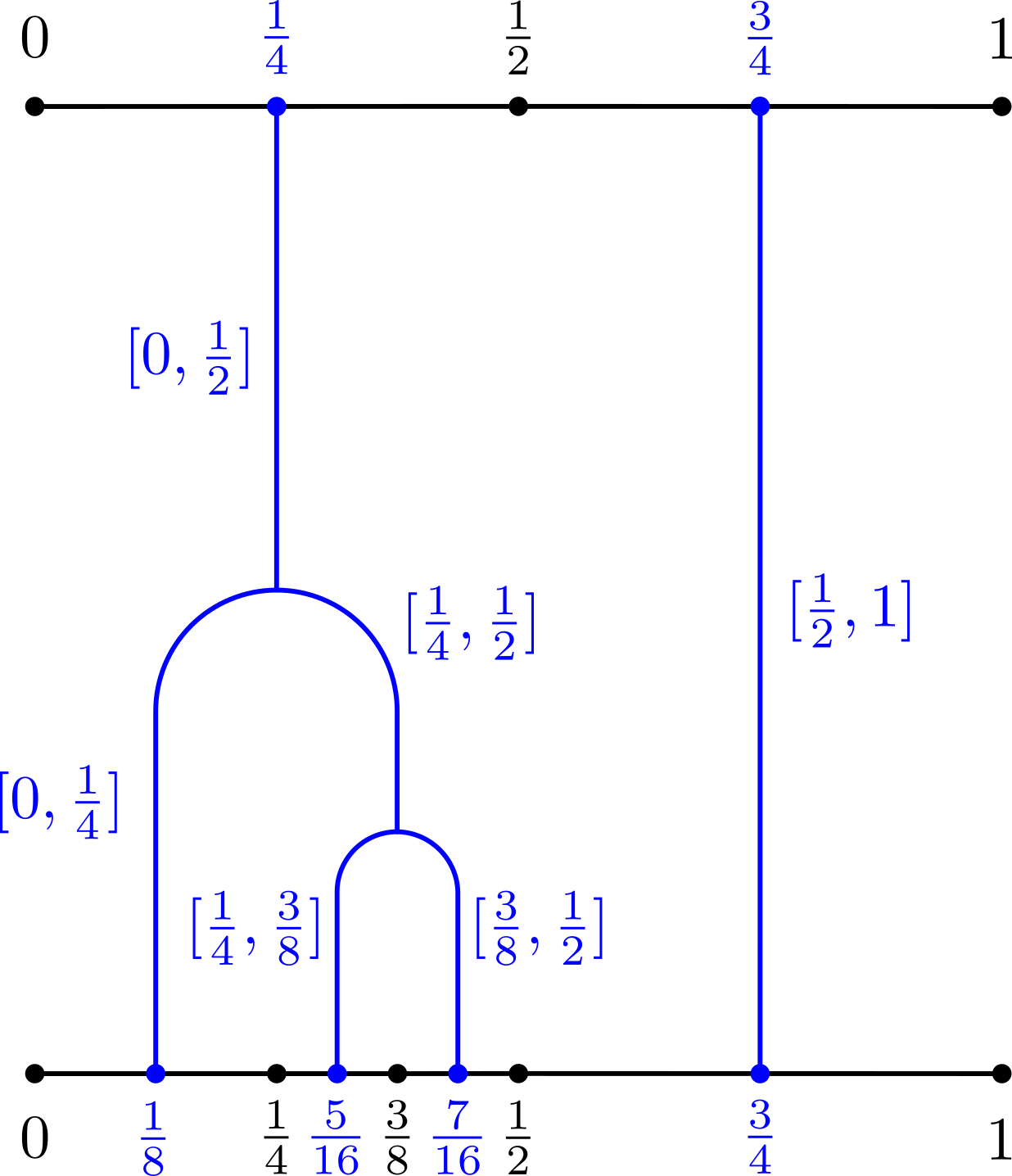}
        \put(-190, 15){$\mathcal{I}$}
        \put(-190, 180){$\mathcal{K}$}
        \put(-190, 120){$\color{blue} F^{\mathcal{K}}_{\mathcal{I}}$}
        \caption{}
        \label{fig: embedded forest example}
    \end{subfigure}
    \hfill
    \begin{subfigure}{0.4\textwidth}
        \centering
        \includegraphics[width=\textwidth]{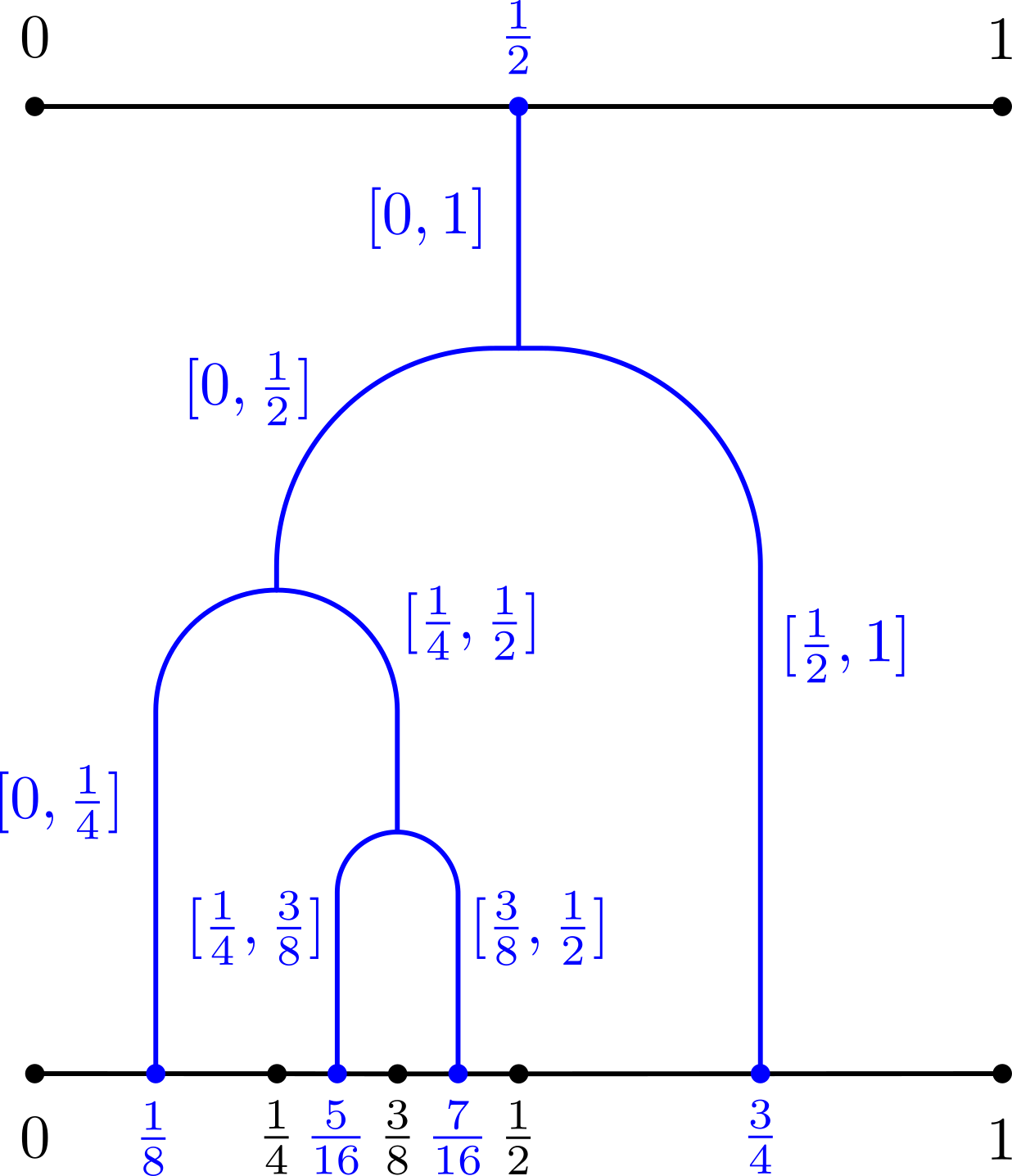}
        \put(-190, 15){$\mathcal{I}$}
        \put(-190, 180){$\delta$}
        \put(-190, 120){$\color{blue} F^{\delta}_{\mathcal{I}}$}
        \caption{}
		\label{fig: embedded tree example}
    \end{subfigure}
    \caption{(a) A forest $F^{\mathcal{K}}_{\mathcal{I}}$ embedded in $[0, 1] \times [0, 1]$ for a pair of s.d. partitions $\mathcal{K} \leq \mathcal{I}$. Breakpoints of $\mathcal{K}$ and $\mathcal{I}$ are marked in black, points in $M(\mathcal{K})$ and $M(\mathcal{I})$ are marked in blue. (b) A tree $F^{\delta}_{\mathcal{I}}$ embedded in $[0, 1] \times [0, 1]$ for an s.d. partition $\mathcal{I}$.}
\end{figure}

\begin{figure}
    \centering
    \includegraphics[width = \textwidth]{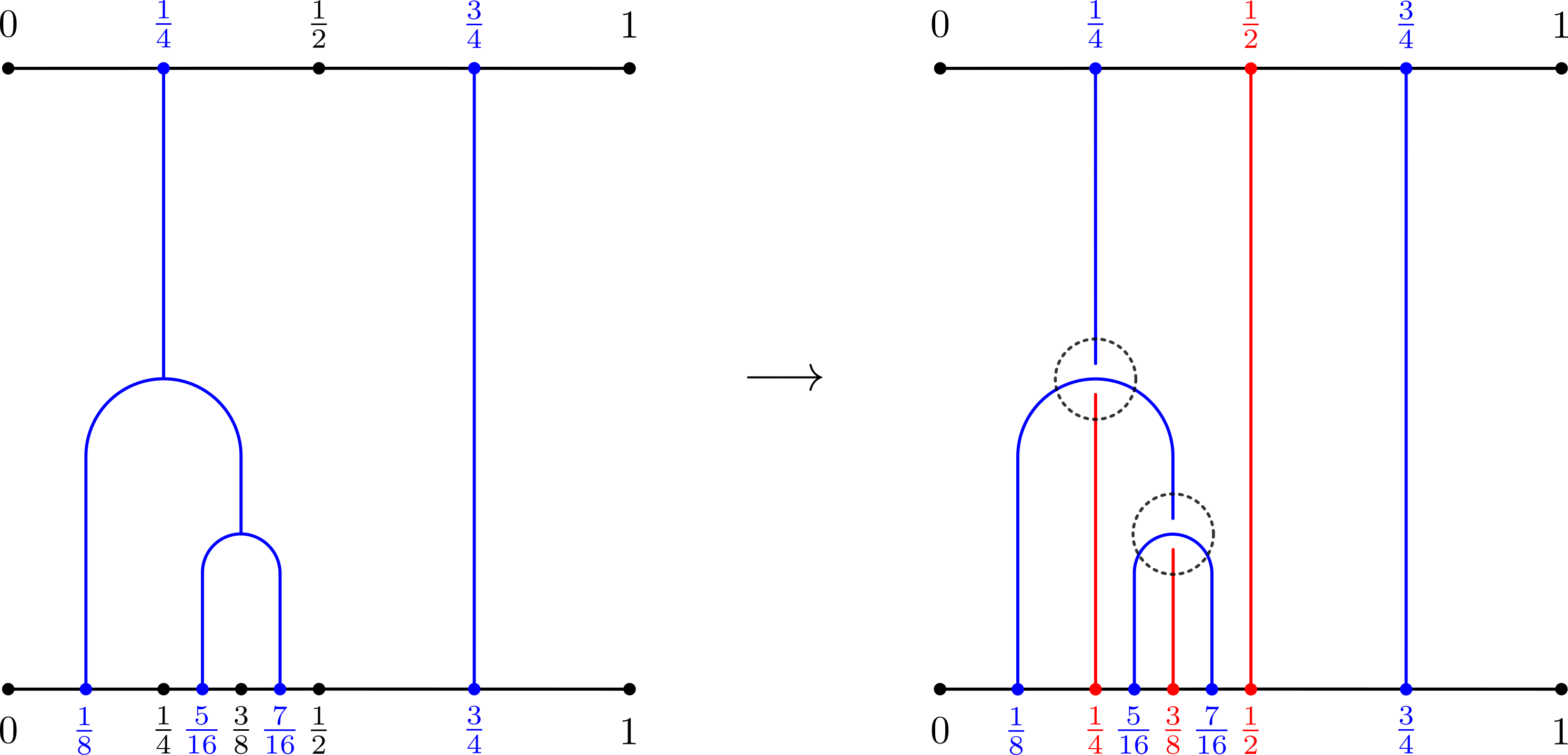}
	\put(-450, 15){$\mathcal{I}$}
	\put(-450, 185){$\mathcal{K}$}
	\put(-450, 120){$\color{blue} F^{\mathcal{K}}_{\mathcal{I}}$}
	\put(-195, 15){$\mathcal{I}$}
	\put(-195, 185){$\mathcal{K}$}
	\put(0, 120){$T^{\mathcal{K}}_{\mathcal{I}}$}
	\put(-120, 115){$R$}
	\put(-100, 75){$R$}
    \caption{Construction of tangle $T^{\mathcal{K}}_{\mathcal{I}}$ for some $\mathcal{K} \leq \mathcal{I}$. In $T^{\mathcal{K}}_{\mathcal{I}}$, Type A arcs are marked in blue, Type B arcs are marked in red.}
    \label{fig: forest to unoriented tangle}
\end{figure}
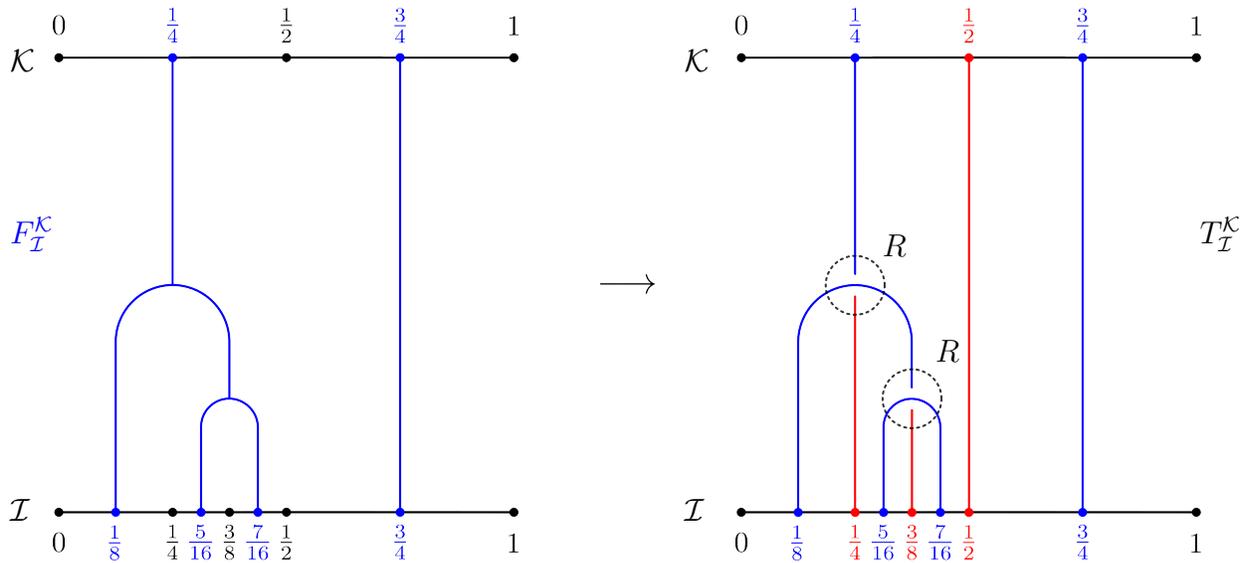

\begin{definition}
	Given a pair of objects $\mathcal{K} \leq \mathcal{I}$ in $\mathfrak{T}$, a labeled tangle $T^{\mathcal{K}}_{\mathcal{I}}\in Hom(\mathcal{K}, \mathcal{I})$ is defined by replacing each trivalent vertex of the forest $F^{\mathcal{K}}_{\mathcal{I}}$ by an instance of $R\in C_2$, with a vertical straight line joining the bottom of the disc containing $R$ to the breakpoint of $\mathcal{I}$ below the replaced vertex, and vertical lines connecting all the breakpoints of $\mathcal{K}$ to those breakpoints of $\mathcal{I}$ in common with $\mathcal{I}'$.

	In the rest of this section, we choose $R$ to be the crossing with a horizontal strand crossing over a vertical strand, as shown in Figure \ref{fig: forest to unoriented tangle}.
\end{definition}

In the construction of tangle $T^{\mathcal{K}}_{\mathcal{I}}$, every edge $e$ of $F^{\mathcal{K}}_{\mathcal{I}}$ becomes an arc $a(e)$ of $T^{\mathcal{K}}_\mathcal{I}$, called a Type A arc. If $e$ is not an lowermost edge, then we need to add a vertical arc $b(e)$ extending $a(e)$ to breakpoint $\phi^{\mathcal{K}}_{\mathcal{I}}(e)$. Such an arc $b(e)$ is called a Type B arc. At last, we need to add some arcs to connect breakpoints, and these arcs are also called Type B arcs.

In \cite{MR3589908} Jones considered a representation of the category $\mathfrak{T}$, associating each s.d. interval $\mathcal{I}\in \mathfrak{T}$ with $n$ subintervals a vector space $H_{\mathcal{I}} = C_{n}$. Then take the direct limit over the ordered set $(\mathfrak{T}, \leq)$, with linear map from $H_{\mathcal{I}'}$ to $H_{\mathcal{I}}$ induced by the canonical morphism $T_{\mathcal{I}'}^{\mathcal{I}}$. The resulting space $\mathcal{V}=\lim_{\mathcal{I} \in \mathfrak{T}} H_{\mathcal{I}}$ has an inner product induced by the inner product on $C_n$.

\begin{definition}
	Suppose that $g\in F$ is represented by a pair of s.d. partitions $(\mathcal{I}, \mathcal{J})$, not necessarily reduced. Then define $g^{\mathcal{I}}_{\mathcal{J}}\in Hom(\mathcal{I}, \mathcal{J})$ to be the tangle with straight lines connecting points of $E(\mathcal{I})$ to points $E(\mathcal{J})$ one by one.
\end{definition}

\begin{figure}
	\centering
	\includegraphics[width = 0.4\textwidth]{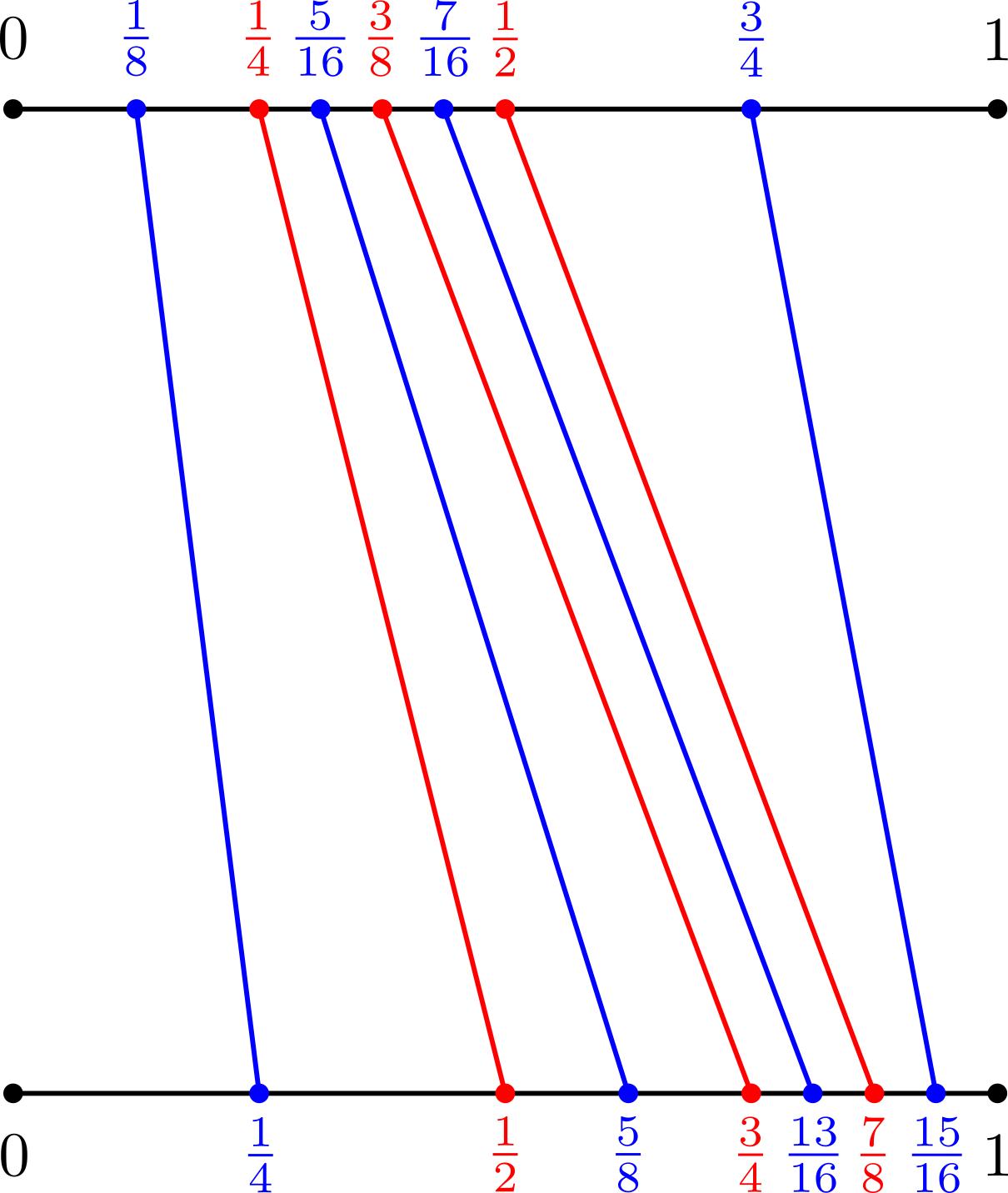}
	\put(-190, 15){$\mathcal{J}$}
    \put(-190, 180){$\mathcal{I}$}
	\put(-190, 120){$g^{\mathcal{I}}_{\mathcal{J}}$}
	\caption{An example of $g^{\mathcal{I}}_{\mathcal{J}}\in Hom(\mathcal{I}, \mathcal{J})$. Breakpoints of $\mathcal{I}$ and $\mathcal{J}$ (not including $0$ and $1$) are marked in red, points in $M(\mathcal{I})$ and $M(\mathcal{J})$ are marked in blue.}
	\label{fig: shift example}
\end{figure}

\begin{remark}
	In Lemma \ref{lem: g restricted on E is bijection}, we showed that $g|_{E(\mathcal{I})}: E(\mathcal{I}) \to E(\mathcal{J})$ is a bijection. Here $g^{\mathcal{I}}_{\mathcal{J}}$ visualizes this bijection. $g^\mathcal{I}_\mathcal{J}$ connects any $x \in E(\mathcal{I})$ to $g(x) \in E(\mathcal{J})$.
\end{remark}

Jones also defined a representation $\pi$ of $F$ on $\mathcal{V}$. Suppose that $g\in F$ corresponds to a reduced pair of s.d. partitions $(\mathcal{I}, \mathcal{J})$, given a vector $v\in \mathcal{V}$, we first find a representative $v\in H_{\mathcal{I'}}$ such that $\mathcal{I} \leq \mathcal{I}'$. Correspondingly, $g$ has an expression $(\mathcal{I}', \mathcal{J}')$ for some $\mathcal{J}'\in \mathfrak{T}$. Then the action of $\pi(g)$ on $v$ is defined to be induced by the trivial tangle $g^{\mathcal{I}'}_{\mathcal{J}'}\in Hom(\mathcal{I'}, \mathcal{J'})$.

As a result, a link class (up to distant unlinked unknots) $\langle \pi(g)(\xi), \xi \rangle \in C_0$ arises as a coefficient of $\pi(g)(\xi)$.

\begin{proposition} \cite[Theorem 5.3.1]{MR3589908} \label{prop: unoriented Thompson Alexander theorem ver 1}
	For any $[L] \in C_0$, there exists $g\in F$ such that $\langle \pi(g)(\xi), \xi \rangle = [L]$.
\end{proposition}
    
Next, we will see an explicit construction of a link in the class $\langle \pi(g)(\xi), \xi \rangle$. Notice that any representative of $\xi$ in $\mathfrak{T}$ can be written as $T^\delta_{\mathcal{I}} \in Hom(1, \mathcal{I})$ for some $\mathcal{I}\in \mathfrak{T}$, so it is not surprising that we can rephrase the above construction using only trees. For convenience, we will denote $F^{\delta}_{\mathcal{I}}$ as $F_{\mathcal{I}}$, $\phi^\delta_{\mathcal{I}}$ as $\phi_{\mathcal{I}}$ and $T^\delta_{\mathcal{I}}$ as $T_{\mathcal{I}}$.

Given $g\in F$, suppose that $(\mathcal{I}, \mathcal{J})$ is a pair of s.d. partitions representing $g$. Then we have a pair of binary trees $F_{\mathcal{I}}$ and $F_{\mathcal{J}}$. Furthermore, we have $\overline{T_\mathcal{J}} g^{\mathcal{I}}_{\mathcal{J}} T_\mathcal{I} \in Hom(\delta, \delta)$, as shown in Figure \ref{fig: construction of link with shift example}. Then we close it up to get a link $L_{(\mathcal{I}, \mathcal{J})}$  in the class $\langle \pi(g)(\xi), \xi \rangle$, where $T_{\mathcal{J}}$ is a representative of $\xi$, and $g^{\mathcal{I}}_{\mathcal{J}} T_{\mathcal{I}}$ is a representative of $\pi(g)(\xi)$.

We let $L_g$ be $L_{(\mathcal{I}, \mathcal{J})}$, where $(\mathcal{I}, \mathcal{J})$ is unique reduced pair of s.d. partitions representing $g$. Notice that for a non-reduced pair $(\mathcal{I}', \mathcal{J}')$ representing $g$, $L_{(\mathcal{I}', \mathcal{J}')}$ contains some distant unknots compared to $L_g$ as shown in Figure \ref{fig: non-reduced pair with shift example}, but $L_{(\mathcal{I}', \mathcal{J}')}$ and $L_g$ are in the same link class in $C_0$.

At last, if we consider $T_{\mathcal{I}}$ and $T_{\mathcal{J}}$ as $(1, 2n-1)$-tangles with unmarked endpoints, and $g^{\mathcal{I}}_{\mathcal{J}}$ as a trivial $(2n-1, 2n-1)$-tangle with unmarked endpoints, then $L_{(\mathcal{I}, \mathcal{J})}$ is simply the closing-up of $\overline{T_{\mathcal{J}}}T_{\mathcal{I}}$. Equivalently, we add a trivial strand on the leftmost of $T_{\mathcal{I}}$ and $T_{\mathcal{J}}$ respectively, then cap them off on the top to get $(0,2n)$-tangles $\widehat{T_{\mathcal{I}}}$ and $\widehat{T_{\mathcal{J}}}$. Then $L_{(\mathcal{I}, \mathcal{J})}  = \overline{\widehat{{T_{\mathcal{J}}}}}\widehat{T_{\mathcal{I}}}$, as shown in Figures \ref{fig: construction of link with no shift example} and \ref{fig: non-reduced pair no shift example}. Notice that this notion is irrelevant to the Thompson group, $L_{(\mathcal{I}', \mathcal{J}')}$ is completely determined by a pair of s.d. partitions $(\mathcal{I}', \mathcal{J}')$ with the same number of breakpoints.

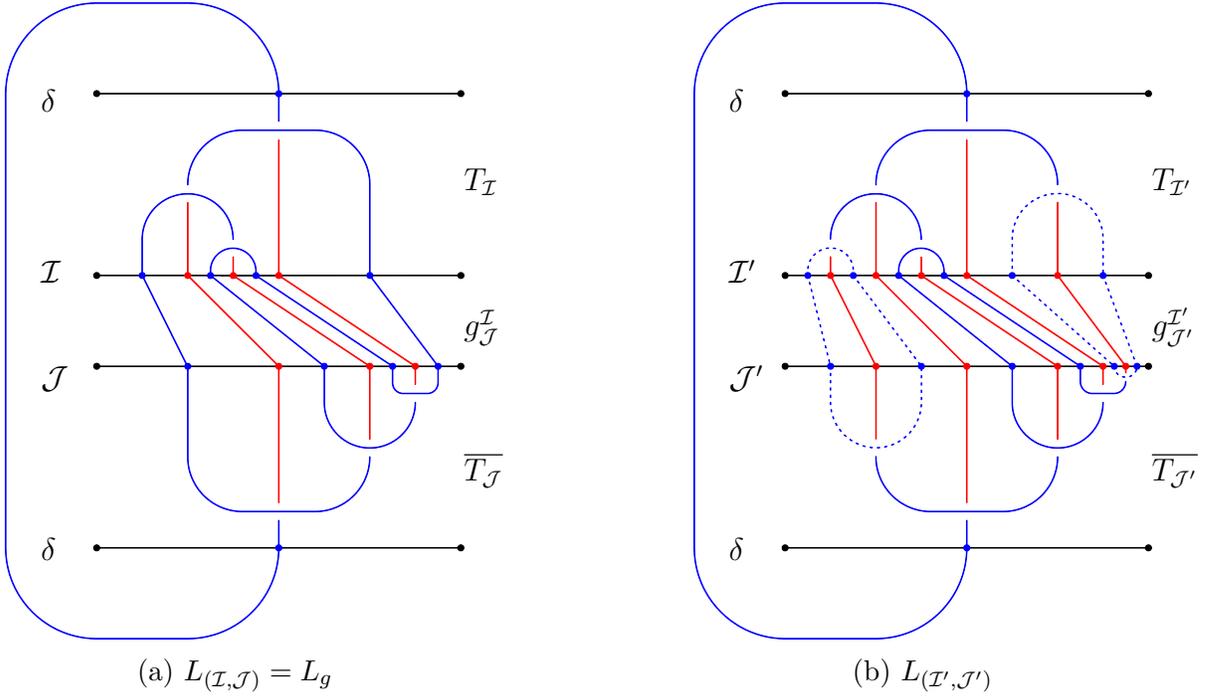
\begin{figure}
    \centering
    \begin{subfigure}{0.4\textwidth}
       \centering
       \includegraphics[width = \textwidth]{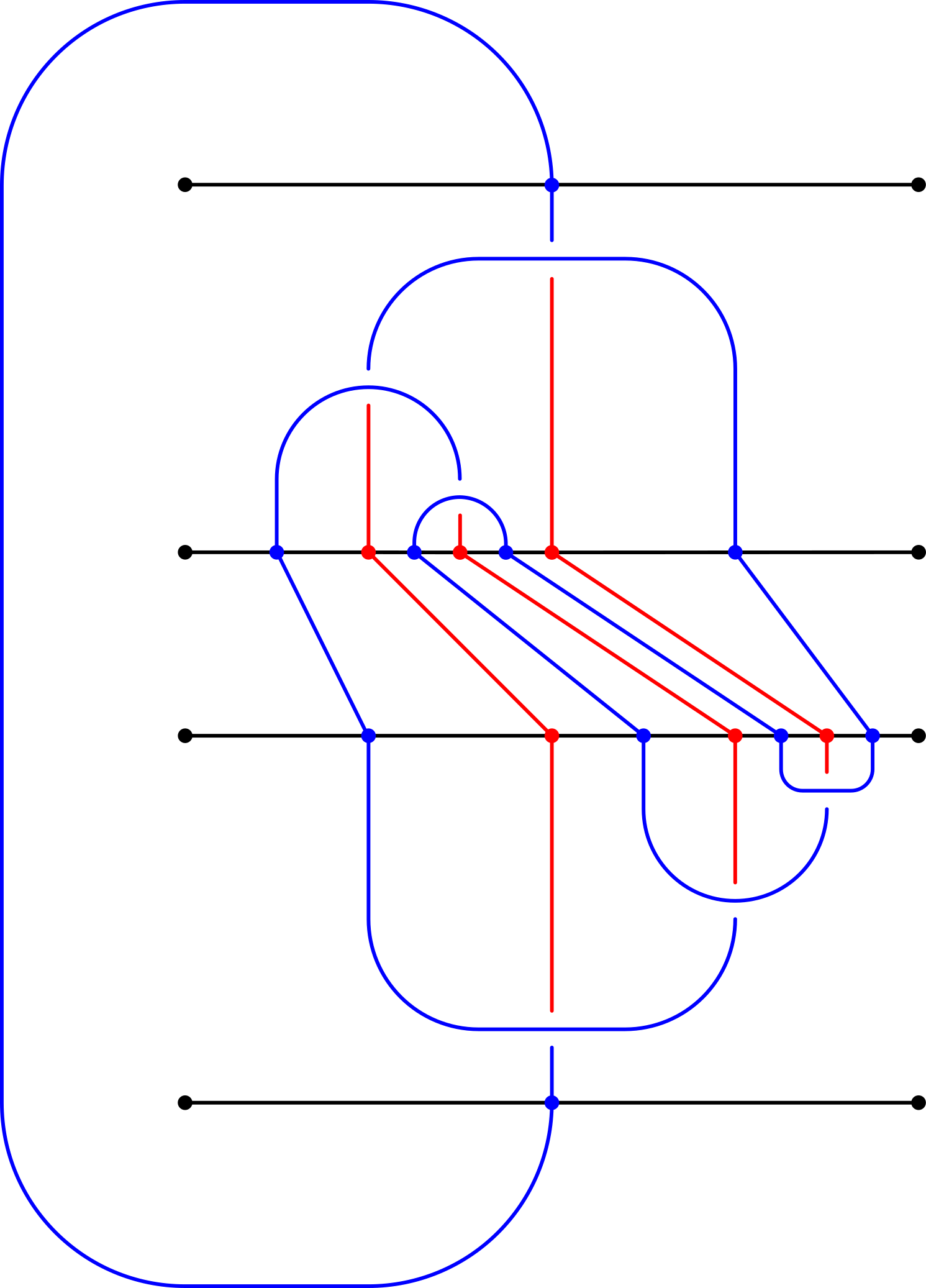}
       \put(-160, 30){$\delta$}
       \put(-160, 95){$\mathcal{J}$}
       \put(-160, 135){$\mathcal{I}$}
       \put(-160, 200){$\delta$}
       \put(0, 60){$\overline{T_\mathcal{J}}$}
       \put(0, 115){$g^{\mathcal{I}}_{\mathcal{J}}$}
       \put(0, 170){$T_\mathcal{I}$}
       \caption{$L_{(\mathcal{I}, \mathcal{J})} = L_g$}
       \label{fig: construction of link with shift example}
   \end{subfigure}
   \hfill
   \begin{subfigure}{0.4\textwidth}
       \centering
       \includegraphics[width = \textwidth]{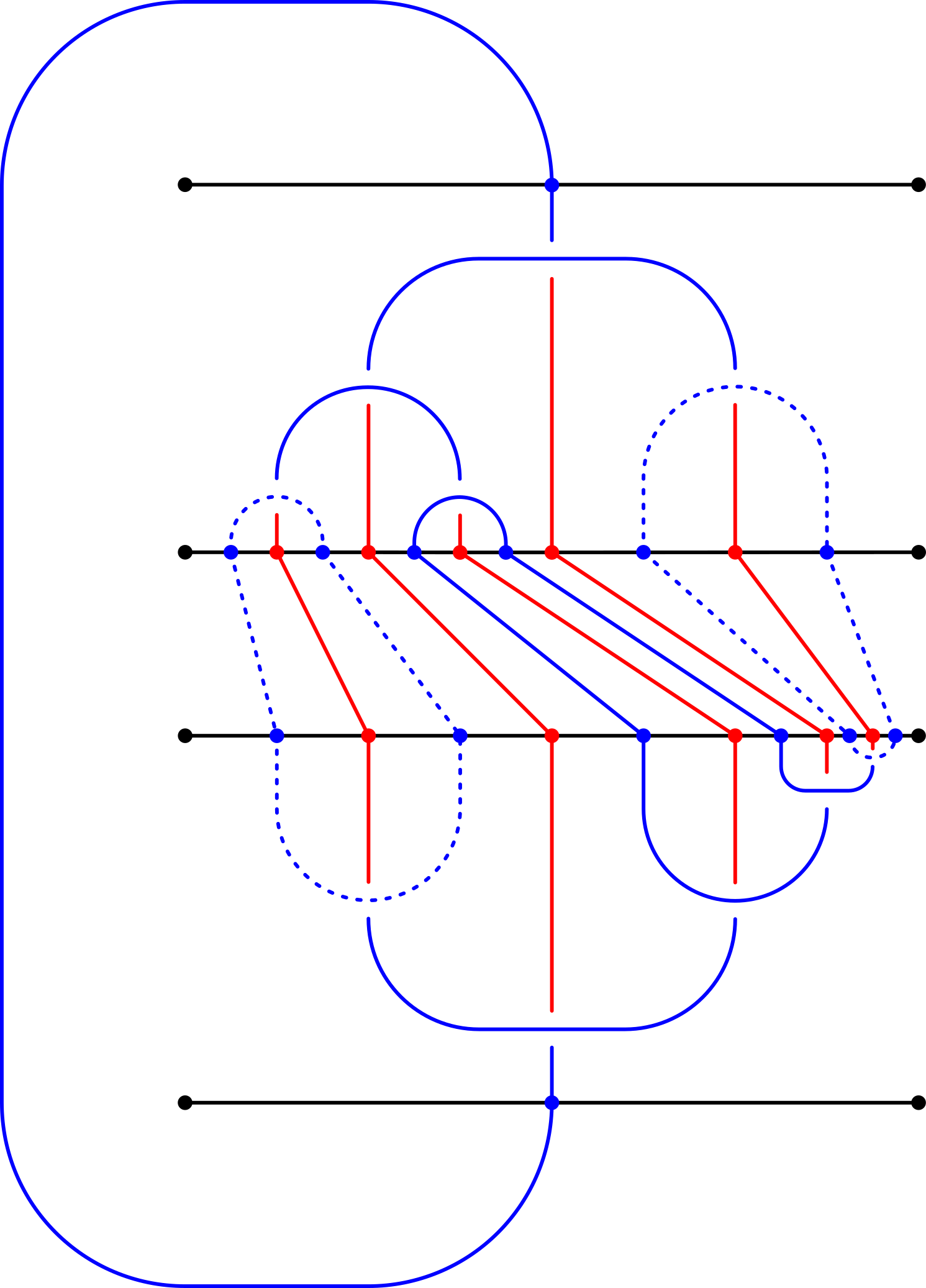}
       \put(-160, 30){$\delta$}
       \put(-160, 95){$\mathcal{J}'$}
       \put(-160, 135){$\mathcal{I}'$}
       \put(-160, 200){$\delta$}
       \put(0, 60){$\overline{T_{\mathcal{J}'}}$}
       \put(0, 115){$g^{\mathcal{I}'}_{\mathcal{J}'}$}
       \put(0, 170){$T_{\mathcal{I}'}$}
       \caption{$L_{(\mathcal{I}', \mathcal{J}')}$}
       \label{fig: non-reduced pair with shift example}
   \end{subfigure}
   \caption{(a) $\overline{T_\mathcal{J}} g^{\mathcal{I}}_{\mathcal{J}} T_\mathcal{I} \in Hom(\delta, \delta)$ for a pair of s.d. partitions $(\mathcal{I}, \mathcal{J})$ representing $g$. Then close up $\overline{T_\mathcal{J}} g^{\mathcal{I}}_{\mathcal{J}} T_\mathcal{I}$ to get link $L_{(\mathcal{I}, \mathcal{J})}$. In this figure, $(\mathcal{I}, \mathcal{J})$ is a reduced pair, so $L_{(\mathcal{I}, \mathcal{J})} = L_g$. (b) $\overline{T_{\mathcal{J}'}} g^{\mathcal{I}'}_{\mathcal{J}'} T_{\mathcal{I}'} \in Hom(\delta, \delta)$ for some non-reduced pair of s.d. partitions $(\mathcal{I}', \mathcal{J}')$ representing $g$. Then close up $\overline{T_{\mathcal{J}'}} g^{\mathcal{I}'}_{\mathcal{J}'} T_{\mathcal{I}'}$ to get link $L_{(\mathcal{I'}, \mathcal{J'})}$. Compared to $L_g$, the extra distant unknots in $L_{(\mathcal{I'}, \mathcal{J'})}$ are marked in dashed arcs.}
\end{figure}

\begin{figure}
	\centering
    \begin{subfigure}{0.4\textwidth}
        \centering
        \includegraphics[width = \textwidth]{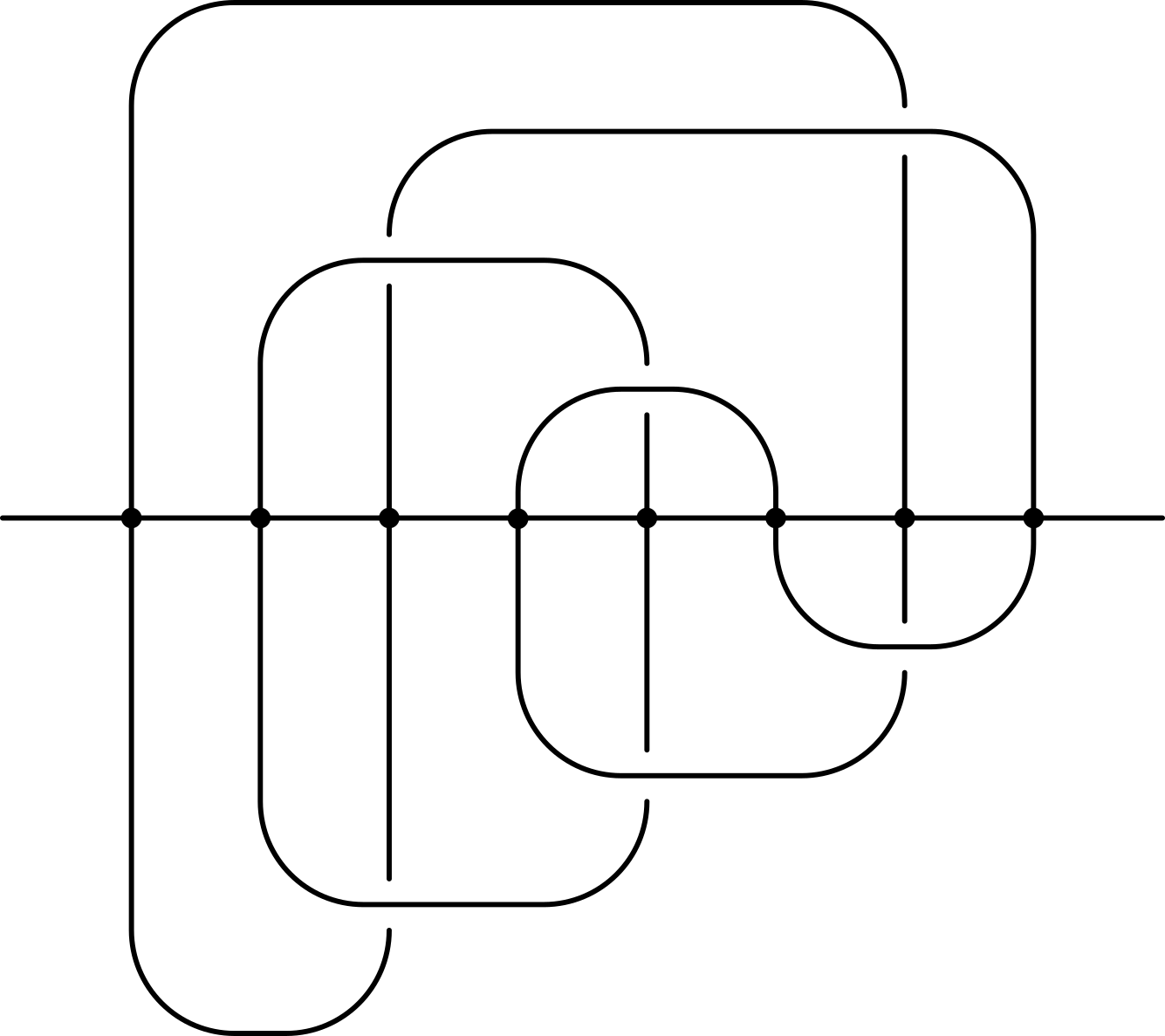}
        \put(0, 30){$\overline{\widehat{{T_{\mathcal{J}}}}}$}
        \put(0, 110){$\widehat{T_{\mathcal{I}}}$}
        \caption{$L_{(\mathcal{I}, \mathcal{J})} = L_g$}
		\label{fig: construction of link with no shift example}
    \end{subfigure}
    \hfill
    \begin{subfigure}{0.4\textwidth}
        \centering
        \includegraphics[width = \textwidth]{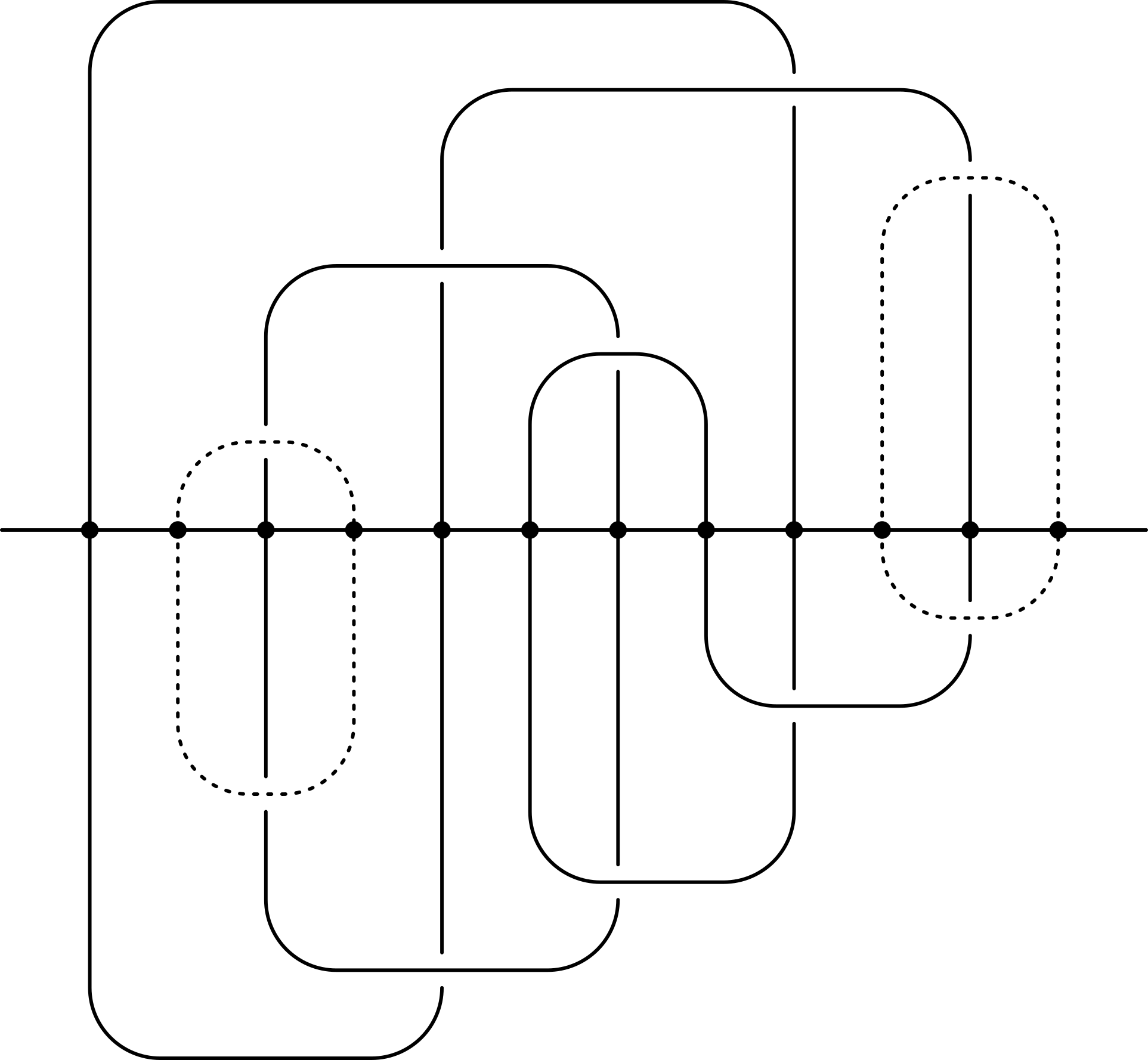}
        \put(0, 30){$\overline{\widehat{{T_{\mathcal{J}'}}}}$}
        \put(0, 110){$\widehat{T_{\mathcal{I}'}}$}
        \caption{$L_{(\mathcal{I}', \mathcal{J}')}$}
		\label{fig: non-reduced pair no shift example}
    \end{subfigure}
    \caption{Equivalently, $L_g = L_{(\mathcal{I}, \mathcal{J})}  = \overline{\widehat{{T_{\mathcal{J}}}}} \widehat{T_{\mathcal{I}}}$, where $(\mathcal{I}, \mathcal{J})$ is the reduced s.d. partition pair representing $g$, and $\widehat{T_{\mathcal{I}}}$ and $\widehat{T_{\mathcal{J}}}$ are considered as $(0, 2n)$-tangles with unmarked endpoints. In general, $L_{(\mathcal{I}', \mathcal{J}')}  = \overline{\widehat{{T_{\mathcal{J}'}}}} \widehat{T_{\mathcal{I}'}}$ for any s.d. partition pair $(\mathcal{I}', \mathcal{J}')$ with same number of breakpoints.}
\end{figure}

\bigskip

In fact, what Jones proved in \cite{MR3589908} is slightly stronger than Proposition \ref{prop: unoriented Thompson Alexander theorem ver 1}, as the following theorem states.

\begin{theorem} \label{thm: unoriented Thompson Alexamder theorem ver 2}
	Given any unoriented link $L$, there exists $g\in F$ such that $L_g = L$.
\end{theorem}

This theorem is the analog of the ``Alexander theorem" for braids and links. It establishes that the Thompson group $F$ is in fact as good as the braid groups at producing unoriented links. So we can define Thompson index of unoriented links as an analogue of the braid index.

\begin{definition}
	Given an unoriented link $L$, we define its Thompson index $ind_{F}(L)$ to be the smallest leaf number required for an element $g$ of $F$ to give rise to $L_g=L$.
\end{definition}

\subsection{Oriented Thompson group and oriented Thompson link}
\label{sec: oriented thompson group}

In the previous section we saw how to construct $T_{\mathcal{I}}$ from $F_{\mathcal{I}}$ for s.d. $n$-partition $\mathcal{I}$. Now we show how to give the canonical orientation to $T_{\mathcal{I}}$ from the signed region of $F_{\mathcal{I}}$'s complement, following \cite{MR3827807}. Given binary tree $F_{\mathcal{I}}$ embedded in $[0,1] \times [0,1]$, the complement of $F_{\mathcal{I}}$ has $(n+1)$ regions. We assign $+$ to the leftmost region, then apply the rules shown in Figure \ref{fig: region sign rule} $(n-1)$ times to assign signs to all other regions except the rightmost one. See Figure \ref{fig: signed tree example} for an example.

\begin{figure}
    \centering
    \begin{subfigure}{0.4\textwidth}
        \centering
        \includegraphics[width=\textwidth]{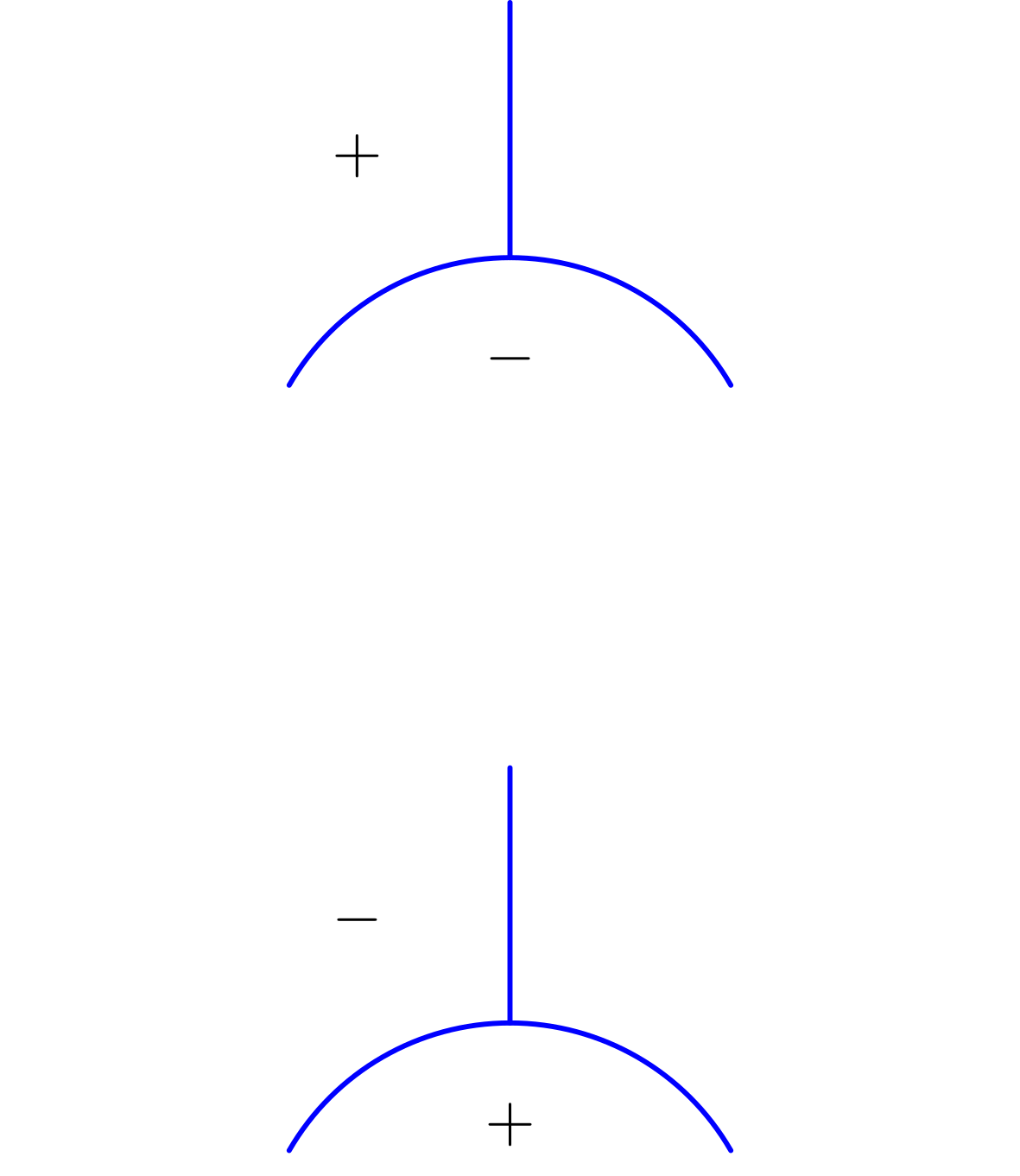}
        \caption{} 
        \label{fig: region sign rule}
    \end{subfigure}
    \hfill
    \begin{subfigure}{0.4\textwidth}
        \centering
        \includegraphics[width=\textwidth]{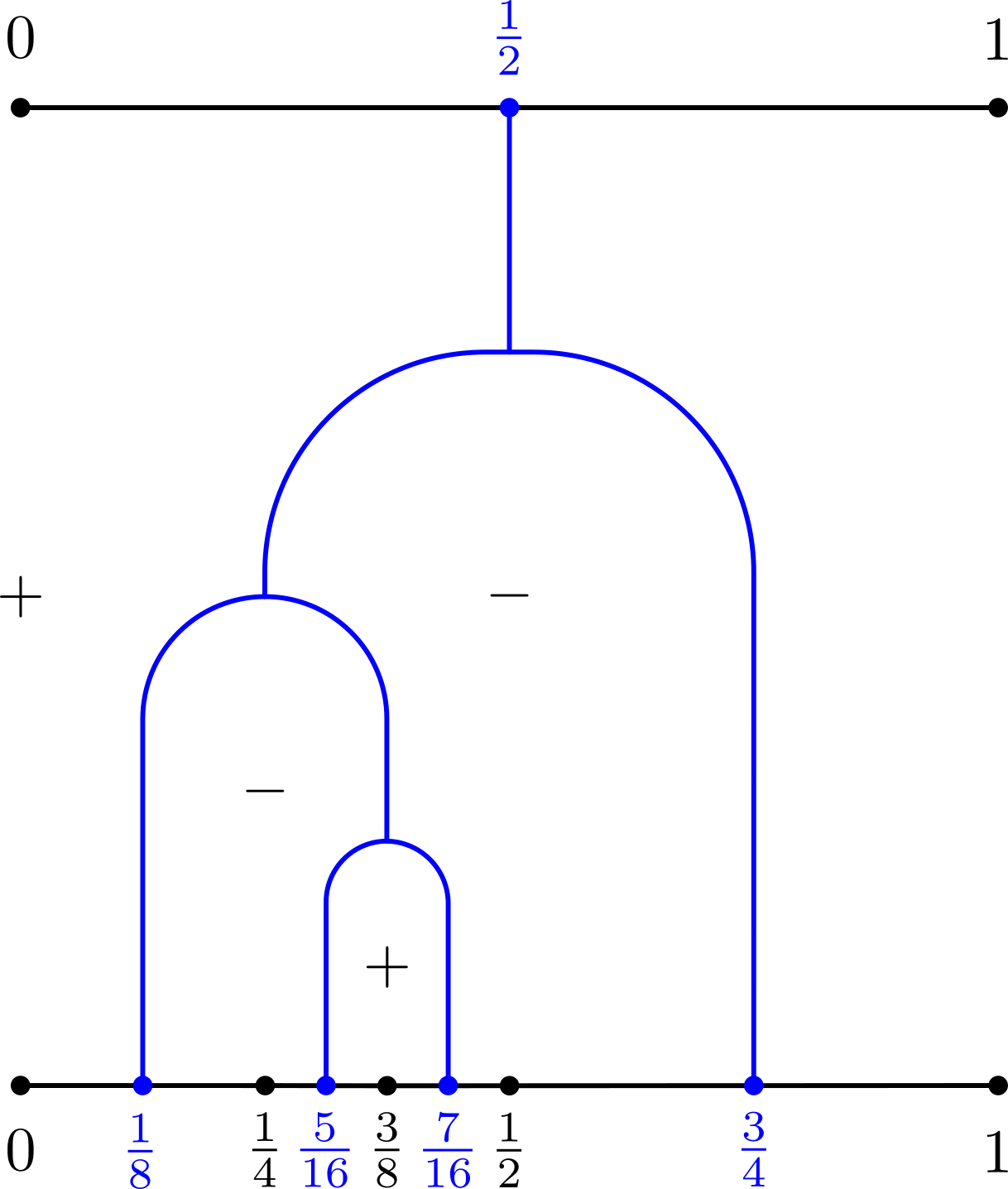}
        \put(-190, 15){$\mathcal{I}$}
        \put(-190, 180){$\delta$}
        \put(-190, 120){$\color{blue} F_{\mathcal{I}}$}
        \caption{}
        \label{fig: signed tree example}
    \end{subfigure}
    \caption{(a) Rules of signing on regions in complement of embedded binary tree. (b) A tree $F_{\mathcal{I}}$ with signs on regions.}
\end{figure}

\begin{definition}
	For $n = 1, 2, 3, ...$, an $n$-sign is a sequence of $n$ $+$'s and $-$'s such that the first sign is $+$ and the second is $-$ (if $n \geq 2$).
\end{definition}

Suppose $\mathcal{I}$ is an s.d. $n$-partition with breakpoints $0 = a_0 < a_1 < ... < a_n = 1$, then $F_{\mathcal{I}}$ induces an $n$-sign $\sigma_{\mathcal{I}}$ such that the $i^{th}$ sign in $\sigma_{\mathcal{I}}$ is the sign of the region above breakpoint $a_{i - 1}$.

\begin{definition} \label{def: oriented Thompson group}
	The oriented Thompson group $\vec{F}$ is the collection of $g$'s in $F$ such that for some s.d. $n$-partition pair $(\mathcal{I}, \mathcal{J})$ representing $g$, their corresponding signed trees $F_{\mathcal{I}}$ and $F_{\mathcal{J}}$ induce same $n$-sign $\sigma_{\mathcal{I}} = \sigma_{\mathcal{J}}$.
\end{definition}

\begin{remark} \label{rmk: replace some by any in def of oriented F}
	In Definition \ref{def: oriented Thompson group}, we can replace ``some s.d. $n$-partition pair" by ``any s.d. $n$-partition pair". Indeed, suppose that the tree diagram $(F_{\mathcal{I}'}, F_{\mathcal{J}'})$ is obtained from $(F_{\mathcal{I}}, F_{\mathcal{J}})$ by adding a caret at the $i^{th}$ leaf. From the rules shown in Figure \ref{fig: region sign rule}, to obtain $\sigma_{\mathcal{I}'}$ from $\sigma_{\mathcal{I}}$, we just need to replace the $i^{th}$ sign of $\sigma_{\mathcal{I}}$ by $(+, -)$ if the original sign is $+$, by $(-,+)$ if the original sign is $-$. So $\sigma_{\mathcal{I}} = \sigma_{\mathcal{J}}$ if and only if $\sigma_{\mathcal{I}'} = \sigma_{\mathcal{J}'}$.
\end{remark}

\begin{remark}
	It is hard to see that $\vec{F}$ is a subgroup of $F$, because it is not obvious that $\vec{F}$ is closed under multiplication from Definition \ref{def: oriented Thompson group}. We will introduce an equivalent algebraic definition of $\vec{F}$ in Section \ref{sec: signs and alg def}, where we can easily see that $\vec{F}$ is a subgroup of $F$.
\end{remark}

The notation of $n$-sign was introduced in \cite{MR3827807}, where $\vec{F}$ is described as group of fractions of a category $\vec{\mathfrak{F}}$. It turns out that Definition \ref{def: oriented Thompson group} is equivalent to the definition of $\vec{F}$ as group of fractions. See Appendix \ref{appx: group of fractions} for more details.

\bigskip

Suppose that $F_{\mathcal{I}}$ is a tree with signed regions. When we add vertical strands to construct $T_{\mathcal{I}}$, each strand separates each signed region into two regions, we move the sign to the right region of the added strand. To obtain the canonical orientation on $T_{\mathcal{I}}$, we let each $+$-signed region induce counterclockwise orientation on its boundary arcs, and each $-$-signed region induce clockwise orientation on its boundary arcs. See Figure \ref{fig: tangle orientation from signed regions example} for an example.

\begin{figure}
    \centering
    \begin{subfigure}{\textwidth}
        \centering
        \includegraphics[width = \textwidth]{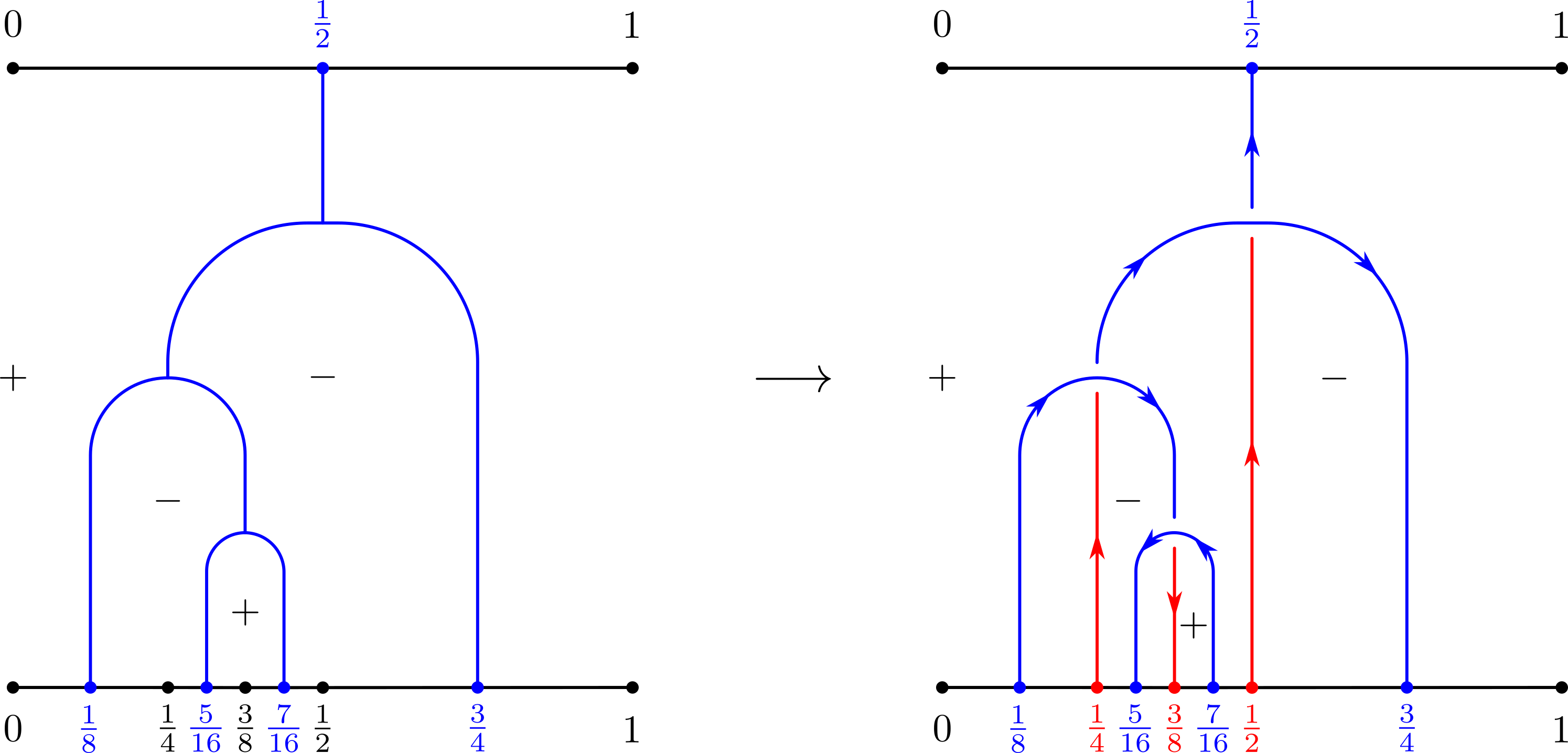}
        \put(-450, 15){$\mathcal{I}$}
	    \put(-450, 185){$\delta$}
	    \put(-450, 120){$\color{blue} F_{\mathcal{I}}$}
	    \put(-195, 15){$\mathcal{I}$}
	    \put(-195, 185){$\delta$}
	    \put(0, 120){$\vec{T_{\mathcal{I}}}$}
        \caption{}
        \label{fig: tangle orientation from signed regions example}
    \end{subfigure}
    \vspace{1cm}
    \vfill
    \begin{subfigure}{0.4\textwidth}
        \centering
        \includegraphics[width = \textwidth]{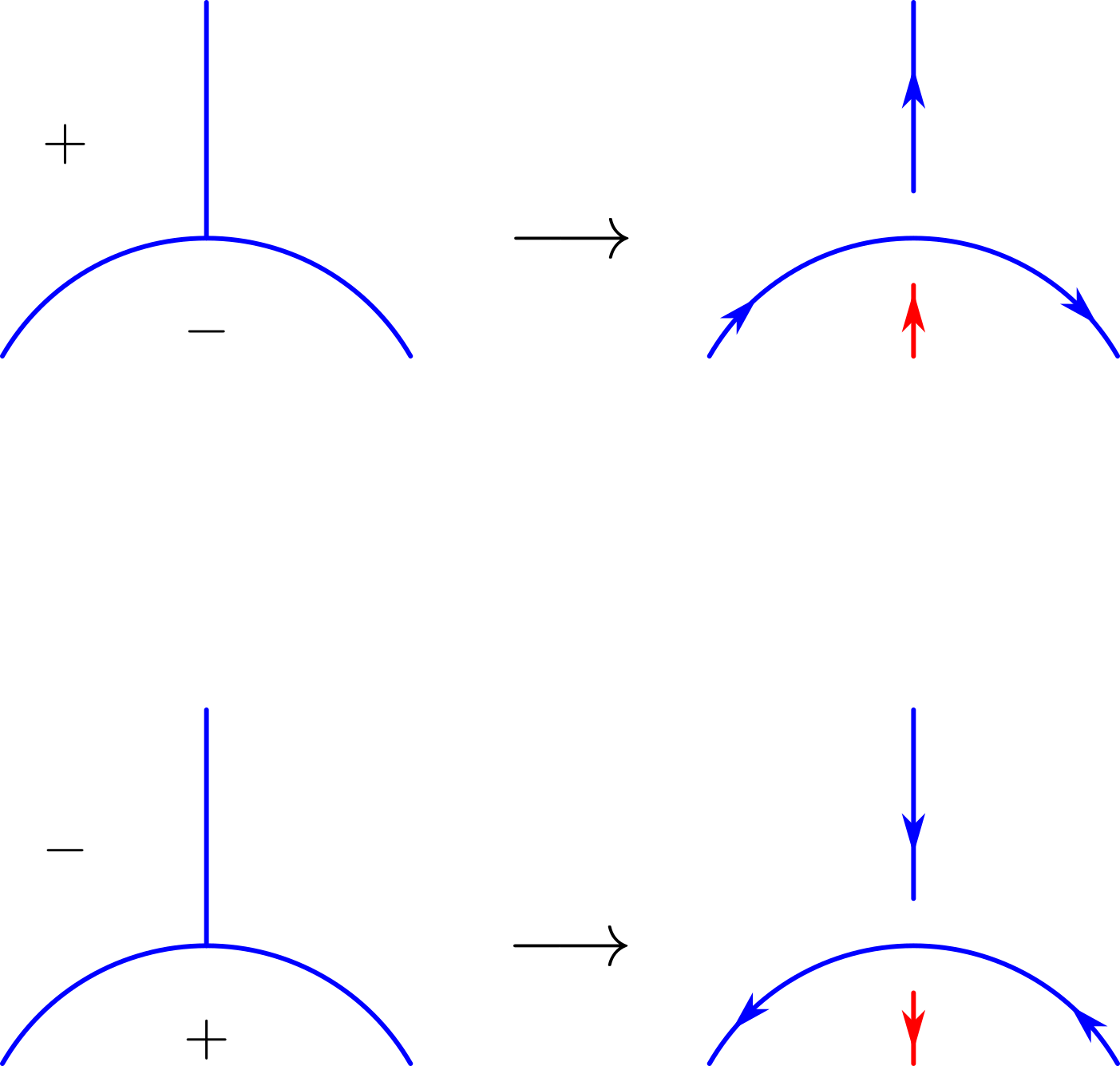}
        \caption{}
        \label{fig: local tangle orientation in terms of signed regions}
    \end{subfigure}
    \hfill
    \begin{subfigure}{0.4\textwidth}
        \centering
        \includegraphics[width = \textwidth]{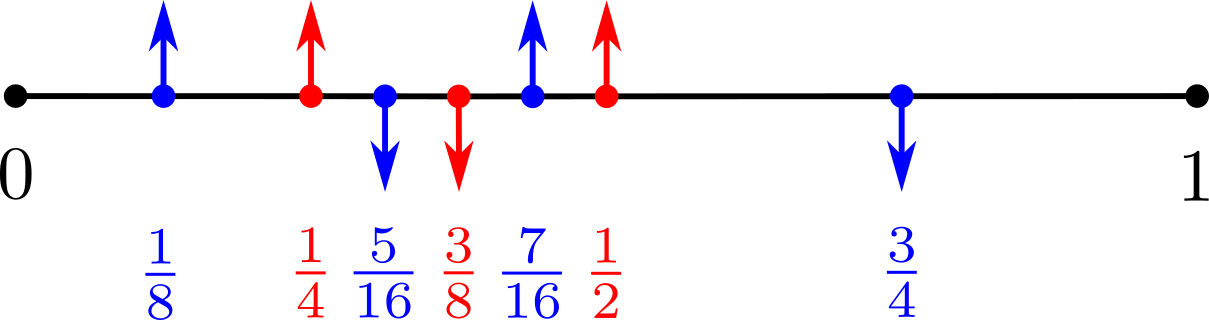}
        \put(-195, 30){$\mathcal{I}$}
        \caption{}
        \label{fig: E(I) signs example}
    \end{subfigure}
    \caption{(a) Signs on regions in complement of $F_{\mathcal{I}}$ give the orientation of $\vec{T_\mathcal{I}}$. (b) Local rules to give the canonical orientation of $\vec{T_\mathcal{I}}$ in terms of signed regions. (c) The signing on $E(\mathcal{I})$ induced by $\vec{T_{\mathcal{I}}}$, where upward arrows mean $+$ signs, downward arrows mean $-$ signs.} 
\end{figure}
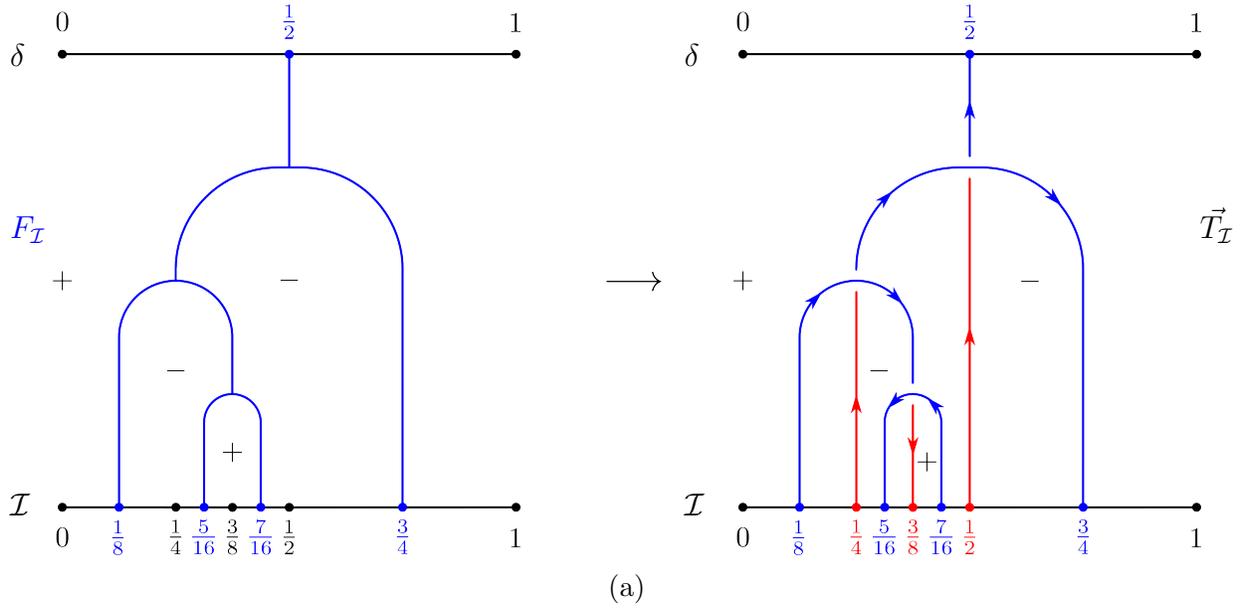
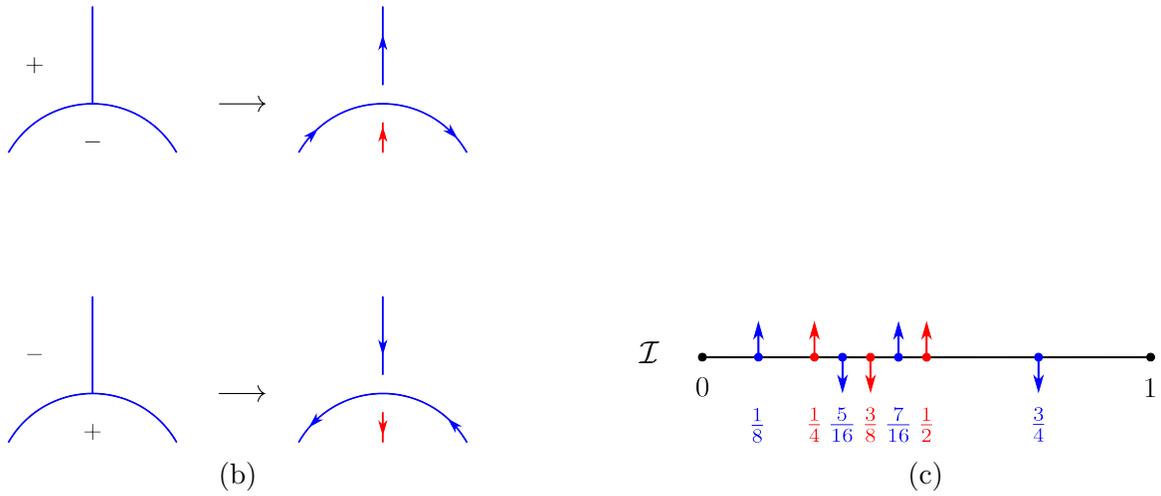

Equivalently, the above procedure can be expressed locally as shown in Figure \ref{fig: local tangle orientation in terms of signed regions}. It shows that the induced orientations of arcs are compatible around each crossing. We denote the oriented $T_{\mathcal{I}}$ as $\vec{T}_{\mathcal{I}}$.

Furthermore, $\vec{T_{\mathcal{I}}}$ induces a canonical signing on $E(\mathcal{I})$ as the lower boundary of $\vec{T_{\mathcal{I}}}$. We let upward orientation induce sign $+$, let downward orientation induce sign $-$. For instance, $E(\delta)$ contains a single point $\frac{1}{2}$ with the $+$ sign. See Figure \ref{fig: E(I) signs example} for another example. Notice that the signing on $E(\mathcal{I})$ is determined by the $n$-sign $\sigma_{\mathcal{I}}$ induced by $F_{\mathcal{I}}$. More specifically, the sign on the midpoint $\frac{a_i + a_{i+1}}{2}$ is the ame as the $(i+1)^{th}$ sign, and the sign on breakpoint $a_i$ is opposite of the sign on midpoint $\frac{a_i + a_{i+1}}{2}$. Conversely, $\sigma_{\mathcal{I}}$ can be recovered by taking signs on midpoints in $E(\mathcal{I})$.

Recall if a pair of s.d. partitions $(\mathcal{I}, \mathcal{J})$ represents an element $g \in F$, then $L_{(\mathcal{I}, \mathcal{J})}$ is the close-up of $\overline{T_\mathcal{J}} g^{\mathcal{I}}_{\mathcal{J}} T_\mathcal{I}$. However, if we replace $T_{\mathcal{I}}$ by $\vec{T}_{\mathcal{I}}$ and replace $\overline{T_{\mathcal{J}}}$ by $-\overline{\vec{T_{\mathcal{J}}}}$, we need $g^{\mathcal{I}}_{\mathcal{J}}$ to have some orientation compatible with the orientations of $\vec{T}_{\mathcal{I}}$ and $-\overline{\vec{T_{\mathcal{J}}}}$. In other words, we need $E(\mathcal{I})$ and $E(\mathcal{J})$ to have compatible signings.

\begin{definition} \label{def: compatible pair of s.d. partitions}
    Let $\mathcal{I}$ and $\mathcal{J}$ be a pair of s.d. partitions with the same number of breakpoints, we say $\mathcal{I}$ and $\mathcal{J}$ are compatible, or $(\mathcal{I}, \mathcal{J})$ is a compatible s.d. partition pair if $(\mathcal{I},\mathcal{J})$ represents an element $g \in \vec{F}$.
\end{definition}

By Remark \ref{rmk: replace some by any in def of oriented F}, if $(\mathcal{I}, \mathcal{J})$ is a compatible pair representing $g \in \vec{F}$, then $F_{\mathcal{I}}$ and $F_{\mathcal{J}}$ induce same $n$-sign, which determines compatible signings on $E(\mathcal{I})$ and $E(\mathcal{J})$. Then $g_{\mathcal{I}}^{\mathcal{J}}$ can be assigned a unique compatible orientation, denoted as $\vec{g}_{\mathcal{I}}^{\mathcal{J}}$. We let $\vec{L}_{(\mathcal{I}, \mathcal{J})}$ be the oriented link obtained by closing up the oriented tangle $(-\overline{\vec{T_\mathcal{J}}}) \vec{g^{\mathcal{I}}_{\mathcal{J}}} \vec{{T_\mathcal{I}}}$. Similar to the unoriented scenario, we have a description in terms of $(0,2n)$-tangles. We can first add a trivial strand on the leftmost of $\vec{T}_\mathcal{I}$ and $\vec{T}_\mathcal{J}$ respectively, with orientation from bottom to top. Then cap them off to get $(0,2n)$-tangles $\widehat{\vec{T}_\mathcal{I}}$ and $\widehat{\vec{T}_\mathcal{J}}$. We have $\vec{L}_{(\mathcal{I}, \mathcal{J})} = (-\overline{\widehat{\vec{T_\mathcal{J}}}}) \widehat{\vec{{T_\mathcal{I}}}}$.

\begin{figure}
    \centering
    \begin{subfigure}{0.4\textwidth}
        \centering
        \includegraphics[width=\textwidth]{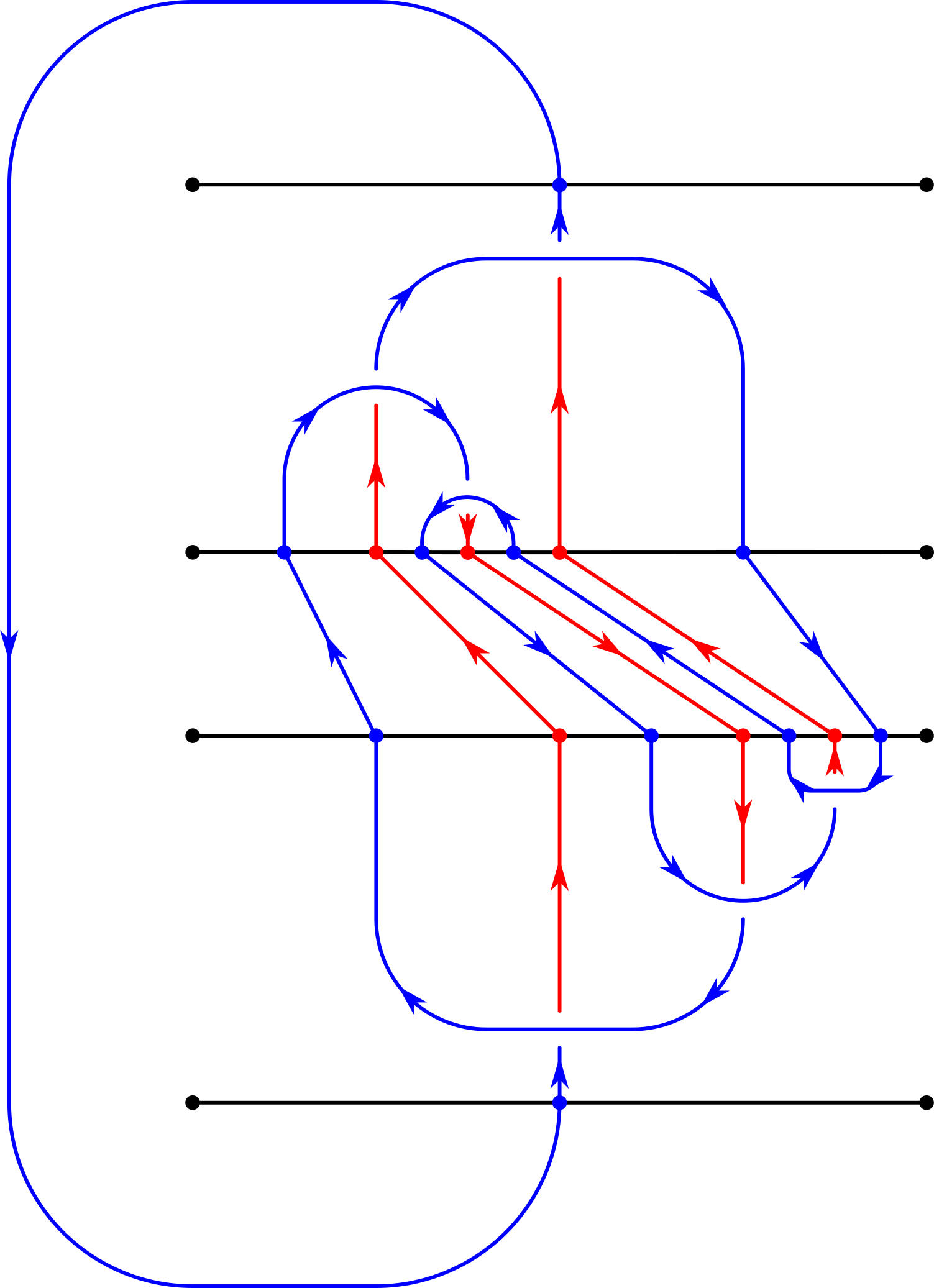}
        \put(-160, 30){$\delta$}
        \put(-160, 95){$\mathcal{J}$}
        \put(-160, 135){$\mathcal{I}$}
        \put(-160, 200){$\delta$}
        \put(-5, 60){$-\overline{\vec{T_\mathcal{J}}}$}
        \put(0, 115){$\vec{g}^{\mathcal{I}}_{\mathcal{J}}$}
        \put(0, 170){$\vec{T_\mathcal{I}}$}
        \caption{$\vec{L}_{(\mathcal{I}, \mathcal{J})}$}
        \label{fig: construction of oriented link with shift example}
    \end{subfigure}
    \hfill
    \begin{subfigure}{0.4\textwidth}
        \centering
        \includegraphics[width=\textwidth]{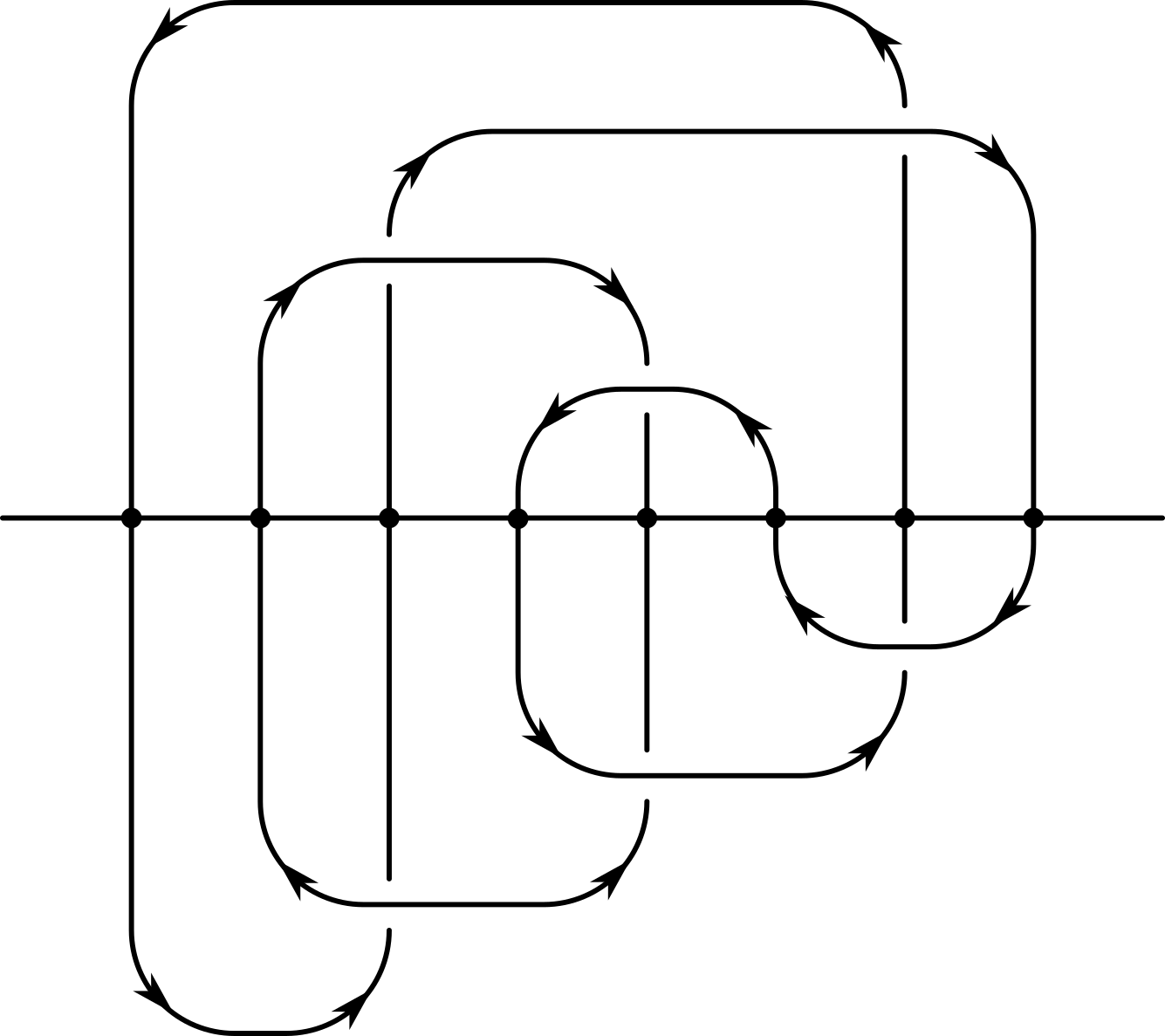}
        \put(-5, 30){$-\overline{\widehat{{\vec{T_{\mathcal{J}}}}}}$}
        \put(0, 110){$\widehat{\vec{T_{\mathcal{I}}}}$}
        \caption{$\vec{L}_{(\mathcal{I}, \mathcal{J})}$}
        \label{fig: construction of oriented link with no shift example}
    \end{subfigure}
    \caption{(a) $\vec{L}_{(\mathcal{I}, \mathcal{J})}$ is the closing-up of $(-\overline{\vec{T_\mathcal{J}}}) \vec{g}_{\mathcal{I}}^{\mathcal{J}} \vec{{T_\mathcal{I}}}$. (b) Equivalently, $\vec{L}_{(\mathcal{I}, \mathcal{J})}  = (-\overline{\widehat{\vec{T_\mathcal{J}}}}) \widehat{\vec{{T_\mathcal{I}}}}$ where $\widehat{\vec{T_{\mathcal{I}}}}$ and $\widehat{\vec{T_{\mathcal{J}}}}$ are $(0, 2n)$-oriented tangles with unmarked endpoints.}
\end{figure}

\begin{lemma} \label{lem: writhe of Thompson link is 0}
    Let $\mathcal{I}$ be an s.d. $n$-partition, then the oriented $(0, 2n)$-tangle $\widehat{\vec{T}_\mathcal{I}}$ has $(n-1)$ crossings and all of them are positive. Furthermore, if $(\mathcal{I}, \mathcal{J})$ is a compatible s.d. partition pair, then $\vec{L}_{(\mathcal{I}, \mathcal{J})} = (-\overline{\widehat{\vec{T_\mathcal{J}}}}) \widehat{\vec{{T_\mathcal{I}}}}$ has writhe $0$.
\end{lemma}

\begin{proof}
    Each crossing in $\widehat{\vec{T}_{\mathcal{I}}}$ comes from one of $(n - 1)$ trivalent vertices of $F_{\mathcal{I}}$, and it must be one of the two cases shown in Figure \ref{fig: local tangle orientation in terms of signed edges}. Notice that either case generates a positive crossing, so $\widehat{\vec{T}_{\mathcal{I}}}$ has $(n - 1)$ positive crossings.

    Now suppose $(\mathcal{I}, \mathcal{J})$ is a compatible s.d. $n$-partition pair, then $\widehat{\vec{T}_{\mathcal{I}}}$ has $(n - 1)$ positive crossings, $\overline{\widehat{\vec{T_\mathcal{J}}}}$ has $(n - 1)$ negative crossings. So $\vec{L}_{(\mathcal{I}, \mathcal{J})}  = \overline{\widehat{\vec{T_\mathcal{J}}}} \widehat{\vec{{T_\mathcal{I}}}}$ has writhe $0$.
\end{proof}

Recall that when $(\mathcal{I}, \mathcal{J})$ is the reduced pair of s.d. partitions representing $g$, we denote $L_{(\mathcal{I}, \mathcal{J})}$ as $L_g$. So $L_g$ also has an orientation $\vec{L}_g$. Unless otherwise stated, we will consider every link and tangle as oriented and omit arrow notations in the rest of this paper, except Section \ref{sec: sym rep and unoriented grid}.

There is also an Alexander theorem for the oriented Thompson group $\vec{F}$.

\begin{theorem} \cite{MR4071378} \label{thm: oriented Thompson Alexander theorem}
	Given any oriented link $L$, there exists $g\in \vec{F}$ such that $L = L_g$.
\end{theorem}

So similar to the unoriented Thompson index of unoriented links, we can define oriented Thompson index of oriented links.

\begin{definition}
	Given an oriented link $L$, we define its oriented Thompson index $ind_{\vec{F}}(L)$ to be the smallest leaf number required for an element $g$ of $\vec{F}$ to give rise to $L_g = L$.
\end{definition}

\section{Grid Construction for oriented Thompson link}
\label{grid construction and equivalence}

In this section we will introduce half grid construction from s.d. partitions, and grid construction from elements in $\vec{F}$. In Section \ref{sec: signs and alg def} we first introduce signs on s.d. intervals, and we give an algebraic definition of $\vec{F}$ based on those signs. Section \ref{sec: grid construction} is our main construction, using ``midpoint order'' and ``length order'' along with signs of s.d. intervals, to construct half grid diagram. Then we show the algebraic definition of $\vec{F}$ implies compatibility of two half grid diagrams given by $g \in \vec{F}$. In Section \ref{sec: equivalence} we show this grid construction is equivalent to Jones' construction of oriented Thompson link. In the end, we prove our main theorem that every oriented link is half grid presentable.

\subsection{Signs of standard dyadic intervals and an algebraic definition of oriented Thompson group}
\label{sec: signs and alg def}

Given an s.d. partition $\mathcal{I}$ with breakpoints $0 = a_0 < a_1 < ... < a_n = 1$. Recall that $\mathcal{S}(\mathcal{I})$ is the collection of $[a_i, a_j]$'s such that $[a_i, a_j]$ is an s.d. interval. We want to construct a half grid diagram $\mathbb{H}_{\mathcal{I}}$, such that each s.d. interval $A \in \mathcal{S}({\mathcal{I}})$ corresponds to a marking in $\mathbb{H}_{\mathcal{I}}$.

To achieve this we need to determine what and where these markings are. We first introduce the sign function on $\mathcal{S}$, the set of all s.d. intervals.

\begin{definition} \label{def: sign on s.d. interval}
	Define $s:S\rightarrow \{+,-\}$ to be the map uniquely determined by the following conditions.
	\begin{enumerate} 
		\item $s([0,1])=+$
		\item For any $A\in \mathcal{S} \setminus \{[0,1]\}$, we have $$s(A)=
		\begin{cases}
			+ & \text{if } A \in \mathcal{S}_L \text{ and } s(A \cup \overline{A})=+\\
			- & \text{if } A \in \mathcal{S}_L  \text{ and } s(A \cup \overline{A})=-\\
			+ & \text{if } A \in \mathcal{S}_R \text{ and } s(A \cup \overline{A})=-\\
			- & \text{if } A \in \mathcal{S}_R \text{ and } s(A \cup \overline{A})=+
		\end{cases}$$  
	\end{enumerate}
\end{definition}

Note that the above conditions in fact give us an inductive definition of $s$. Condition 1 tells us $s([0,1])=+$, then we know $s([0,\frac{1}{2}])=+$ and $s([\frac{1}{2},1])=-$ by condition 2. Generally, we can apply the second condition $n$ times to get the value of $s$ on an s.d. interval of length $\frac{1}{2^n}$.

Recall that $E$ is the set of rational dyadic numbers except $0$ and $1$, and we have a bijection $m: \mathcal{S} \to E$, sending s.d. intervals to their midpoints. Then we put signs on $E$ so that $A$ and $m(A)$ have the same sign for any $A \in \mathcal{S}$. Also recall Lemma \ref{lem: g restricted on E is bijection}, stating that any $g \in F$ restricted on $E$ is a bijection from $E$ to itself.

\begin{definition}
	Define $\hat{F}$ to be the subset of $F$ consisting of homeomorphisms that preserve signs on $E$.
\end{definition}

Note that if $g, h \in F$ preserve signs on $E$, then $g^{-1} h$ also preserves signs on $E$. Thus, $\hat{F}$ is a subgroup of $F$.

\begin{theorem} \label{thm: equivalent definition of oriented Thompson group}
	$\hat{F} = \vec{F}$ as subgroups of $F$.
\end{theorem}

Before we prove Theorem \ref{thm: equivalent definition of oriented Thompson group}, we first give an equivalent definition of the canonical orientation of $T_I$. Recall that given an s.d. $n$-partition $\mathcal{I}$, the embedded tree $F_{\mathcal{I}}$ separates $[0,1]\times [0,1]$ into $(n+1)$ regions, and each region has a sign except the rightmost one. Now we assign signs to all of the edges of $F_{\mathcal{I}}$ instead, by letting each edge have the same sign as the region on its left. Correspondingly, rules of assigning signs on regions become rules of assigning signs on edges as Figure \ref{fig: signing rules on edges} shows, and the requirement that the leftmost region has sign $+$ becomes the requirement that the uppermost edge has sign $+$.

In Section \ref{sec: some general theories of sdp} we saw a bijection $\epsilon_{\mathcal{I}}: \mathcal{S}(\mathcal{I}) \to \mathcal{E}(F_{\mathcal{I}})$ sending left intervals to left edges, right intervals to right edges. The signs on $S(\mathcal{I})$ are given by the sign function $s$, defined by the recursive rules in Definition \ref{def: sign on s.d. interval}. Note that those rules are the same as the recursive signing rules (Figure \ref{fig: signing rules on edges}) on $\mathcal{E}(F_{\mathcal{I}})$ up to the bijection $\epsilon_{\mathcal{I}}$. Then we get the following lemma.

\begin{lemma} \label{lem: interval to edge is sign preserving}
	$\epsilon_{\mathcal{I}}: \mathcal{S}(\mathcal{I}) \to \mathcal{E}(F_{\mathcal{I}})$ is a sign-preserving bijection.
\end{lemma}

Notice that we can get the canonical orientation of $T_{\mathcal{I}}$ using signs on the edges of $F_{\mathcal{I}}$ instead, because we just moved signs from regions to edges. So similar to the local rules to assign orientation of $T_{\mathcal{I}}$ in terms of signed regions (see Figure \ref{fig: local tangle orientation in terms of signed regions}), we have corresponding local rules based on signed edges, shown in Figure \ref{fig: local tangle orientation in terms of signed edges}. See Figure \ref{fig: tangle orientation from signed edges example} for an example of recovering the canonical orientation of $T_{\mathcal{I}}$ from the signed edges of $F_{\mathcal{I}}$.

\begin{figure}
    \centering
    \begin{subfigure}{0.4\textwidth}
        \includegraphics[width = \textwidth]{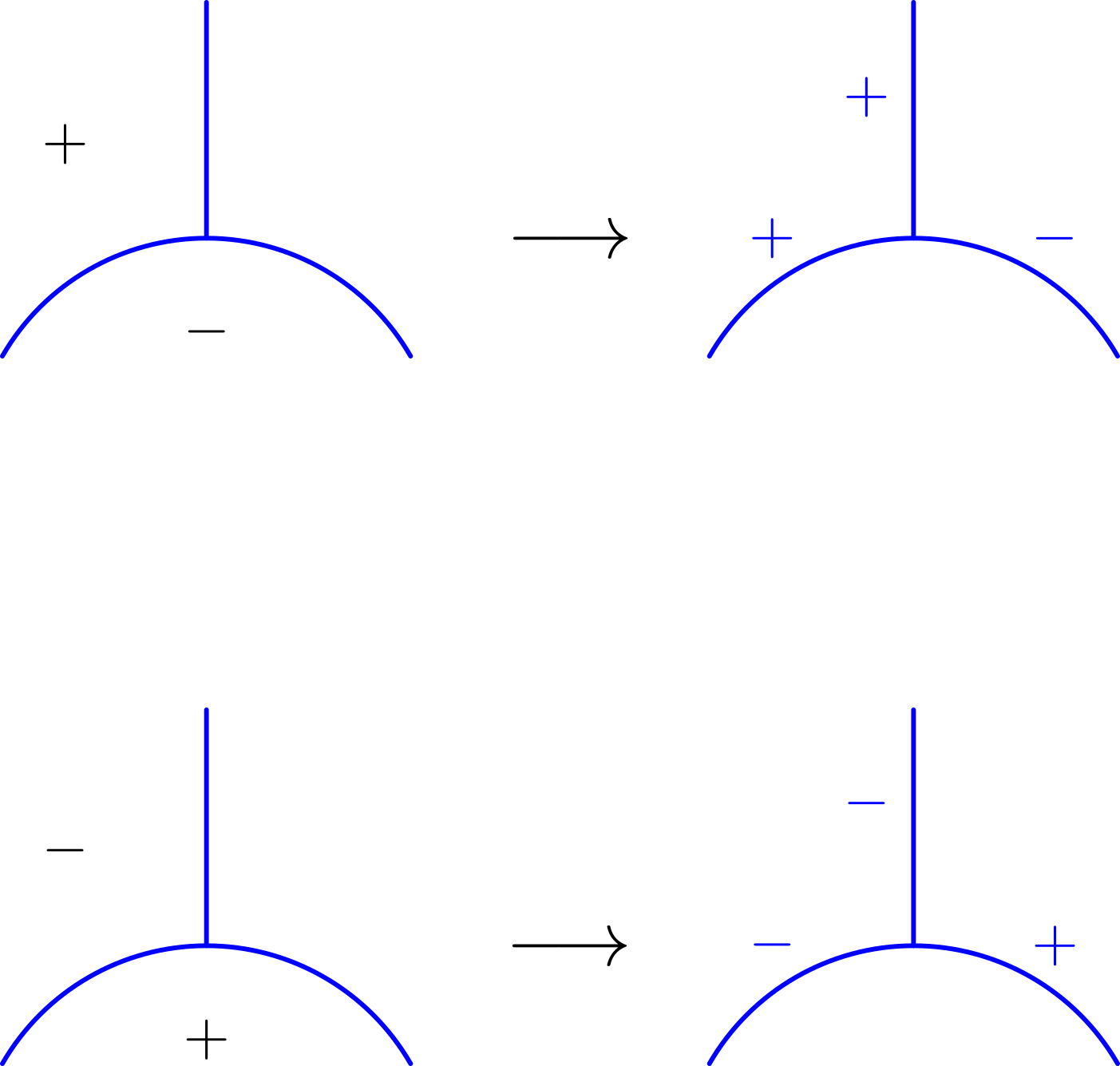}
        \caption{}
        \label{fig: signing rules on edges}
    \end{subfigure}
    \hfill
    \begin{subfigure}{0.4\textwidth}
        \includegraphics[width = \textwidth]{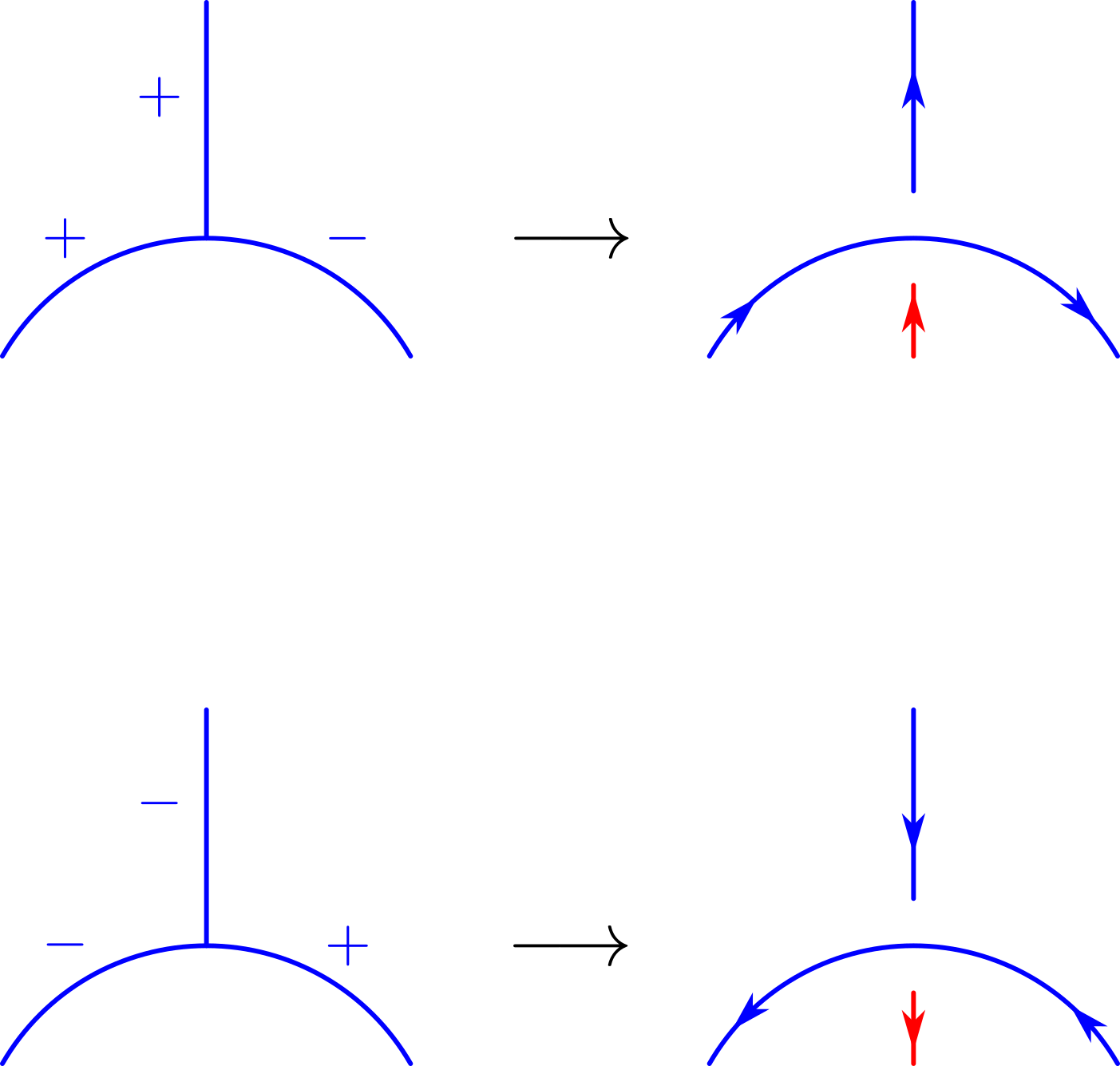}
        \caption{}
        \label{fig: local tangle orientation in terms of signed edges}
    \end{subfigure}
	\vspace{1cm}
	\vfill
	\begin{subfigure}{\textwidth}
		\includegraphics[width = \textwidth]{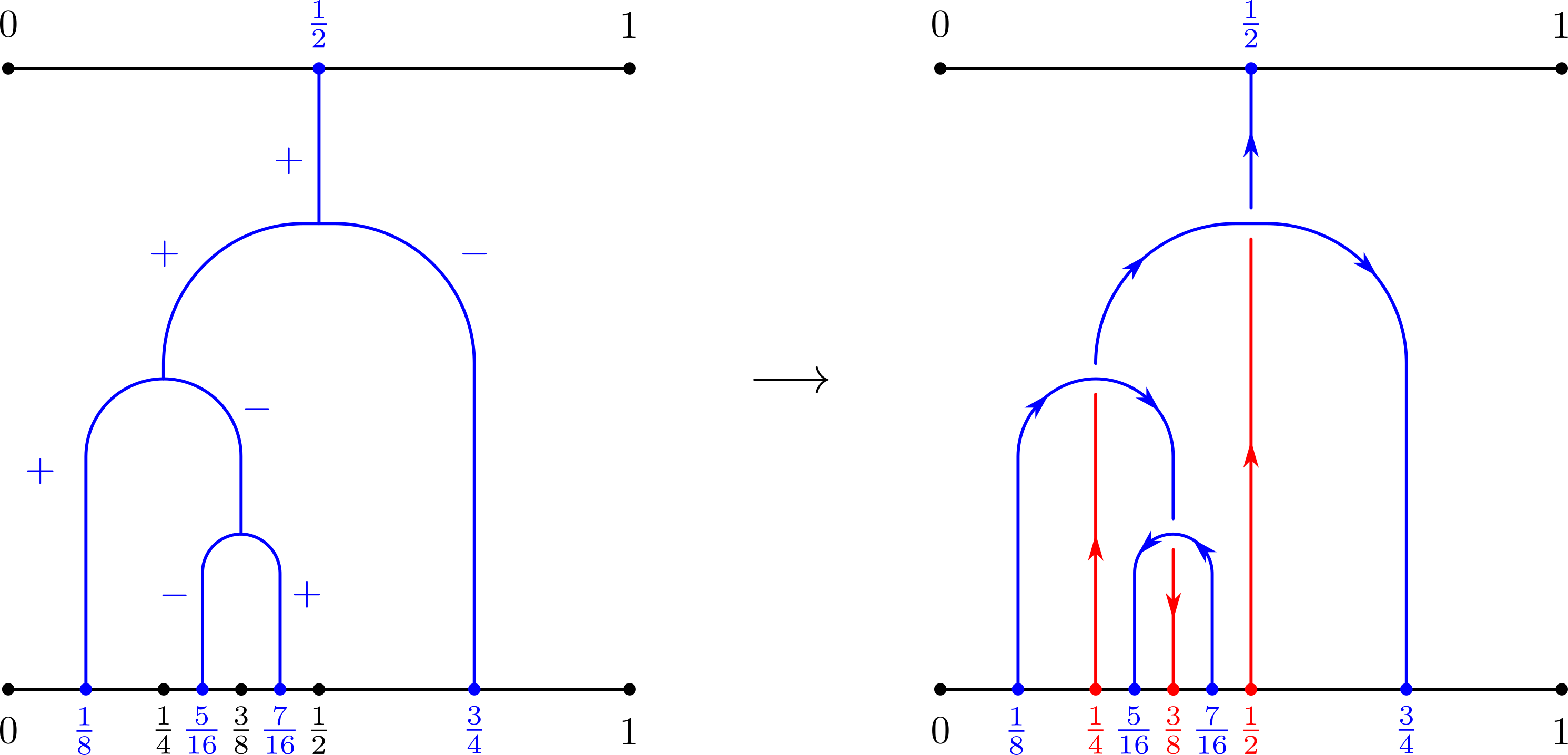}
		\put(-450, 15){$\mathcal{I}$}
	    \put(-450, 185){$\delta$}
	    \put(-450, 120){$\color{blue} F_{\mathcal{I}}$}
	    \put(-195, 15){$\mathcal{I}$}
	    \put(-195, 185){$\delta$}
	    \put(0, 120){$T_{\mathcal{I}}$}
		\caption{}
		\label{fig: tangle orientation from signed edges example}
	\end{subfigure}
    \caption{(a) Two local pictures on the left are the signing rules of $F_{\mathcal{I}}$ on regions. Two local pictures on the right are the corresponding signing rules of $F_{\mathcal{I}}$ on edges, obtained by letting each edge have the same sign as the region on its left. (b) Local rules to give the canonical orientation of $T_\mathcal{I}$ in terms of signed edges, obtained by combining Figure \ref{fig: local tangle orientation in terms of signed regions} and \ref{fig: signing rules on edges}. (c) An example that signs on edges of $F_{\mathcal{I}}$ give orientation of $T_{\mathcal{I}}$. Notice the correspondence between the sign on an edge of $F_{\mathcal{I}}$ and the sign on the lower endpoint of $T_{\mathcal{I}}$ below this edge.}
\end{figure}
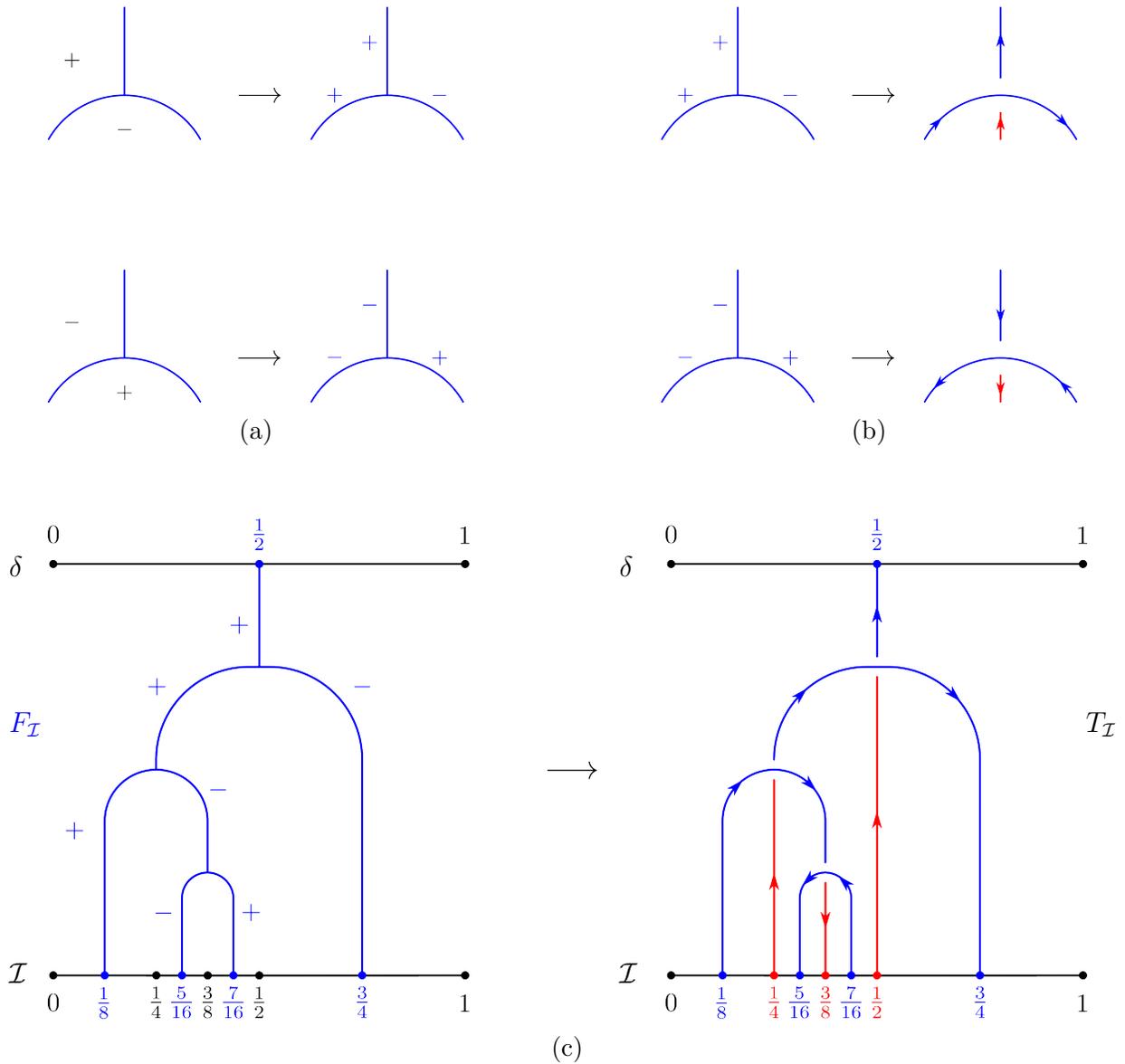

\bigskip

Recall that each point in $E(\mathcal{I})$ has a sign induced by the canonical orientation of $T_{\mathcal{I}}$.

\begin{lemma}
	$\phi_{\mathcal{I}}: \mathcal{E}(F_{\mathcal{I}}) \to E(\mathcal{I})$ is a sign-preserving bijection.
\end{lemma}

\begin{proof}
	Suppose that $p\in E(\mathcal{I})$, then $p$ is either a breakpoint of $\mathcal{I}$ or a midpoint of a subinterval of $\mathcal{I}$. If $p$ is a breakpoint, then it is attached to a Type B arc $b$ of $T_{\mathcal{I}}$, which extends Type A arc $a(\phi_{\mathcal{I}}^{-1}(p))$ to $p$. From Figure \ref{fig: local tangle orientation in terms of signed edges} and \ref{fig: tangle orientation from signed edges example}, we can observe that a Type A arc $a(e)$ has an upward orientation if $e$ has sign $+$, downward orientation if $e$ has sign $-$. Since the sign of $p$ is induced by $b$, which extends the orientation of $a(\phi_{\mathcal{I}}^{-1}(p))$, we must have the same sign on $p$ and $\phi_{\mathcal{I}}^{-1}(p)$. If $p$ is a midpoint, then it is attached to Type A arc $a(\phi_{\mathcal{I}}^{-1}(p))$ directly, so $p$ and $\phi_{\mathcal{I}}^{-1}(p)$ have the same sign.
\end{proof}

Thus, we have a sign preserving bijection $m|_{\mathcal{S}(\mathcal{I})} = \phi_{\mathcal{I}} \circ \epsilon_{\mathcal{I}}: \mathcal{S}(\mathcal{I}) \to E(\mathcal{I})$, where the signs on $\mathcal{S}(\mathcal{I})$ are given by the sign function $s$, the signs on $E(\mathcal{I})$ are induced by the canonical orientation of $T_\mathcal{I}$. In other words, the signs on $E(\mathcal{I})$ can be equivalently obtained from the signs on $\mathcal{S}(\mathcal{I})$, by taking midpoints. Notice that it is exactly how we defined signs on $E$.

\begin{proof} [Proof of Theorem \ref{thm: equivalent definition of oriented Thompson group}]
	Suppose that $g \in \vec{F}$. For any $a \in E$, we can find an s.d. partition pair $(\mathcal{I}, \mathcal{J})$ representing $g$ such that $a$ is a breakpoint of $\mathcal{I}$. Remember that the signs on $E(\mathcal{I})$ and $E(\mathcal{J})$ are determined by $\sigma_{\mathcal{I}}$ and $\sigma_{\mathcal{J}}$ , respectively. So $\sigma_{\mathcal{I}} = \sigma_{\mathcal{J}}$ implies that the $i^{th}$ sign of $E(\mathcal{I})$ and the $i^{th}$ sign of $E(\mathcal{J})$ are same for any $i$. By Lemma \ref{lem: g restricted on E is bijection} and the fact that $g$ is increasing, we know that $g$ sends the $i^{th}$ point of $E(\mathcal{I})$ to the $i^{th}$ point of  $E(\mathcal{J})$ for any $i$. Thus $a \in E(\mathcal{I})$ and $g(a) \in E(\mathcal{J})$ have the same sign. It is true for any $a \in E$, so $g \in \hat{F}$.
	
	Conversely, suppose that $g \in \hat{F}$ preserves signs on $E$. Take an s.d. partition pair $(\mathcal{I}, \mathcal{J})$ representing $g$, then for any $a \in E(\mathcal{I})$, $a$ and $g(a)$ have the same sign. In other words, the $i^{th}$ sign of $E(\mathcal{I})$ and the $i^{th}$ sign of $E(\mathcal{J})$ are same for any $i$. Recall that $\sigma_{\mathcal{I}}$ and $\sigma_{\mathcal{J}}$ can be obtained by taking signs of midpoints in $E(\mathcal{I})$ and $E(\mathcal{J})$ respectively, so we have $\sigma_{\mathcal{I}} = \sigma_{\mathcal{J}}$. Thus $g \in \vec{F}$.
\end{proof}

Before we move on to the next section, let's first introduce some notations related to signs on $\mathcal{S}$. We denote $\mathcal{S}_+ = \{A\in \mathcal{S}: s(A)=+\}$ and $\mathcal{S}_- = \{A\in \mathcal{S}: s(A)=-\}$. An element in $\mathcal{S}_+$ is called a positive s.d. interval, an element in $\mathcal{S}_-$ is called a negative s.d. interval. From the definition of $s$, we can easily observe that taking conjugate of $A\in \mathcal{S}$ will change its sign.

Combine this observation with Lemma \ref{lem: cardinality of SDI}, and notice that $[0,1]\in \mathcal{S}_+$ has no conjugate, we have the following lemma.

\begin{lemma} \label{lem: cardinality of signed SDI}
	Suppose that $\mathcal{I}$ is an s.d. $n$-partition, then
	\begin{enumerate}
		\item $|\mathcal{S}_+(\mathcal{I})| = n$
		\item $|\mathcal{S}_-(\mathcal{I})| = n - 1$
	\end{enumerate}
\end{lemma}

\subsection{Half grid construction}
\label{sec: grid construction}

Given an s.d. $n$-partition $\mathcal{I}$, now we are ready to define position functions $p_x$ and $p_y$ on $\mathcal{S}(\mathcal{I})$, through the following two orders on $\mathcal{S}$.

\begin{definition}
	Let $(\mathcal{S}, \leq_{mid}, \leq_{len})$ represent the set $\mathcal{S}$ with two partial orders $\leq_{mid}$ and $\leq_{len}$. For $A,B\in \mathcal{S}$, we define $A\leq_{mid} B$ if $m(A) \leq m(B)$. Define $A \leq_{len} B$ if either the length of $A$ is strictly less than the length of $B$, or the length of $A$ equals the length $B$ and $A \leq_{mid} B$.
\end{definition}

It is easy to check that both of $\leq_{mid}$ and $\leq_{len}$ are total orders on $\mathcal{S}$.

We first restrict the midpoint order $\leq_{mid}$ on $\mathcal{S}(\mathcal{I})$. Since $\leq_{mid}$ is a total order, elements in $\mathcal{S}(\mathcal{I})$ can be ordered as $A_1 <_{mid} A_2 <_{mid} {...} <_{mid} A_{2n-1}$. We define a bijection

\begin{align*}
	p_x: \mathcal{S}(\mathcal{I})  &\to \{2,3,...,2n\} \\
	A_i &\mapsto i+1, \quad i = 1,2,...,2n-1
\end{align*}

Then restrict the length order on $\mathcal{S}_+(\mathcal{I})$ and reorder elements in $S_+(I)$ as $B_1 <_{len} B_2 <_{len} ... <_{len} B_{n}$. Notice that $B_n$ is always $[0, 1]$. We define

\begin{align*}
	p_y: & \mathcal{S}(\mathcal{I}) \to \{1,2,...,n\} \\
	&\begin{cases}
		B_i \mapsto i, & i=1,2,...,n\\
		 \overline{B_i} \mapsto i, & i=1,2,...,n-1
	\end{cases}
\end{align*}

Then $p_y|_{\mathcal{S}_+(\mathcal{I})}: \mathcal{S}_+(\mathcal{I}) \to \{1,2,...,n\}$ and $p_y|_{\mathcal{S}_-(\mathcal{I})}: \mathcal{S}_-(\mathcal{I}) \to \{1,2,...,n-1\}$ are both bijections.

Next, we will show how to construct an $n \times 2n$ half grid diagram associated to $\mathcal{I}$. First, we always assign marking $O$ to the $(1,n)^{th}$ entry by default. Then we let each s.d. interval $A\in \mathcal{S}(\mathcal{I})$ assign a marking to the $(p_x(A), p_y(A))^{th}$ entry in the  $n \times 2n$ grid, where the marking is $X$ if $A\in \mathcal{S}_+(\mathcal{I})$, $O$ if $A\in \mathcal{S}_-(\mathcal{I})$. We denote the resulting diagram as $\mathbb{H}_\mathcal{I}$.

\begin{proposition} {\label{prop: half grid diagram from SDI}}
	$\mathbb{H}_{\mathcal{I}}$ is a half grid diagram.
\end{proposition}

\begin{proof}
	First, it is impossible that there exists any entry with two markings, because $p_x: \mathcal{S}(\mathcal{I}) \to \{2, 3, ..., 2n\}$ is a bijection.

	Then we need to show that each column has exactly one marking and each row has an $X$ and an $O$. Indeed, the $O$ at the $(1,n)^{th}$ entry is the only marking in the first column. For any $i \geq 2$, let $A = p_x^{-1}(i)$, then the $X$ or $O$ at the $(p_x(A), p_y(A))$ entry is the only marking in the $i^{th}$ column, because $p_x$ is a bijection.
	For any $j \geq 1$, let $B = (p_y|_{\mathcal{S}_+(\mathcal{I})})^{-1}(j) \in \mathcal{S}_+(\mathcal{I})$, then the $X$ at the $(p_x(B), p_y(B))$ entry is the only $X$ in the $j^{th}$ row, because $p_y|_{\mathcal{S}_+(\mathcal{I})}$ is a bijection. If $j\leq n-1$, then $B \neq [0,1]$. By the definition of $p_y$, we have $p_y(\overline{B})=p_y(B)=j$ . Then the $O$ at the $(p_{x}(\overline{B}), p_y(\overline{B}))^{th}$ entry is the only $O$ in the $j^{th}$ row, because $p_y|_{\mathcal{S}_-(\mathcal{I})}$ is a bijection. At last, the $O$ at the $(1,n)$ entry is the only $O$ in the $n^{th}$ row.
\end{proof}

\begin{figure}
	\centering
	\includegraphics[width = \textwidth]{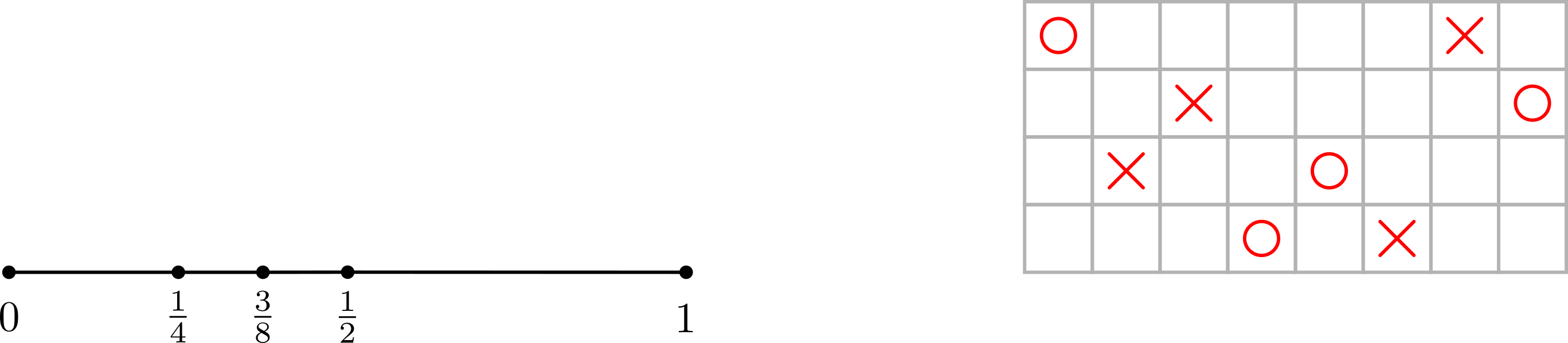}
	\put(-450, 15){$\mathcal{I}$}
	\put(-180, 50){$\mathbb{H}_{\mathcal{I}}$}
	\caption{An s.d. partition $\mathcal{I}$ and its corresponding half grid diagram $\mathbb{H}_{\mathcal{I}}$.}
	\label{fig: sdp to half grid diagram example}
\end{figure}

Now suppose that we have a pair of s.d. partitions $(\mathcal{I}, \mathcal{J})$ with same number of breakpoints, we want to know whether    $\mathbb{H}_{\mathcal{I}}$ and $\mathbb{H}_{\mathcal{J}}$ are compatible, so that we can put them together to obtain a complete grid diagram $(\mathbb{H}_{\mathcal{I}}, \mathbb{H}_{\mathcal{J}})$. It turns out that the answer is affirmative if $(\mathcal{I}, \mathcal{J})$ represents an element $g \in \vec{F}$. Recall that such a pair is called a compatible pair in Definition \ref{def: compatible pair of s.d. partitions}.

\begin{proposition} \label{prop: compatible half grid diagrams}
	If $(\mathcal{I}, \mathcal{J})$ is a compatible s.d. partition pair, then $(\mathbb{H}_{\mathcal{I}},\mathbb{H}_{\mathcal{J}})$ is a compatible half grid pair.
\end{proposition}

\begin{proof}
We reorder s.d. intervals in $\mathcal{S}(\mathcal{I})$ and $\mathcal{S}(\mathcal{J})$ as
\begin{align*}
	\mathcal{S}(\mathcal{I}) &= \{A_1 <_{mid} A_2 <_{mid} ... <_{mid} A_{2n-1} \} \\
	\mathcal{S}(\mathcal{J}) &= \{B_1 <_{mid} B_2 <_{mid} ... <_{mid} B_{2n-1} \}
\end{align*}

Reorder points in $E(\mathcal{I})$ and $E(\mathcal{J})$ as
\begin{align*}
	E(\mathcal{I}) &= \{a_1 < a_2 < ... < a_{2n-1} \} \\
	E(\mathcal{J}) &= \{b_1 < b_2 < ... < b_{2n-1} \}
\end{align*}

By Lemma \ref{lem: taking midpoint is a bijection from S to E}, we have a bijection $m|_{\mathcal{S}(\mathcal{I})}: \mathcal{S}(\mathcal{I}) \to E(\mathcal{I})$ sending an s.d. interval to its midpoint. Furthermore, $m|_{\mathcal{S}(\mathcal{I})}$ preserves the midpoint order on $\mathcal{S}(\mathcal{I})$. So $a_i$ must be the midpoint of $A_i$, and $b_i$ must be the midpoint of $B_i$.

Now suppose $(\mathcal{I}, \mathcal{J})$ is a compatible s.d. partition pair representing $g \in \vec{F}$. By Lemma \ref{lem: g restricted on E is bijection} and the fact that $g$ is increasing, $g|_{E(\mathcal{I})}: E(\mathcal{I}) \to E(\mathcal{J})$ is a bijection sending $a_i$ to $b_i$ (for an illustration, see Figure \ref{fig: shift example}). By Theorem \ref{thm: equivalent definition of oriented Thompson group}, we know that $a_i$ and $b_i$ have the same sign for any $i$.

Recall that we assigned signs on $E$ through the bijection $m: \mathcal{S} \to E$, so $A_i$ and $B_i$ also have the same sign for any $i$. By the construction of a half grid diagram, the $(i+1)^{th}$ column of $\mathbb{H}(\mathcal{I})$ contains a single marking $X$ if $A_i \in \mathcal{S}_+(\mathcal{I})$, $O$ if $A_i \in \mathcal{S}_-(\mathcal{I})$. It's similar for the $(i+1)^{th}$ column of $\mathbb{H}(\mathcal{J})$. Also notice that the first column of $\mathbb{H}(\mathcal{I})$ and the first column $\mathbb{H}(\mathcal{J})$ both contain an $O$, so $\mathbb{H}_{\mathcal{I}}$ and $\mathbb{H}_{\mathcal{J}}$ are compatible.
\end{proof}

\begin{figure}
	\centering
	\includegraphics[width = \textwidth]{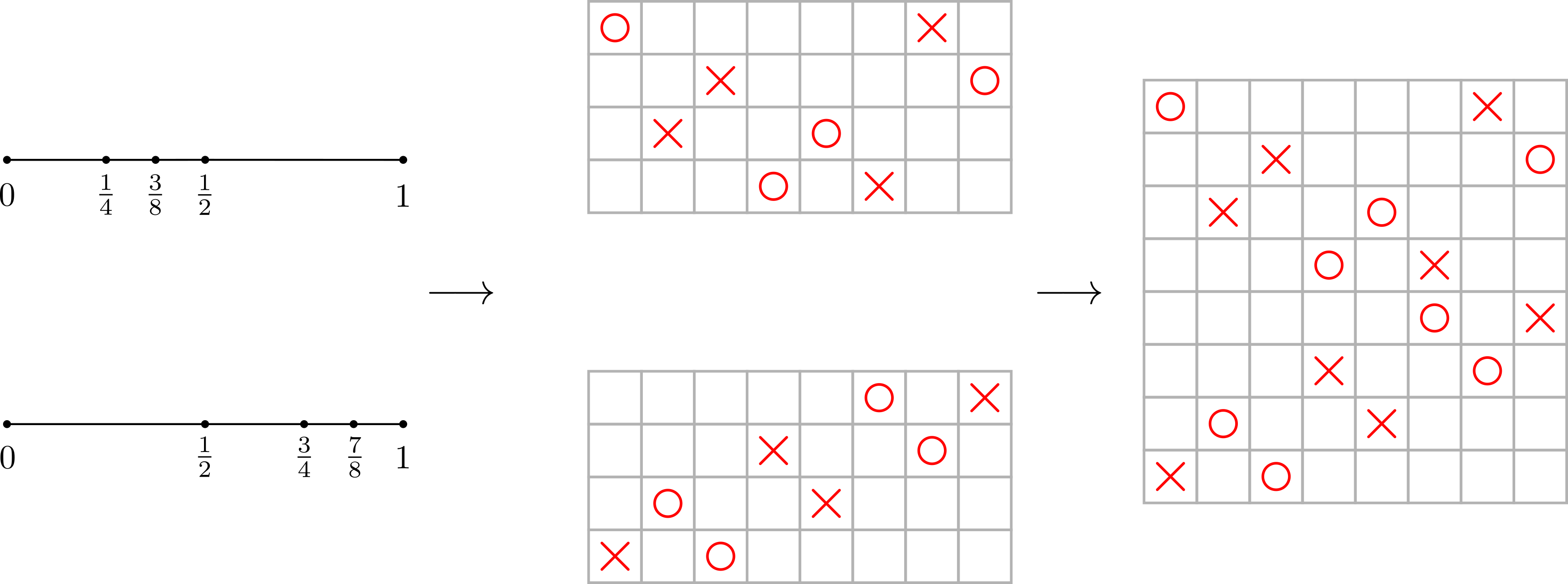}
	\put(-450, 110){$\mathcal{I}$}
	\put(-450, 40){$\mathcal{J}$}
	\put(-295, 120){$\mathbb{H}_{\mathcal{I}}$}
	\put(-302, 30){$-\overline{\mathbb{H}_{\mathcal{J}}}$}
	\put(-80, 8){$(\mathbb{H}_{\mathcal{I}}, \mathbb{H}_{\mathcal{J}})$}
	\caption{Construction of grid diagram $(\mathbb{H}_{\mathcal{I}}, \mathbb{H}_{\mathcal{J}})$ from a compatible s.d. partition pair $(\mathcal{I}, \mathcal{J})$.}
	\label{fig: grid construction example}
\end{figure}
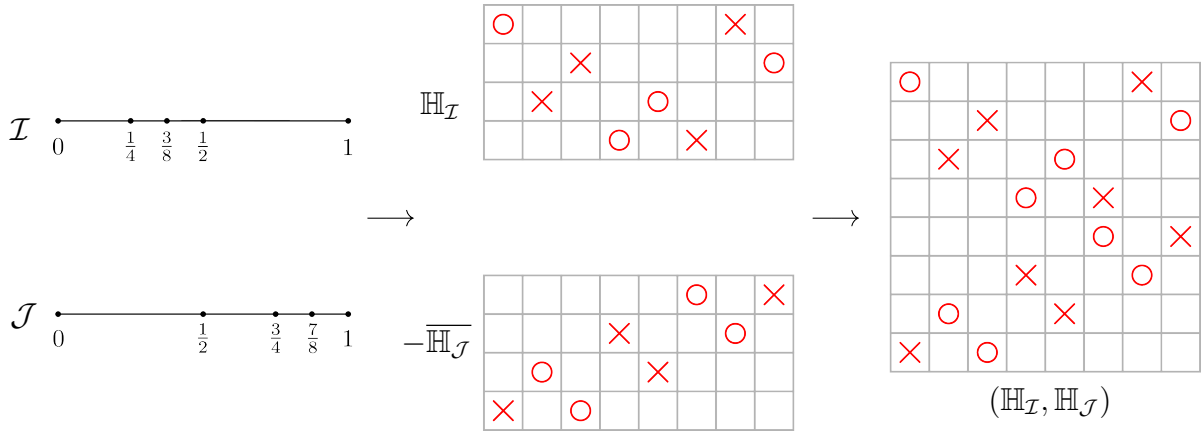

\begin{theorem} \label{thm: equivalent oriented links}
	Suppose that $(\mathcal{I}, \mathcal{J})$ is a compatible s.d. partition pair, then
	$$
	\mathcal{L}(\mathbb{H}_{\mathcal{I}}, \mathbb{H}_{\mathcal{J}})= L_{(\mathcal{I},\mathcal{J})}
	$$
	as oriented links. Especially if $(\mathcal{I}, \mathcal{J})$ is the reduced s.d. partition pair representing some $g \in \vec{F}$, then $\mathcal{L}(\mathbb{H}_{\mathcal{I}}, \mathbb{H}_{\mathcal{J}})= L_g$.
\end{theorem}

We will prove the above theorem in the next section. Assuming Theorem \ref{thm: equivalent oriented links}, our main theorem follows immediately.

\begin{proof} [Proof of Theorem \ref{thm: link is half grid presentable}]
	Given an oriented link $L$, we can find some $g \in \vec{F}$ such that $L_g = L$ by Theorem \ref{thm: oriented Thompson Alexander theorem}. Suppose that $(\mathcal{I}, \mathcal{J})$ is the reduced s.d. partition pair representing $g$, then we have $\mathcal{L}(\mathbb{H}_{\mathcal{I}}, \mathbb{H}_{\mathcal{J}})= L_g = L$ by Theorem \ref{thm: equivalent oriented links}. Thus, $(\mathbb{H}_{\mathcal{I}}, \mathbb{H}_{\mathcal{J}})$ is a half grid representative of $L$.
\end{proof}

\subsection{Equivalence to Jones' construction}
\label{sec: equivalence}

Given an s.d. $n$-partition $\mathcal{I}$, in Section \ref{sec: oriented thompson group} we have an oriented $(0,2n)$-tangle $\widehat{T_\mathcal{I}}$ coming from Jones' construction, and in Section \ref{sec: grid construction} we also have an oriented $(0,2n)$-tangle $\mathcal{T}(\mathbb{H}_{\mathcal{I}})$ coming from half grid diagram construction. We claim that these two oriented tangles are actually same.

\begin{proposition} {\label{prop: equivalent oriented tangles}}
	Let $\mathcal{I}$ be an s.d. partition, then $\widehat{T_\mathcal{I}} = \mathcal{T}(\mathbb{H}_{\mathcal{I}})$ as oriented $(0,2n)$-tangles.
\end{proposition}

\begin{proof}
	First we move the trivalent vertices of $F_{\mathcal{I}}$ vertically corresponding to their depth on the tree, so that vertex of smaller depth always has larger $y$-coordinate. If two vertices are of same depth, we let the vertex on the right has larger $y$-coordinate. See Figure \ref{fig: pre-isotoped tree} for an example of pre-isotoped tree $F_{\mathcal{I}}$.
	
	Recall that when we turn a tree $F_{\mathcal{I}}$ into a tangle $T_{\mathcal{I}}$, each edge $e$ of $F_{\mathcal{I}}$ becomes a Type A arc $a(e)$ of $T_\mathcal{I}$. If $e$ is not a lowermost edge, then we add vertical Type B arc $b(e)$ extending $a(e)$ to $\phi_{\mathcal{I}}(e)$. See Figure \ref{fig: tree to tangle} for an example, where Type A arcs are marked in blue and Type B arcs are marked in red. Now we keep Type B arcs unchanged and redraw Type A arcs as follows.
	
	If $e$ is a left edge, we redraw $a(e)$ as $\ulcorner$, and let $c(e)$ be the upper-left corner. If $e$ is a right edge, we redraw $a(e)$ as $\urcorner$, and let $c(e)$ be the upper-right corner. Then we consider the top edge $e_0$ attached to $(\frac{1}{2}, 1)$ as a right edge and put an upper-right corner $c(e_0)$ at $(\frac{1}{2},1)$. So the corresponding arc $\urcorner$ first goes upward to $(\frac{1}{2},1)$ then turn left. We let it stop at $(0,1)$ and put an upper-left corner $\ulcorner$ there to make it turn downwards until it hits $(0,0)$.

	After we redraw Type A arc $a(e)$, we label its corner $c(e)$ by $X$ if $e$ has sign $+$, by $O$ if $e$ has sign $-$. For example, the corner at $(\frac{1}{2},1)$ always has marking $X$. As the last step, we mark the corner at $(0,1)$ by $O$. See Figure \ref{fig: tree to tangle} and Figure \ref{fig: corner isotopy} for the described isotopy and markings.

	\begin{figure}
		\centering
    	\begin{subfigure}{0.4\textwidth}
        	\centering
        	\includegraphics[width=\textwidth]{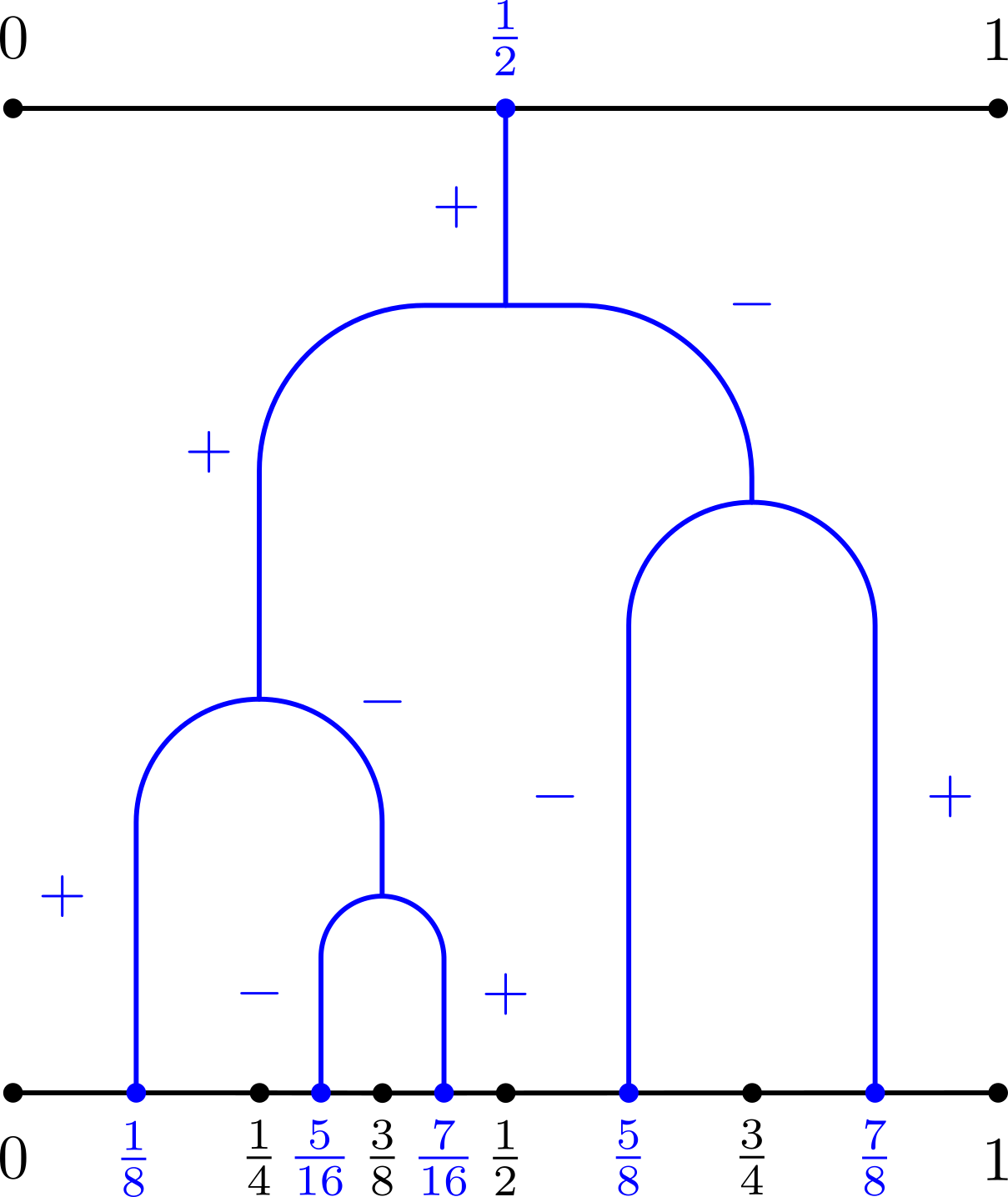}
			\put(-190, 120){$\color{blue} F_{\mathcal{I}}$}
        	\caption{}
        	\label{fig: pre-isotoped tree}
    	\end{subfigure}
		\hfill
		\begin{subfigure}{0.4\textwidth}
        	\centering
        	\includegraphics[width=\textwidth]{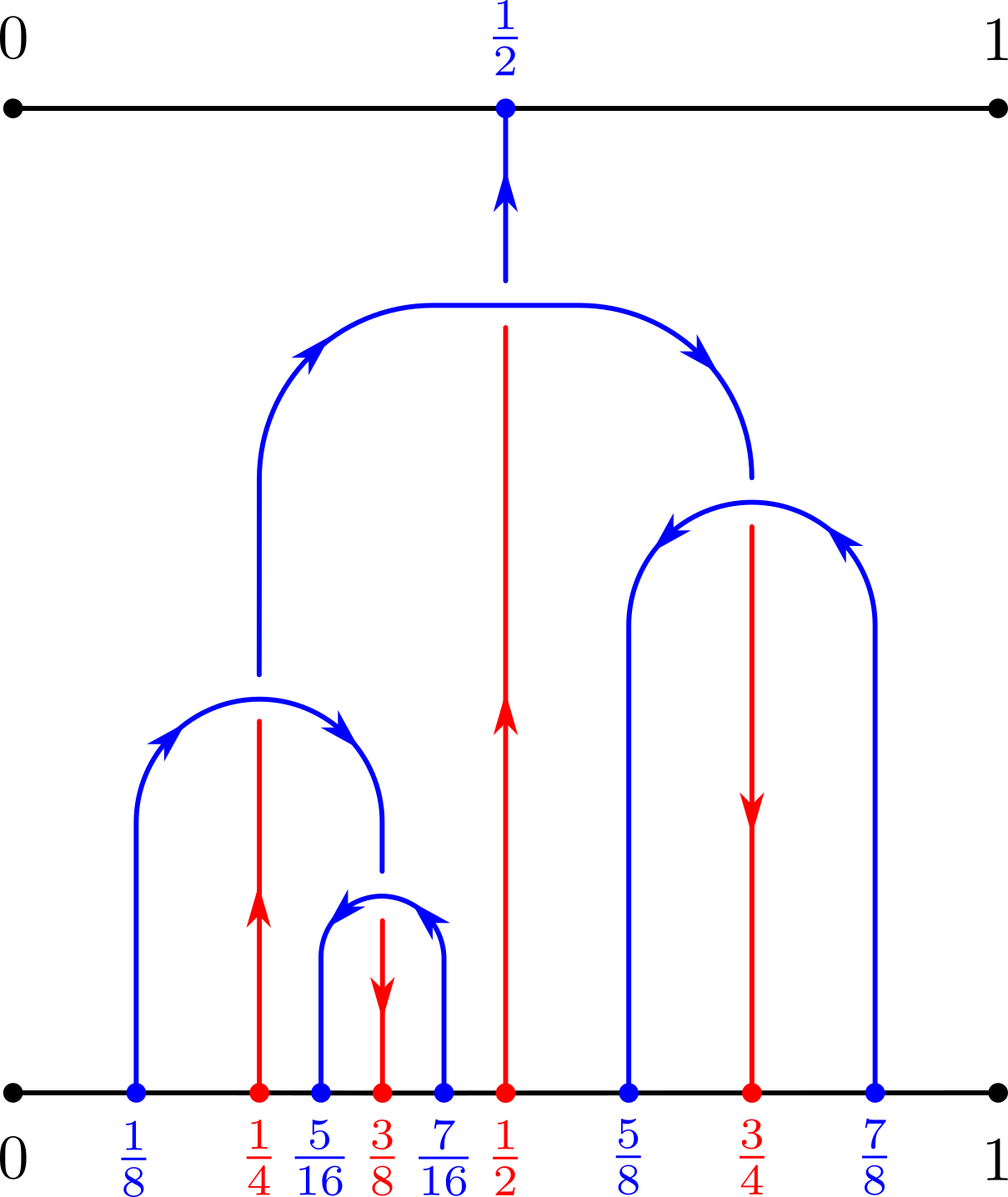}
			\put(-190, 120){$T_{\mathcal{I}}$}
        	\caption{}
        	\label{fig: tree to tangle}
    	\end{subfigure}
		\vspace{1cm}
		\vfill
		\begin{subfigure}{0.4\textwidth}
        	\centering
        	\includegraphics[width=\textwidth]{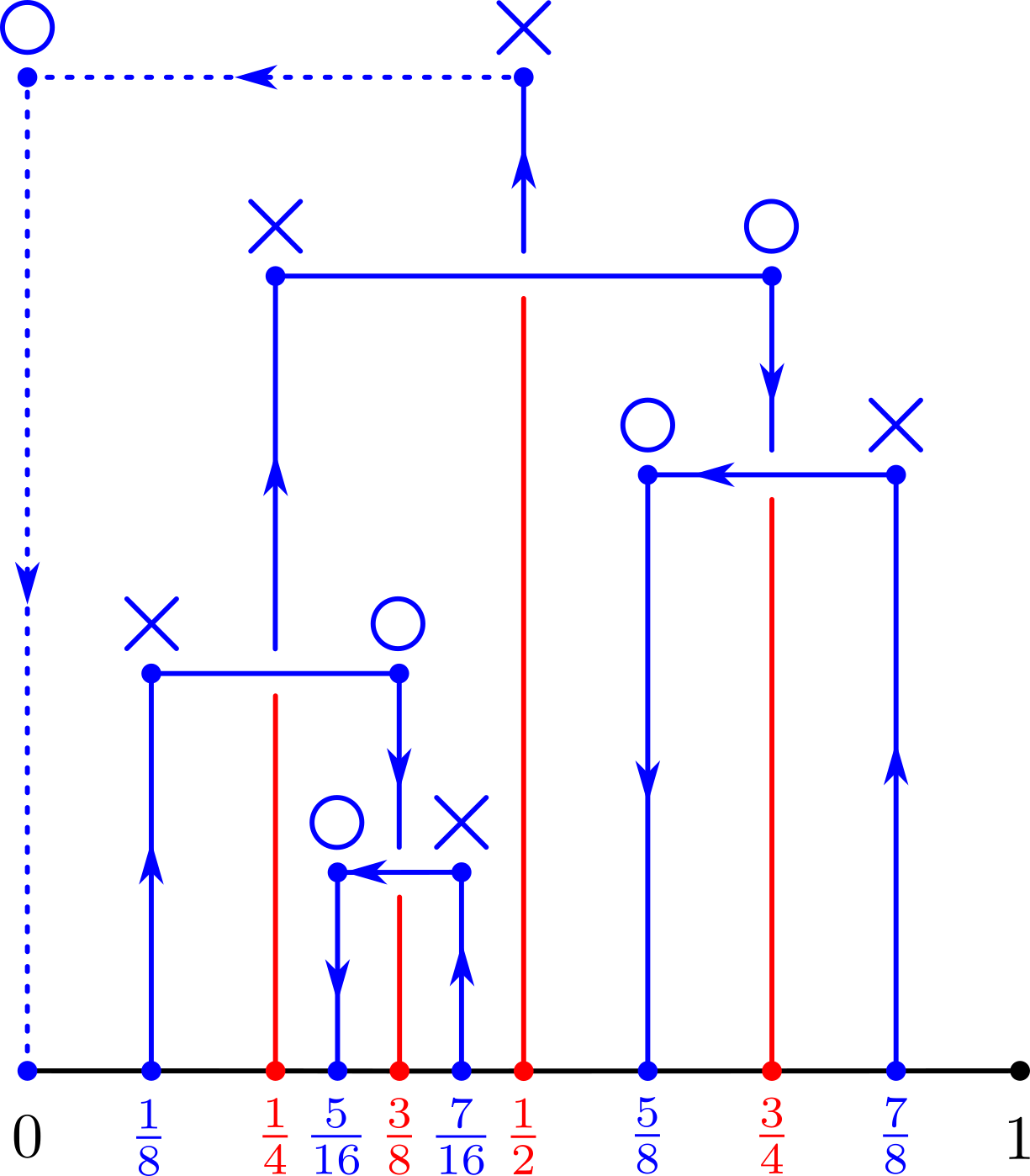}
			\put(-190, 120){$\widehat{T_{\mathcal{I}}}$}
        	\caption{}
        	\label{fig: corner isotopy}
    	\end{subfigure}
		\hfill
		\begin{subfigure}{0.4\textwidth}
        	\centering
        	\includegraphics[width=\textwidth]{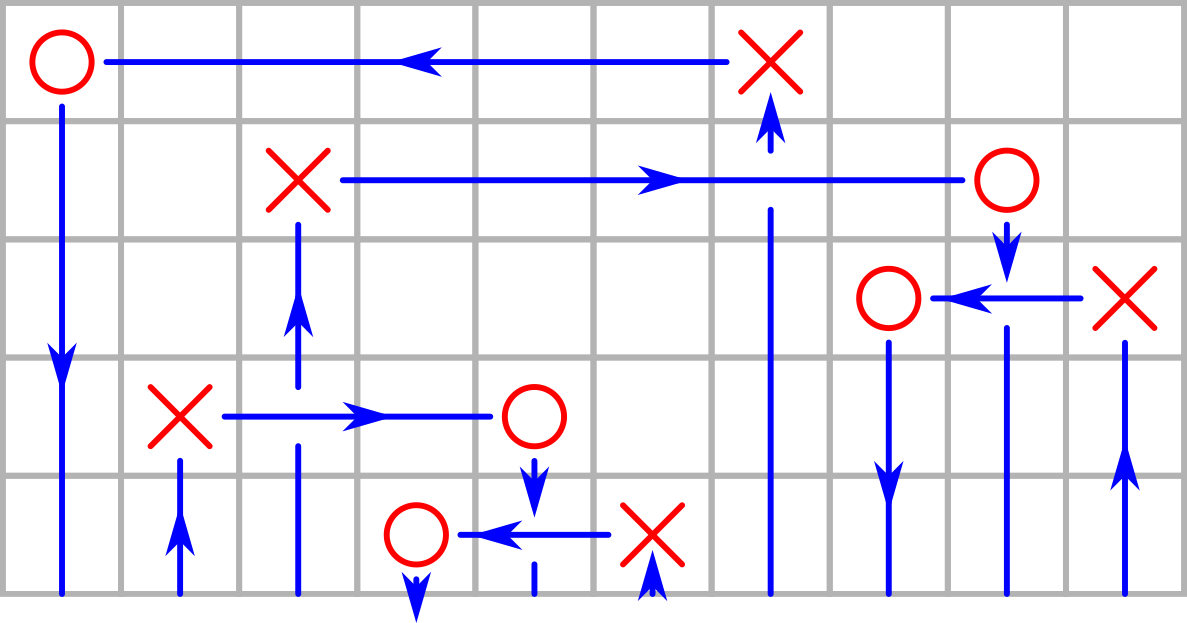}
			\put(-200, 50){$\mathbb{T}_{\mathcal{I}}$}
        	\caption{}
        	\label{fig: tangle to half grid}
    	\end{subfigure}
		\caption{(a) A pre-isotoped tree $F_{\mathcal{I}}$. (b) The oriented tangle $T_{\mathcal{I}}$ associated with the tree $F_{\mathcal{I}}$. Type A arcs are marked in blue, Type B arcs are marked in red. (c) Redraw each Type B arc of $T_{\mathcal{I}}$ as an L-shaped arc with corners marked by $X$ or $O$. If we extend $a(e_0)$ from $(\frac{1}{2}, 1)$ to $(0, 1)$, then to $(0, 0)$, we can obtain $\widehat{T_{\mathcal{I}}}$. (d) Rescale $\widehat{T_{\mathcal{I}}}$ into a tangle generated by a half grid diagram $\mathbb{T}_{\mathcal{I}}$.}
	\end{figure}

	Now, note that the $x$-coordinate of corner $c(e)$ is the midpoint of $\epsilon^{-1}(e)$, which is exactly $\phi(e)$. So for each point in $E(\mathcal{I})$, there is exactly one marking above it, and there is one marking $O$ at $(0,1)$ above $0$. Also, for each trivalent vertex, there is exactly one marking on its left and one marking on its right, corresponding to its left edge and right edge, respectively. These two markings must be different, because the left edge and right edge of a same vertex have different signs (see Figure \ref{fig: local tangle orientation in terms of signed edges}). At last, there is an $X$ at $(\frac{1}{2},1)$ and an $O$ at $(0,1)$ with $y$-coordinate $1$. Thus, if we rescale $[0,1] \times [0,1]$ into $[1,2n] \times [1,n]$ vertically and horizontally to make sure that each marking has integer coordinates and regard $\{1,...,2n\} \times \{1,...,n\}$ as a half grid, then we get a half grid diagram $\mathbb{T}_{\mathcal{I}}$. See Figure \ref{fig: corner isotopy} and Figure \ref{fig: tangle to half grid} for an example.

	The above transformation can also be expressed locally as Figure \ref{fig: tangle to half grid local transformation} shows, and capping off $T_{\mathcal{I}}$ is equivalent to adding corners at $(\frac{1}{2},1)$ and $(0,1)$. So we have $\mathcal{T}(\mathbb{T}_{\mathcal{I}})=\widehat{T_{\mathcal{I}}}$ 
	
	\begin{figure}
		\centering
		\includegraphics[width = \textwidth]{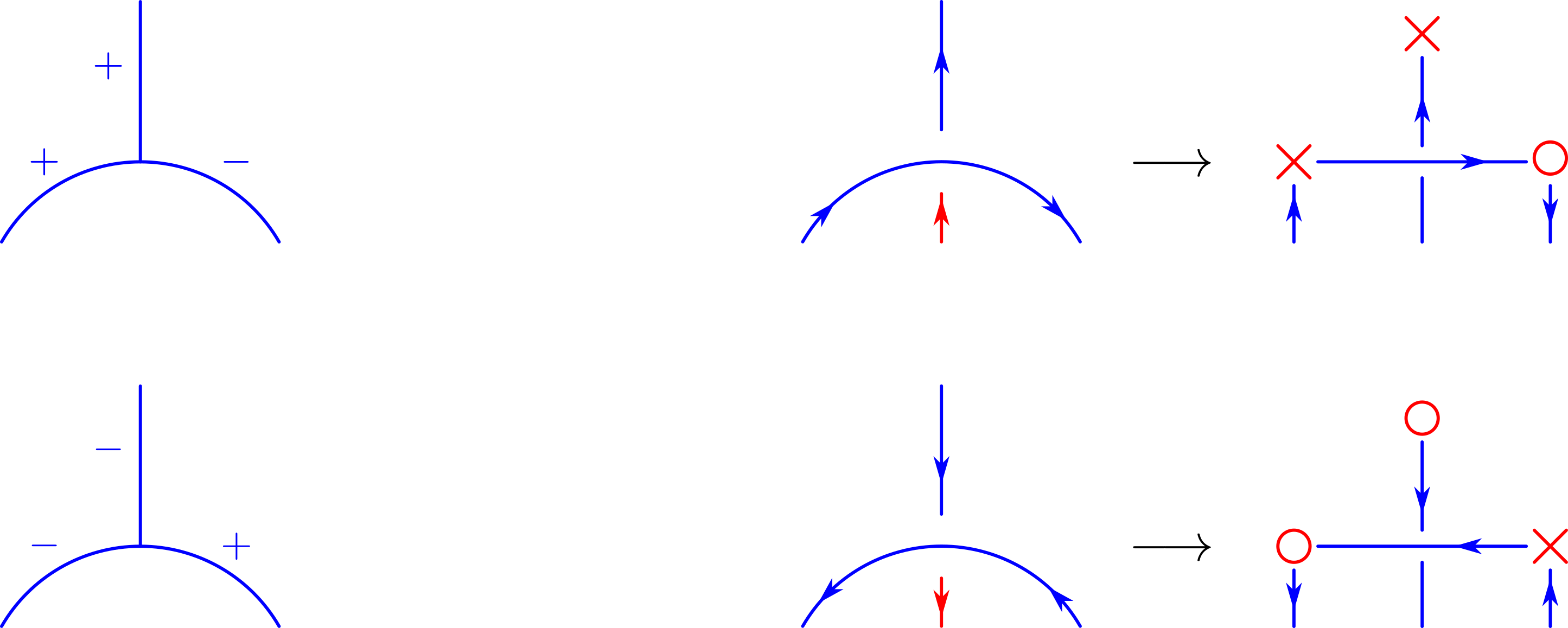}
		\put(-400, -30){$\color{blue} F_{\mathcal{I}}$}
		\put(-175, -30){$T_{\mathcal{I}}$}
		\put(-45, -30){$\mathbb{T}_{\mathcal{I}}$}
		\put(-385, 150){$\color{blue} e_1$}
		\put(-415, 110){$\color{blue} e_2$}
		\put(-375, 110){$\color{blue} e_3$}
		\put(-385, 45){$\color{blue} e_1$}
		\put(-415, 5){$\color{blue} e_2$}
		\put(-375, 5){$\color{blue} e_3$}
		\put(-165, 160){$a(e_1)$}
		\put(-210, 130){$a(e_2)$}
		\put(-160, 130){$a(e_3)$}
		\put(-200, 100){$b(e_1)$}
		\put(-165, 55){$a(e_1)$}
		\put(-210, 25){$a(e_2)$}
		\put(-160, 25){$a(e_3)$}
		\put(-200, -5){$b(e_1)$}
		\put(-50, 180){$c(e_1)$}
		\put(-85, 145){$c(e_2)$}
		\put(-15, 145){$c(e_3)$}
		\put(-50, 75){$c(e_1)$}
		\put(-85, 40){$c(e_2)$}
		\put(-15, 40){$c(e_3)$}
		\caption{Local transformation from $T_{\mathcal{I}}$ to $\mathbb{T}_{\mathcal{I}}$.}
		\label{fig: tangle to half grid local transformation}
	\end{figure}

	Now, given any $A \in \mathcal{S}(\mathcal{I})$, notice that the $x$ coordinate of $c(\epsilon_{\mathcal{I}}(A))$ is the position of midpoint $m(A) \in E(\mathcal{I})$ in the normal order (the first index is $2$, because of the $O$ marking at the $(1,n)^{th}$ entry). It is same as the position of $A \in \mathcal{S}(\mathcal{I})$ in the midpoint order, which is exactly $p_x(A)$. The $y$ coordinate of $c(\epsilon_{\mathcal{I}}(A))$ is the position of $\epsilon_{\mathcal{I}}(A)$'s upper endpoint in the reversed depth order. It is also the position of $A$ or $\overline{A}$ in the length order, which is $p_y(A)$.

	To show $\mathbb{T}_{\mathcal{I}} = \mathbb{H}_{\mathcal{I}}$, it remains to show that they have the same marking in the $(p_x(A), p_y(A))^{th}$ entry. For $\mathbb{T}_{\mathcal{I}}$, this marking is given by the sign of $\epsilon_{\mathcal{I}}(A)$. For $\mathbb{H}_{\mathcal{I}}$, this marking is given by the sign of $A$. By Lemma \ref{lem: interval to edge is sign preserving}, $\epsilon_{\mathcal{I}}(A)$ and $A$ have the same sign. So we have $\mathbb{T}_{\mathcal{I}} = \mathbb{H}_{\mathcal{I}}$.

	Thus, finally we have $\widehat{T_{\mathcal{I}}} = \mathcal{T}(\mathbb{T}_{\mathcal{I}}) = \mathcal{T}(\mathbb{H}_{\mathcal{I}})$.
\end{proof}

Now we put two half grid diagrams together to prove Theorem \ref{thm: equivalent oriented links}.

\begin{proof}[Proof of Theorem \ref{thm: equivalent oriented links}]
	Given $g \in \vec{F}$ and an s.d partition pair $(\mathcal{I}, \mathcal{J})$ representing $g$, a grid diagram $(\mathbb{H}_{\mathcal{I}}, \mathbb{H}_{\mathcal{J}})$ is well-defined by Proposition \ref{prop: compatible half grid diagrams}. It generates an oriented link $\mathcal{L}(\mathbb{H}_{\mathcal{I}}, \mathbb{H}_{\mathcal{J}}) = (-\overline{\mathcal{T}(\mathbb{H}_{\mathcal{J}})}) \mathcal{T}(\mathbb{H}_{\mathcal{I}})$. We've shown that $\mathcal{T}(\mathbb{H}_{\mathcal{I}}) = \widehat{T_{\mathcal{I}}}$ and $\mathcal{T}(\mathbb{H}_{\mathcal{J}}) = \widehat{T_{\mathcal{J}}}$, it follows that
	$$
	\mathcal{L}(\mathbb{H}_{\mathcal{I}}, \mathbb{H}_{\mathcal{J}}) = (-\overline{\mathcal{T}(\mathbb{H}_{\mathcal{J}})}) \mathcal{T}(\mathbb{H}_{\mathcal{I}}) = (-\overline{\widehat{T_{\mathcal{J}}}}) \widehat{T_{\mathcal{I}}}=\vec{L}_{(\mathcal{I},\mathcal{J})}
	$$
\end{proof}

\section{Applications}

\subsection{Legendrian, transverse knot and half grid diagram} \label{sec: Legendrian knot and half grid diagram}

We first give a very brief introduction to the contact 3-manifold and Legendrian, transverse knot (see \cite{MR2179261} for details).
\begin{definition}
	A contact 3 manifold $(Y,\xi)$ is a smooth 3 manifold $Y$ together with a 2-plane field distribution $\xi$ such that for any $1$-form $\alpha$ with $\ker(\alpha)=\xi$, we have $\alpha \wedge d\alpha > 0$.
\end{definition}

\begin{definition}
	A Legendrian link $L$ in $(Y,\xi)$ is a disjoint union of embedded $S^1$'s that are always tangent to $\xi$.
\end{definition} 

\begin{definition}
	A transverse link $L$ in $(Y,\xi)$ is a disjoint union of oriented embedded $S^1$'s that are always positively transverse to $\xi$.
\end{definition}

We will only consider the oriented Legendrian knot $Y$ in standard tight contact structure in $\mathbb{R}^3$ given by $\alpha=dz-ydx$. We can explicitly describe an oriented Legendrian link using the front projection $\Pi: (x,y,z) \rightarrow (x,z)$.

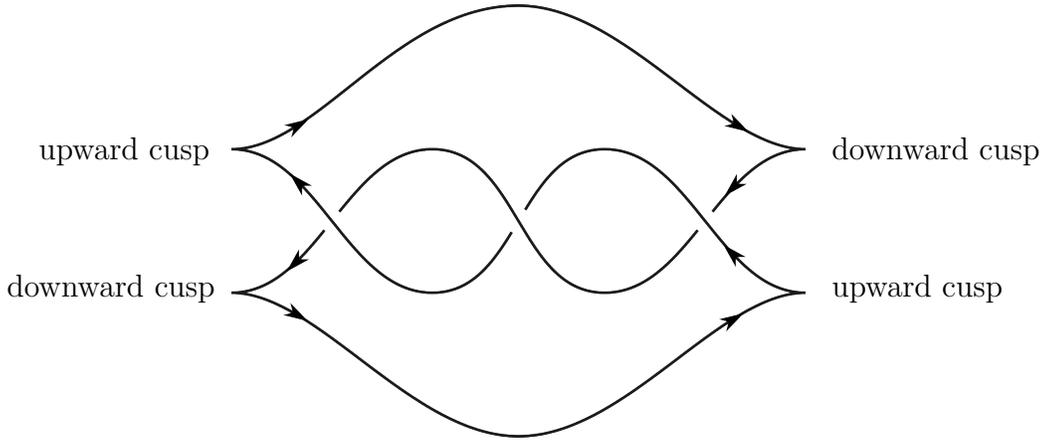
\begin{figure}
	\centering
	\includegraphics[width = 0.5\textwidth]{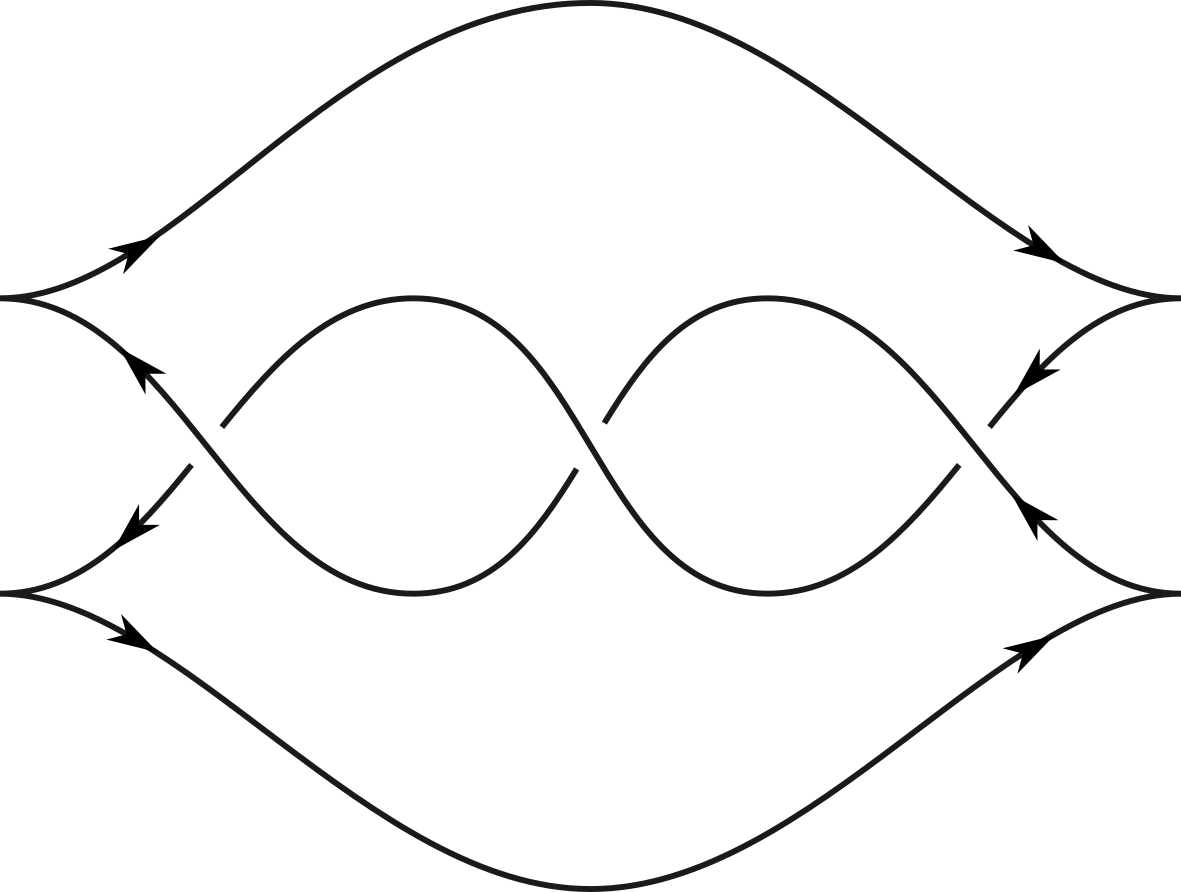}
	\put(-290, 105){upward cusp}
	\put(-302, 53){downward cusp}
	\put(10, 105){downward cusp}
	\put(10, 53){upward cusp}
	\caption{A front projection of right handed trefoil.}
	\label{fig: front projection example}
\end{figure}

Here are some basic properties of a front projection $\Pi(L)$ of a Legendrian link $L$ (see Figure \ref{fig: front projection example}).

\begin{enumerate}
	\item There is no vertical tangency in $\Pi(L)$. 
	\item The vertical tangencies change to generalized cusps, and the generalized cusps are the only non-smooth points.
	\item At each crossing the slope of the overcrossing is smaller than the undercrossing.
\end{enumerate}

Moreover, using the front projection, we can combinatorially define two classical invariants of Legendrian link. 

\begin{definition}
	Given a front projection of $\Pi(L)$ of a Legendrian link $L$,  the Thurston-Bennequin number of $\Pi(L)$ can be defined to be
	$$
	tb(\Pi(L))= \text{writhe}(\Pi(L))- \frac{1}{2} \left(\text{\# cusps in } \Pi(L) \right),
	$$
	and the rotation number of $\Pi(L)$ to be
	$$
	rot(\Pi(L))=\frac{1}{2} \left( \text{\# downward cusps in }\Pi(L) - \text{\# upward cusps in }\Pi(L) \right)
	$$
The max Thurston-Bennequin number $\overline{tb}(K)$ of a $K$ is the maximum value of $tb(\Pi(L))$ among front projections of Legendrian knot $L$ having the same knot type as $K$.
\end{definition}

We also have a classical invariant associate to a transverse link called the self-linking number $sl(L)$ (we omit the definition for $sl(L)$ and refer readers to see \cite{MR2179261} for details). For an oriented knot $K$ the max self-linking number $\overline{sl}(K)$ of $K$ is the maximum value of self-linking numbers among all transverse representatives of $K$. Moreover, We can always take the transverse push-off $T$ of a Legendrian link $L$ and their classical invariants are related by $sl(T)=tb(L)-rot(L)$. 

Now we are moving to the lemmas we need to prove Theorem \ref{thm: parity between link component and leaf number}. 

\begin{lemma} \label{lem: parity of Legendrian link}
	Let $\Pi(L)$ be any front projection of a Legendrian link $L$ with $m$ components, we have
	$$
	tb(\Pi(L))-rot(\Pi(L))= m \pmod{2}
	$$
\end{lemma}

\begin{proof}
	Suppose $L = K_1\cup K_2\cup ... \cup K_m$, where $K_i$ is Legendrian knot component of $L$. Let $\Pi(K_i)$ be the front projection of $K_i$ by deleting other components in $\Pi(L)$. Then we can decompose $tb(\Pi(L)$ as follows.
	\begin{align*}
		tb(\Pi(L)) =& \text{writhe}(\Pi(L))- \frac{1}{2} \text{(\# cusps in } \Pi(L) )\\   
		=& \sum_{i=1}^m \text{writhe}(\Pi(K_i))+ \left(\text{writhe}(\Pi(L)) -\sum_{i=1}^m \text{writhe}(\Pi(K_i)) \right) \\ 
		&- \sum_{i=1}^m \frac{1}{2} (\text{\# cusps in } \Pi(K_i)) \\
		=& \sum_{i=1}^m tb(\Pi(K_i)) + \left(\text{writhe}(\Pi(L)) -\sum_{i=1}^m \text{writhe}(\Pi(K_i )) \right),
	\end{align*}
	   
	For $rot(\Pi(L))$, we can directly decompose it as
	\begin{align*}
		rot(\Pi(L)) = \sum_{i=1}^m rot(\Pi(K_i))
	\end{align*}
	
	Then
	\begin{align} 
		&tb(\Pi(L)) - rot(\Pi(L)) \nonumber \\
		= &\sum_{i=1}^m (tb(\Pi(K_i))-rot(\Pi(K_i)))+ \left( \text{writhe}(\Pi(L)) -\sum_{i=1}^m \text{writhe}(\Pi(K_i )) \right) \label{eq: link tb - rot}
	\end{align}
	
	For any Legendrian knot $K$, $tb(\Pi(K))-rot(\Pi(K))$ is odd (see \cite[proof of Theorem 1]{MR2077671}), we have
	\begin{equation} \label{eq: knot tb - rot}
		\sum_{i=1}^n (tb(\Pi(K_i))-rot(\Pi(K_i)))=m \pmod{2}
	\end{equation}

	Moreover, observe that  $\left( \text{writhe}(\Pi(L)) -\sum_{i=1}^m   \text{writhe}(\Pi(K_i)) \right)$ is the writhe coming from the crossings that correspond to two different $K_i$'s, so we infer that
	\begin{equation} \label{eq: link writhe - knot writhe}
		\text{writhe}(\Pi(L)) -\sum_{i=1}^m \text{writhe}(\Pi(K_i )) = 0 \pmod{2}
	\end{equation}

	Combine Equations (\ref{eq: link tb - rot}), (\ref{eq: knot tb - rot}) and (\ref{eq: link writhe - knot writhe}) we conclude the lemma.
\end{proof}

Next, we recall that given a link $\mathcal{L}(\mathbb{G})$ for some grid diagram $\mathbb{G}$, if we rotate the diagram clockwise by $\frac{\pi}{4}$ radians, we get a front projection $\Pi({\mathcal{L}(\mathbb{G})})$ of a Legendrian representative of $\mathcal{L}(\mathbb{G})$. Notice that after rotation, the upper right corner $\urcorner$ and lower left corner $\llcorner$ of the original grid diagram become the cusps in the front projection (see Figure \ref{fig: grid to legendrian example}).

\begin{figure}
	\centering
	\includegraphics[width = \textwidth]{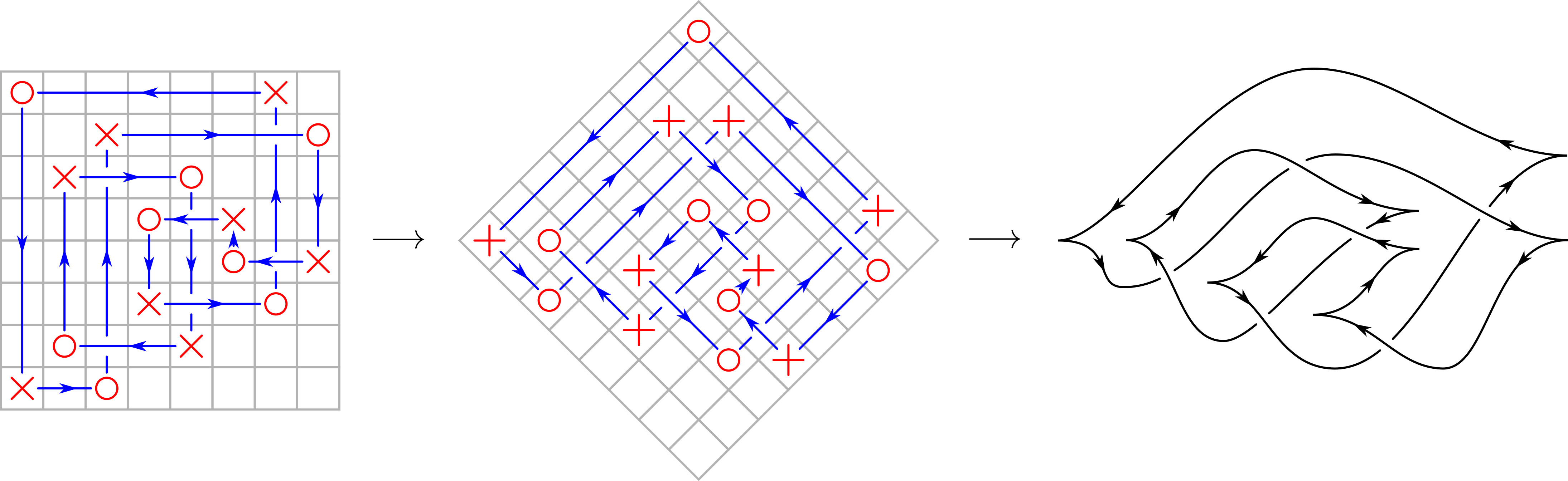}
	\put(-420, 0){$\mathbb{G} = (\mathbb{H}_+, \mathbb{H}_-)$}
	\put(-90, 0){$\Pi(\mathcal{L}(\mathbb{G}))$}
	\caption{We can rotate a grid diagram clockwise by $\frac{\pi}{4}$ radians to get a front projection of the Legendrian link.}
	\label{fig: grid to legendrian example}
\end{figure}

Thus, if we have a compatible half grid pair $(\mathbb{H}_+,\mathbb{H}_-)$, then we get a Legendrian link with front projection $\Pi(\mathcal{L}((\mathbb{H}_+,\mathbb{H}_-))$. The next Proposition gives us $tb$ and $rot$ of such Legendrian link. We first need a lemma.

\begin{lemma} \label{lem: right interval sign compatibility}
	Let $(\mathcal{I}, \mathcal{J})$ be a compatible s.d. $n$-partition pair. Then $\mathcal{S}_{R}(\mathcal{I})$ and $S_R(\mathcal{J})$ have the same number of $+$ signs and the same number of $-$ signs.
\end{lemma}

\begin{proof}
	By the Definition \ref{def: sign on s.d. interval}, if $A\in \mathcal{S}_R(\mathcal{I}) \backslash \{[0,1]\}$ satisfies $s(A)=+$ then there is a $B=A\cup \overline{A}$ such that $s(B)=-$, conversely if $B=A\cup \overline{A}$ for some $A\in \mathcal{S}(\mathcal{I}) \backslash \{[0,1]\} $ then there is an element $C=A$ or $C=\overline{A}$ such that $C\in \mathcal{S}_R(\mathcal{I}) \backslash \{[0,1]\} $ and $s(C)=+$. Together with the fact from Lemma \ref{lem: conjugate in SDI} that $A\cup \overline{A}\in S(\mathcal{I})$, we know that there is a bijection between
	\begin{align*}
		\{A \in \mathcal{S}_R(\mathcal{I}) \setminus [0,1] &\mid s(A) = +\} \\
	  	&\updownarrow \\
	  	\{B \in \mathcal{S}(\mathcal{I}) \mid B = A \cup \overline{A} \text{ for some }& A \in \mathcal{S}(\mathcal{I}) \setminus [0,1], \text{ and } s(B) = -\}
	\end{align*}

	Then by the bijection between $S(\mathcal{I})$ and $E(\mathcal{I})$ from the proof of Lemma \ref{lem: taking midpoint is a bijection from S to E} we know $A\cup \overline{A}\in S(\mathcal{I})$ correspond to the breakpoints in $E(\mathcal{I})$, so there is another bijection between
	\begin{align*}
		\{B \in \mathcal{S}(\mathcal{I}) \mid B = A \cup \overline{A} \text{ for some }& A \in \mathcal{S}(\mathcal{I}) \setminus [0,1], \text{ and } s(B) = -\} \\
		&\updownarrow \\
		\{ \text{breakpoints} &\text{ with $-$ sign in } E(\mathcal{I})\}
	\end{align*}
		
	Combine these two bijections we get 
	$$
	\text{\# $+$ signs on $\mathcal{S}_R(\mathcal{I})$} = \left(\text{\# $-$ signs on breakpoints in $E(\mathcal{I})$} \right) + 1
	$$
	where the added ``$1$" comes from the convention that  $[0,1] \in \mathcal{S}_R(\mathcal{I})$, and $s([0,1])=+$. The same argument works for $\mathcal{J}$, so we have
	$$
	\text{\# $-$ signs on $\mathcal{S}_R(\mathcal{J})$} = \left( \text{\# $-$ signs on breakpoints in $E(\mathcal{J})$} \right) + 1
	$$
		
	Now recall that by Definition \ref{def: compatible pair of s.d. partitions}, $(\mathcal{I},\mathcal{J})$ compatible means that the element $g$ represented by $(\mathcal{I},\mathcal{J})$ is in oriented Thompson group $\vec{F}$, and by Theorem \ref{thm: equivalent definition of oriented Thompson group} $g$ is sign preserving on $E$ and maps the breakpoints of $\mathcal{I}$ to breakpoints of $\mathcal{J}$ bijectively. Thus
	$$
	\text{\# $-$ signs on breakpoints in $E(\mathcal{I})$} =  \text{\# $-$ signs on breakpoints in $E(\mathcal{J})$}
	$$
	
	It tells us $\text{\# $+$ signs on $\mathcal{S}_R(\mathcal{I})$} = \text{\# $+$ signs on $\mathcal{S}_R(\mathcal{J})$}$. The same argument works for the case of $-$ sign.
\end{proof}

Using the above lemma we are able to find the classical invariants of 
the Legendrian $\Pi(\mathcal{L}(\mathbb{H}_{\mathcal{I}},\mathbb{H}_{J}))$.

\begin{proposition} \label{prop: tb and rot of grid (I,J)}
	Let $(\mathcal{I}, \mathcal{J})$ be a compatible s.d. $n$-partition pair. Then
	\begin{enumerate}
		\item $tb(\Pi(\mathcal{L}(\mathbb{H}_{\mathcal{I}},\mathbb{H}_{J})))=-n$
		\item $rot(\Pi(\mathcal{L}(\mathbb{H}_{\mathcal{I}},\mathbb{H}_{J})))=0$
	\end{enumerate}
\end{proposition}

\begin{proof}
	The first statement is straight forward. Notice that the cusps in $\Pi(\mathcal{L}(\mathbb{H}_{\mathcal{I}},\mathbb{H}_{J}))$ correspond to upper right corner and lower left corner in grid diagram $(\mathbb{H}_{\mathcal{I}},\mathbb{H}_{J})$. Since half grid diagram always has $n$ upper right corners and $n$ upper left corners when reverse it, there will be $n$ lower left corner which means there are total $2n$ cusps. Moreover, by Lemma \ref{lem: writhe of Thompson link is 0}, the writhe of $\mathcal{L}(\mathbb{H}_{\mathcal{I}},\mathbb{H}_{J})$ is $0$, so
	\begin{align*}
		tb(\Pi(\mathcal{L}(\mathbb{H}_{\mathcal{I}},\mathbb{H}_{\mathcal{J}}))) &= \text{writhe}(\Pi(\mathcal{L}(\mathbb{H}_{\mathcal{I}},\mathbb{H}_{\mathcal{J}}))) - \frac{1}{2}(\text{\# cusps in } \Pi(\mathcal{L}(\mathbb{H}_{\mathcal{I}},\mathbb{H}_{\mathcal{J}})))\\
		&= 0 - \frac{1}{2}(2n) \\
		&= -n
	\end{align*}

	Next, let's prove the second statement. By comparing the grid diagram $(\mathbb{H}_{\mathcal{I}},\mathbb{H}_{\mathcal{J}})$ and the front projection $\Pi(\mathcal{L}(\mathbb{H}_{\mathcal{I}},\mathbb{H}_{\mathcal{J}}))$ (Figure \ref{fig: grid to legendrian example} for example), we can observe that
	\begin{align*} 
		&\text{\# upward cusps} \\
		=& (\text{\# right } X \text{ on } \mathbb{H}_{\mathcal{I}})+(\text{\# left } O \text{ on } -\overline{\mathbb{H}_{\mathcal{J}}}) \\
		=& (\text{\# right } X \text{ on } \mathbb{H}_{\mathcal{I}})+(\text{\# left } X \text{ on } \mathbb{H}_{\mathcal{J}}) \\
		=& (\text{\# right } X \text{ on } \mathbb{H}_{\mathcal{I}})+(\text{\# right } O \text{ on } \mathbb{H}_{\mathcal{J}})
	\end{align*}

	Similarly, we have
	\begin{align*} 
		&\text{\# downward cusps} \\
		=& (\text{\# right } O \text{ on } \mathbb{H}_{\mathcal{I}})+(\text{\# left } X \text{ on } -\overline{\mathbb{H}_{\mathcal{J}}}) \\
		=& (\text{\# right } O \text{ on } \mathbb{H}_{\mathcal{I}})+(\text{\# left } O \text{ on } \mathbb{H}_{\mathcal{J}}) \\
		=& (\text{\# right } O \text{ on } \mathbb{H}_{\mathcal{I}})+(\text{\# right } X \text{ on } \mathbb{H}_{\mathcal{J}})
	\end{align*}
	
	The correspondence between signed s.d. intervals and markings in half grid diagram tells us that
	$$
	\text{\# right $X$ (resp. $O$) on $\mathbb{H}_{\mathcal{I}}$} = \text{\# $+$ (resp.$-$) signs on $\mathcal{S}_R(\mathcal{I})$}
	$$
	
	Then Lemma \ref{lem: right interval sign compatibility} implies
	$$
	\text{\# right $X$ (resp. $O$) on $\mathbb{H}_{\mathcal{I}}$} = \text{\# left $X$ (resp. $O$) on $\mathbb{H}_{\mathcal{J}}$}
	$$
	
	Finally we have
	\begin{align*}
		&rot(\Pi(\mathcal{L}(\mathbb{H}_{\mathcal{I}},\mathbb{H}_{\mathcal{J}}))) \\
		= &\frac{1}{2}( \text{\# downward cusps } - \text{\# upward cusps} ) \\
		= &( (\text{\# right } O \text{ on } \mathbb{H}_{\mathcal{I}})+(\text{\# right } X \text{ on } \mathbb{H}_{\mathcal{J}}) ) \\
		&- ((\text{\# right } X \text{ on } \mathbb{H}_{\mathcal{I}})+(\text{\# right } O \text{ on } \mathbb{H}_{\mathcal{J}})) \\
		= &((\text{\# right } O \text{ on } \mathbb{H}_{\mathcal{I}}) - (\text{\# right } O \text{ on } \mathbb{H}_{\mathcal{J}})) \\
		&+ ((\text{\# right } X \text{ on } \mathbb{H}_{\mathcal{J}}) - (\text{\# right } X \text{ on } \mathbb{H}_{\mathcal{I}})) \\
		= &0
	\end{align*}
\end{proof}

Now Theorem \ref{thm: parity between link component and leaf number} follows immediately.

\begin{proof} [Proof of Theorem \ref{thm: parity between link component and leaf number}]
	Let $(\mathcal{I},\mathcal{J})$ be a compatible pair of s.d. $n$-partitions. By Proposition \ref{prop: tb and rot of grid (I,J)} we know $tb(\Pi(\mathcal{L}(\mathbb{H}_{\mathcal{I}},\mathbb{H}_{J})))=-n$ and $rot(\Pi(\mathcal{L}(\mathbb{H}_{\mathcal{I}},\mathbb{H}_{J})))=0$. Combine with Lemma \ref{lem: parity of Legendrian link} we have $-n-0=|\mathcal{L}(\mathbb{H}_{\mathcal{I}},\mathbb{H}_{J})| \pmod{2}.$ Finally using Theorem \ref{thm: equivalent oriented links} we conclude $|\vec{L}_{(\mathcal{I},\mathcal{J})}|=n \pmod{2}).$

	Now suppose that $(\mathcal{I}, \mathcal{J})$ is the reduced pair of s.d. $n$-partitions representing $g \in \vec{F}$, then $|\vec{L}_g| = |\vec{L}_{(\mathcal{I},\mathcal{J})}|=n \pmod{2}$, where $n$ is also the leaf number of $g$.
\end{proof}

The first statement of Proposition \ref{prop: tb and rot of grid (I,J)} also implies Theorem \ref{thm: inequality between Thompson index and max tb number}.

\begin{proof} [Proof of Theorem \ref{thm: inequality between Thompson index and max tb number}]
	We will first directly proof the general inequality (\ref{eq: generalized inequality}) for the $\overline{tb}$, and the inequality (\ref{eq: short inequality}) follows immediately. We will show that for any oriented knot $K$, both of $ind_{\vec{F}}(K)$ and $ind_{\vec{F}}(-K)$ are less than or equal to both of $\overline{tb}(K)$ and $\overline{tb}(\overline{K})$.

	By the definition of $ind_{\vec{F}}(K)$, we can pick a compatible s.d. $n$-partition pair $(\mathcal{I},\mathcal{J})$ such that $n = \text{ind}_{\vec{F}}(K)$ and $L_{(\mathcal{I}, \mathcal{J})} = K$. By Theorem \ref{thm: equivalent oriented links}, we have $L_{(\mathcal{I}, \mathcal{J})}= \mathcal{L}(\mathbb{H}_{\mathcal{I}},\mathbb{H}_{J})$, so $\Pi(\mathcal{L}(\mathbb{H}_{\mathcal{I}},\mathbb{H}_{J}))$ is a Legendrian representative of $K$. By Proposition \ref{prop: tb and rot of grid (I,J)}, we have  $tb(\Pi(\mathcal{L}(\mathbb{H}_{\mathcal{I}},\mathbb{H}_{J}))) = -n$. Therefore,
	$$
	-ind_{\vec{F}}(K) = -n = tb(\Pi(\mathcal{L}(\mathbb{H}_{\mathcal{I}},\mathbb{H}_{J}))) \leq \overline{tb}(K)
	$$

	By Theorem \ref{thm: equivalent oriented links} and Remark \ref{rmk: reverse orientation of mirror}, we have $L_{(\mathcal{J}, \mathcal{I})}=- \overline{L_{(\mathcal{I}, \mathcal{J})}}$, so $\Pi(\mathcal{L}(\mathbb{H}_{\mathcal{J}},\mathbb{H}_{\mathcal{I}}))$ is a Legendrian representative of $- \overline{K}$ with $tb(\Pi(\mathcal{L}(\mathbb{H}_{\mathcal{J}},\mathbb{H}_{\mathcal{I}}))) = -n$. Moreover, since $\overline{tb}$ is invariant under reversing orientation, we have
	$$
	-ind_{\vec{F}}(K) = -n = tb(\Pi(\mathcal{L}(\mathbb{H}_{\mathcal{J}},\mathbb{H}_{\mathcal{I}}))) \leq \overline{tb}(-\overline{K}) = \overline{tb}(\overline{K}).
	$$

	At the end we replace $K$ by $-K$ to get
	\begin{align*}
		&-ind_{\vec{F}}(-K) \leq \overline{tb}(-K) = \overline{tb}(K)\\
		&-ind_{\vec{F}}(-K) \leq \overline{tb}(-\overline{K}) = \overline{tb}(\overline{K}).
	\end{align*}

 Now to prove the situation for the maximal self-linking number, we just recall from Proposition \ref{prop: tb and rot of grid (I,J)} that the special Legendrian representative $\Pi(\mathcal{L}(\mathbb{H}_{\mathcal{J}},\mathbb{H}_{\mathcal{I}}))$ has rotation $0$, thus the transverse push-off of $\Pi(\mathcal{L}(\mathbb{H}_{\mathcal{J}},\mathbb{H}_{\mathcal{I}}))$ has self-linking number equal to $tb(\Pi(\mathcal{L}(\mathbb{H}_{\mathcal{J}},\mathbb{H}_{\mathcal{I}})))-rot(\Pi(\mathcal{L}(\mathbb{H}_{\mathcal{J}},\mathbb{H}_{\mathcal{I}})))=-n$. The rest follows.
\end{proof}

\subsection{Half grid diagram and topological invariants}

\subsubsection{Symmetric representation of half grid diagram, unoriented grid diagram, and Link group}
\label{sec: sym rep and unoriented grid}

In this section we will first discuss the relationship between half grid diagrams and elements in symmetric groups and prove Theorem \ref{thm: half grid and element of symmetric group}. Then we will introduce the notion of unoriented grid diagrams and describe how to combine any two half grid diagrams into an unoriented grid diagram. We will use the fact that every (unoriented) link is half presentable to prove Theorem \ref{thm: half grid presentation of link group}. After that we will see how to apply the unoriented grid construction from $g\in F$ instead of $\vec{F}$, and apply it to prove Theorem \ref{thm: grid number inequality}.

Given an $n \times 2n$ half grid diagram $\mathbb{H}$, we can view $\mathbb{H}$ as two functions
\begin{align*}
	&X_{\mathbb{H}}:\{1,2,...,n\}\rightarrow \{1,2,...,2n\} \\
	&O_{\mathbb{H}}:\{1,2,...,n\}\rightarrow \{1,2,...,2n\}
\end{align*}
such that $(X_{\mathbb{H}}(i),i)$ and $(O_{\mathbb{H}}(j),j)$ are the coordinates of $X$ and $O$ respectively. Notice that the conditions in Definition \ref{def: half grid diagram} are equivalent to the requirement that $X_{\mathbb{H}}$ and $O_{\mathbb{H}}$ are both injective and their images are disjoint (so their union is the whole set $\{1,2,...,2n\}$).

\begin{proof} [Proof of Theorem \ref{thm: half grid and element of symmetric group}]
	Let's first start with an $n \times 2n$ half grid diagram $\mathbb{H}$, we have two functions $X_{\mathbb{H}}$ and $O_{\mathbb{H}}$ as described above. We define an element $\sigma_{\mathbb{H}}\in \text{Sym}(2n)$ to be
	$$
	\sigma_{\mathbb{H}} =
	\begin{pmatrix}
		1 & 2 & 3 & 4 & ... & 2n - 1 & 2n\\
		X_{\mathbb{H}}(1) & O_{\mathbb{H}}(1) & X_{\mathbb{H}}(2) & O_{\mathbb{H}}(2) & ... & X_{\mathbb{H}}(n) & O_{\mathbb{H}}(n)
	\end{pmatrix}
	$$

	Since $X_{\mathbb{H}}$ and $O_{\mathbb{H}}$ are both injective and the disjoint union of their images is $\{1, 2, ..., 2n\}$, $\sigma_{\mathbb{H}}$ is a well-defined element in $\text{Sym}(2n)$. If $\mathbb{H}_1$ and $\mathbb{H}_2$ are different half grid diagrams, then $(X_{\mathbb{H}_1}, O_{\mathbb{H}_1}) \neq (X_{\mathbb{H}_2}, O_{\mathbb{H}_2})$, which implies $\sigma_{\mathbb{H}_1} \neq \sigma_{\mathbb{H}_2}$.

	It remains to show any element $\sigma \in \text{Sym}(2n)$ can be obtained as $\sigma_{\mathbb{H}}$ for some half grid diagram $\mathbb{H}$. Define $X: \{1,2,...,n\} \to \{1,2,...,2n\}$ sending $i$ to $\sigma(2i - 1)$ and $O:\{1,2,...,n\}\rightarrow \{1,2,...,2n\}$ sending $j$ to $2j$. Then both $X$ and $O$ are injective and the disjoint union of their images is $\{1, 2, ..., 2n\}$, so they describe a half grid diagram $\mathbb{H}$. At last, we have $\sigma_{\mathbb{H}} = \begin{pmatrix} 1 & 2 & 3 & 4 & ... & 2n - 1 & 2n \\ \sigma(1) & \sigma(2) & \sigma(3) & \sigma(4) & ... & \sigma(2n - 1) & \sigma(2n) \end{pmatrix} = \sigma$.
\end{proof}

\begin{figure}
	\centering
	\includegraphics[width = \textwidth]{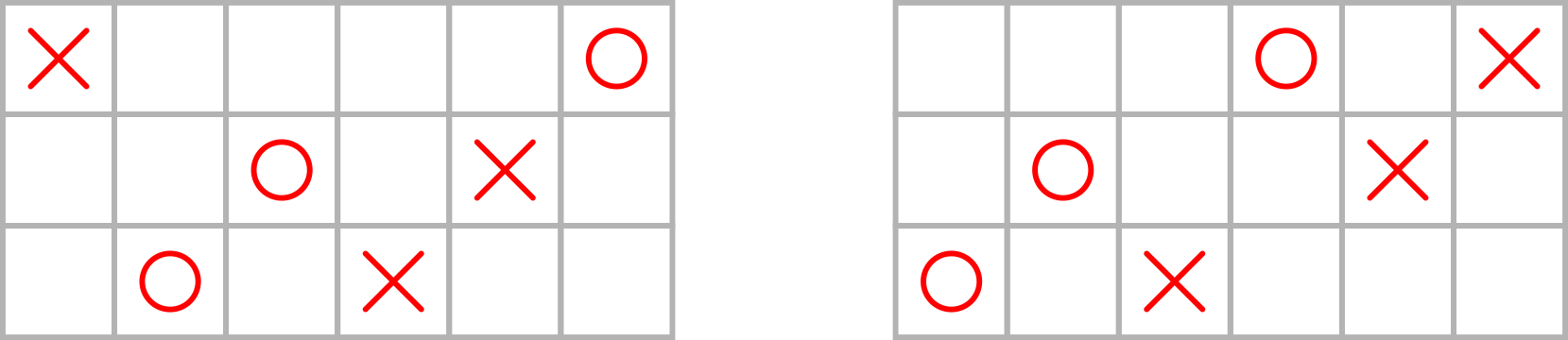}
	\put(-417, 100){$1$}
	\put(-387, 100){$2$}
	\put(-356, 100){$3$}
	\put(-326, 100){$4$}
	\put(-295, 100){$5$}
	\put(-265, 100){$6$}
	\put(-171, 100){$1$}
	\put(-141, 100){$2$}
	\put(-111, 100){$3$}
	\put(-81, 100){$4$}
	\put(-50, 100){$5$}
	\put(-20, 100){$6$}
	\put(-428, -15){\small$\sigma_+(5)$}
	\put(-397, -15){\small$\sigma_+(2)$}
	\put(-366, -15){\small$\sigma_+(4)$}
	\put(-336, -15){\small$\sigma_+(1)$}
	\put(-305, -15){\small$\sigma_+(3)$}
	\put(-274, -15){\small$\sigma_+(6)$}
	\put(-181, -15){\small$\sigma_-(2)$}
	\put(-151, -15){\small$\sigma_-(4)$}
	\put(-121, -15){\small$\sigma_-(1)$}
	\put(-91, -15){\small$\sigma_-(6)$}
	\put(-60, -15){\small$\sigma_-(3)$}
	\put(-29, -15){\small$\sigma_-(5)$}
	\put(-455, 50){$\mathbb{H}_+$}
	\put(-210, 50){$\mathbb{H}_-$}
	\caption{Example of $\sigma_+$ and $\sigma_-$ corresponding to half grid diagram $\mathbb{H}_+$ and $\mathbb{H}_-$ respectively. Note that they are not compatible.}
	\label{fig: sym representation example}
\end{figure}
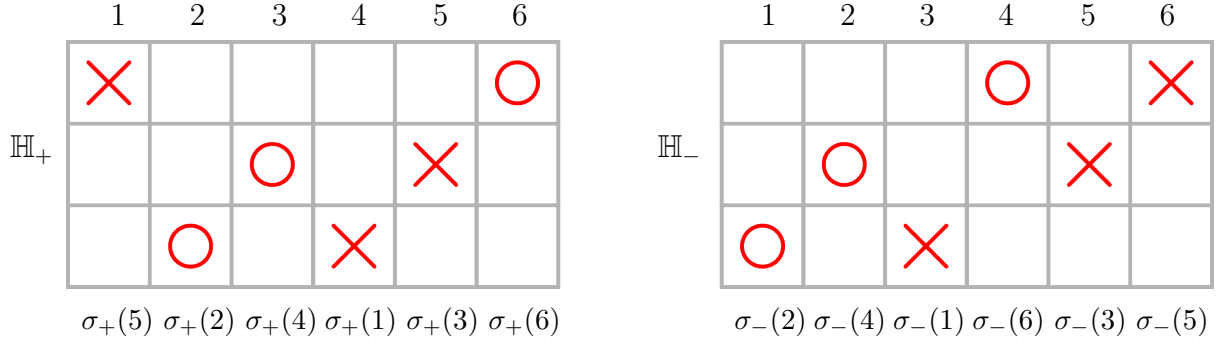

\begin{example} \label{ex: sym representation}
	Let's observe the pair of half grid diagrams $(\mathbb{H}_+, \mathbb{H}_-)$ in Figure \ref{fig: sym representation example}. First, we try to write down their symmetric representations $\sigma_{\mathbb{H}_+}$ and $\sigma_{\mathbb{H}_-}$ (simply denoted as $\sigma_+$ and $\sigma_-$). Take $\sigma_+$ as an example, for odd inputs, $\sigma_+(2i - 1) = X_{\mathbb{H}}(i)$ is the column index of $X$ in the $i^{th}$ row, so $\sigma_+(1) = 4, \sigma_+(3) = 5, \sigma_+(5) = 1$. For even inputs, $\sigma_+(2i) = O_{\mathbb{H}}(i)$ is the column index of $O$ in the $i^{th}$ row, so $\sigma_+(2) = 2, \sigma_+(4) = 3, \sigma_+(6) = 6$. We can similarly write down $\sigma_-$. The result is
	$$
	\sigma_{+} =
	\begin{pmatrix}
		1 & 2 & 3 & 4 & 5 & 6 \\
		4 & 2 & 5 & 3 & 1 & 6
	\end{pmatrix}
	, \quad
	\sigma_{-} = 
	\begin{pmatrix}
		1 & 2 & 3 & 4 & 5 & 6 \\
		3 & 1 & 5 & 2 & 6 & 4
	\end{pmatrix}
	$$

	Notice that $(\mathbb{H}_+, \mathbb{H}_-)$ is not a compatible pair, because the first column of $\mathbb{H}_+$ contains an $X$, but the first column of $\mathbb{H}_-$ contains an $O$. In the language of representation representation, $\sigma_+^{-1}(1) = 5$ is odd, but $\sigma_-^{-1}(1) = 2$ is even.

	To make sure that $(\mathbb{H}_+, \mathbb{H}_-)$ is a compatible pair, we just need to require that $\sigma_+^{-1}(j)$ and $\sigma_-^{-1}(j)$ have the same parity for any $j$. An equivalent condition is that $\{\sigma_{\mathbb{H}_+}(i) \mid i = 2, 4, ..., 2n\} = \{\sigma_{\mathbb{H}_-}(i) \mid i = 2, 4, 6, ..., 2n\}$ as two sets. Or, $\{\sigma_{\mathbb{H}_+}(i) \mid i = 1, 3, ..., 2n - 1\} = \{\sigma_{\mathbb{H}_-}(i) \mid i = 1, 3, ..., 2n - 1\}$ as two sets.
\end{example}

Next, we will study the link group given by a pair of half grid diagrams. Since the link group is independent of the orientation of the link, it is natural to introduce an unoriented grid diagram and an associated unoriented link.

\begin{definition}
	An unoriented grid diagram $\mathbb{G}$ is an $n\times n$ grid of squares, each of which is either empty or contains a marking $\bigotimes$. We require that in each row and column there are exactly two nonempty squares.
\end{definition}

To distinguish (oriented) grid diagrams and unoriented grid diagrams, in the rest of this section we will denote oriented grid diagram as $\vec{\mathbb{G}}$.

Any unoriented grid diagram $\mathbb{G}$ has an associated unoriented link $\mathcal{L}(\mathbb{G})$ as follows.

\begin{definition}
	Let $\mathbb{G}$ be an $n \times n$  unoriented grid diagram. The unoriented link $\mathcal{L}(\mathbb{G})$ associated to $\mathbb{G}$ is obtained by connecting two $\bigotimes$'s in each row and connecting two $\bigotimes$'s in each column, such that horizontal segments are always above vertical segments.
\end{definition}

\begin{remark}
	Given an oriented grid diagram $\vec{\mathbb{G}}$, we can replace every $X$ and $O$ with $\bigotimes$, to obtain an unoriented grid diagram $\mathbb{G}$. Conversely, given an unoriented diagram diagram $\mathbb{G}$ and associated unoriented link $\mathcal{L}(\mathbb{G})$, we can always manually replace each $\bigotimes$ with an $X$ or an $O$, so that the resulting diagram can give any orientation on $\mathcal{L}(\mathbb{G})$. The procedure is shown as follows.

	Start with an arbitrary $\bigotimes$ in an arbitrary link component of $\mathcal{L}(\mathbb{G})$, replace it with $O$, then it will automatically specify an $X$ and an $O$ on any other $\bigotimes$ on this component, to make sure that each row and each column has exactly one $O$ and exactly one $X$. We repeat this procedure on all link components of $\mathcal{L}(\mathbb{G})$ to obtain an oriented grid diagram $\vec{\mathbb{G}}$, which gives an orientation on $\mathcal{L}(\mathbb{G})$. To get other orientations, we just need to change the orientations of some link components by flipping all of $X$'s and $O$'s on those components.
\end{remark}

A nice thing of unoriented grid diagram is that even if we have two non-compatible half grid diagrams $\mathbb{H}_+$ and $\mathbb{H}_-$, we still have an unoriented grid diagram $(\mathbb{H}_+,\mathbb{H}_-)$ by replacing every $X$ and $O$ in $\mathbb{H}_+$ and $\overline{\mathbb{H}_-}$ by $\bigotimes$, then stacking $\mathbb{H}_+$ over $\overline{\mathbb{H}_-}$ (see Figure \ref{fig: non compatible pair to unoriented grid example} for an example). If we regard $\mathcal{T}(\mathbb{H}_{\mathcal{I}})$ and $\mathcal{T}(\mathbb{H}_{\mathcal{J}})$ as unoriented tangles, then we have $\mathcal{L}(\mathbb{H}_{\mathcal{I}}, \mathbb{H}_{\mathcal{J}}) = \overline{\mathcal{T}(\mathbb{H}_{\mathcal{J}})} \mathcal{T}(\mathbb{H}_{\mathcal{I}})$. Furthermore, unoriented links $\mathcal{L}(\mathbb{H}_+,\mathbb{H}_-)$ and $\mathcal{L}(\mathbb{H}_-,\mathbb{H}_+)$ are mirrors of each other.

\begin{figure}
	\centering
	\includegraphics[width = \textwidth]{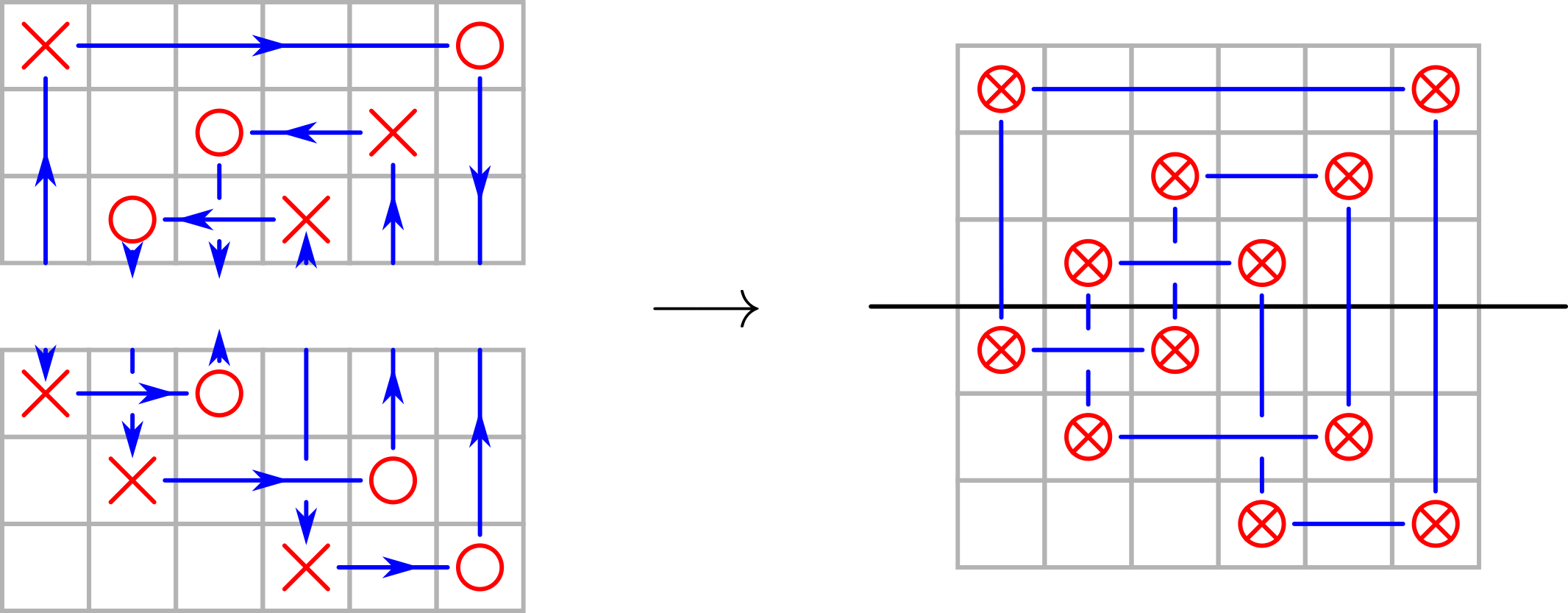}
	\put(-455, 140){$\mathbb{H}_+$}
	\put(-465, 30){$-\overline{\mathbb{H}_-}$}
	\put(-120, -10){$(\mathbb{H}_+, \mathbb{H}_-)$}
	\caption{Two half grid diagram $\mathbb{H}_+$ and $\mathbb{H}_-$ are the same as the ones in Figure \ref{fig: sym representation example}, which are not compatible, but we can still make a well-defined unoriented grid diagram $(\mathbb{H}_+,\mathbb{H}_-)$ and associated unoriented link $\mathcal{L}(\mathbb{H}_+,\mathbb{H}_-)$. In this specific example, $\mathcal{L}(\mathbb{H}_+,\mathbb{H}_-)$ is the right handed trefoil.}
	\label{fig: non compatible pair to unoriented grid example}
\end{figure}

We have unoriented counterparts of Definition \ref{def: half grid presentable} and Theorem \ref{thm: link is half grid presentable}.

\begin{definition}
	An unoriented link $L$ is said to be half grid presentable if there exists a pair of (not necessarily compatible) half grid diagrams $\mathbb{H}_+$ and $\mathbb{H}_-$ such that $\mathcal{L}(\mathbb{H}_+,\mathbb{H}_-)=L$ as unoriented links, and we call $(\mathbb{H}_+,\mathbb{H}_-)$ a half grid representative of $L$.
\end{definition}

\begin{proposition} \label{prop: unoriented link is half grid presentable}
	Every unoriented link is half grid presentable. 
\end{proposition}

\begin{proof}
	Suppose $L$ is an unoriented link, we give it an arbitrary orientation and denoted the resulting oriented link as $\vec{L}$. By Theorem \ref{thm: link is half grid presentable}, there exists a pair of compatible half grid diagrams $\mathbb{H}_+$ and $\mathbb{H}_-$ such that $\vec{L} = \mathcal{L}(\mathbb{H}_+, \mathbb{H}_-)$ as oriented links. Then $L = \mathcal{L}(\mathbb{H}_+, \mathbb{H}_-)$ as unoriented links.
\end{proof}

To prove Theorem \ref{thm: half grid presentation of link group}, we first recall how to give a presentation of link group $\mathcal{L}(\vec{\mathbb{G}})$ for some $m \times m$ grid diagram $\vec{\mathbb{G}}$.

\begin{lemma} \cite[Lemma 3.5.1]{MR3381987} \label{lem: grid presentation of link group}
	The link group $\pi_1(S^3 \backslash \mathcal{L}(\vec{\mathbb{G}}))$ of $\mathcal{L}(\vec{\mathbb{G}})$ has a presentation $\langle x_1,...,x_m \ | \ r_1,...,r_{m-1} \rangle$, called a grid presentation of link group. The generators $\{x_1,..., x_m\}$ correspond to the vertical segments in $\vec{\mathbb{G}}$ (connecting an $X$ and an $O$ in a column). The relations $\{r_1, ..., r_{m - 1}\}$ correspond to the horizontal lines separating the rows. The relation $r_j$ is the product of the generators corresponding to those vertical segments of the link diagram that meet the $j^{th}$ horizontal line, in the order they are encountered from left to right.
\end{lemma}

Note that link group is independent of link orientations, so we can replace the oriented grid diagram $\vec{\mathbb{G}}$ in Lemma \ref{lem: grid presentation of link group} by unoriented grid diagram $\mathbb{G}$. See Figure \ref{fig: grid presentation of link group example} for an example of grid presentation of $\mathcal{L}(\mathbb{G})$'s link group.

\begin{figure}
	\centering
	\begin{subfigure}{0.25\textwidth}
		\centering
		\includegraphics[width = \textwidth]{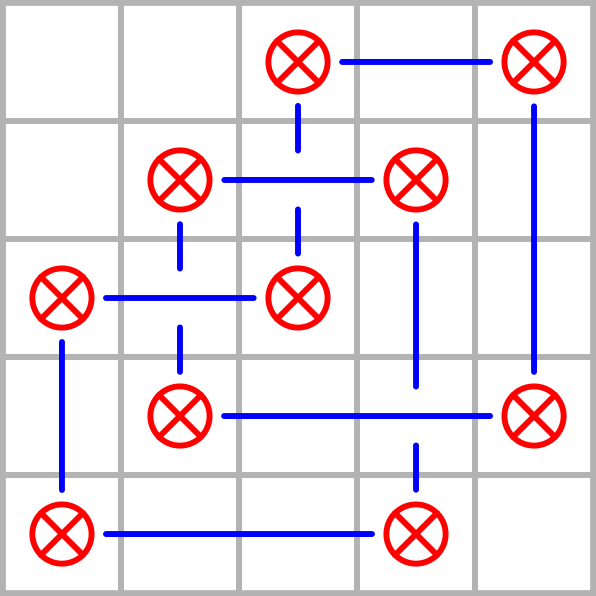}
		\put(-101, -15){$x_1$}
		\put(-80, -15){$x_2$}
		\put(-59, -15){$x_3$}
		\put(-38, -15){$x_4$}
		\put(-16, -15){$x_5$}
		\put(-170, 20){$r_1 = x_1 x_4$}
		\put(-193, 42){$r_2 = x_1 x_2 x_4 x_5$}
		\put(-193, 63){$r_3 = x_2 x_3 x_4 x_5$}
		\put(-170, 84){$r_4 = x_3 x_5$}
		\caption{}
		\label{fig: grid presentation of link group example}
	\end{subfigure}
	\hspace{1cm}
	\begin{subfigure}{0.3\textwidth}
		\centering
		\includegraphics[width = \textwidth]{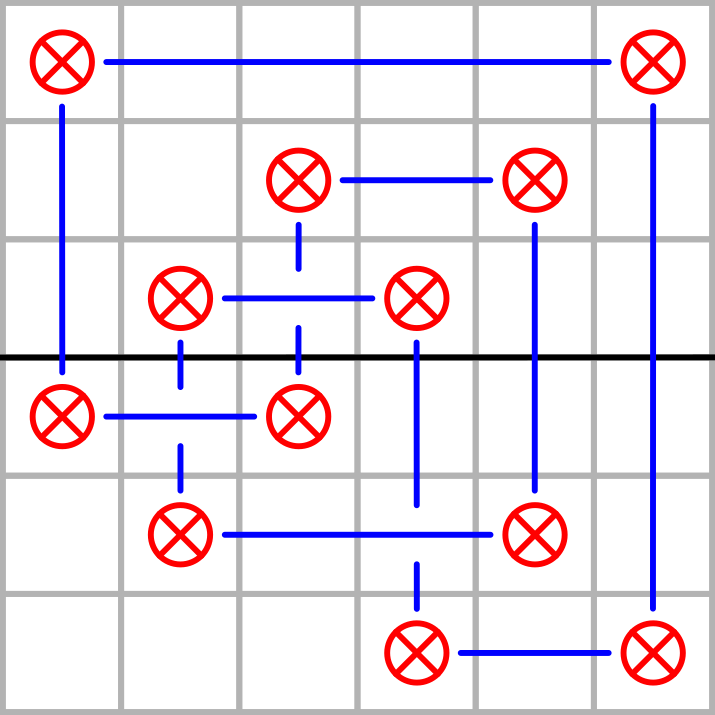}
		\put(-123, -15){$x_1$}
		\put(-101, -15){$x_2$}
		\put(-80, -15){$x_3$}
		\put(-59, -15){$x_4$}
		\put(-38, -15){$x_5$}
		\put(-16, -15){$x_6$}
		\put(12, 20){$r_{-2} = x_4 x_6$}
		\put(12, 42){$r_{-1} = x_2 x_4 x_5 x_6$}
		\put(12, 63){$r_0 = x_1 x_2 x_3 x_4 x_5 x_6$}
		\put(12, 83){$r_{+1} = x_1 x_3 x_5 x_6$}
		\put(12, 105){$r_{+2} = x_1 x_6$}
		\caption{}
		\label{fig: half grid presentation of link group example}
	\end{subfigure}
	\caption{(a) Given an unoriented $5 \times 5$ grid diagram $\mathbb{G}$, the link group of $\mathcal{L}(\mathbb{G})$ has a presentation $\langle x_1, x_2, x_3, x_4, x_5 \mid r_1, r_2, r_3, r_4 \rangle$. (b) Given two $3 \times 6$ half grid diagrams $\mathbb{H}_+, \mathbb{H}_-$ in Figure \ref{fig: non compatible pair to unoriented grid example}, we have an unoriented $6 \times 6$ grid diagram $(\mathbb{H}_+, \mathbb{H}_-)$. The link group of $\mathcal{L}(\mathbb{H}_+, \mathbb{H}_-)$ has a presentation $\langle x_1, x_2, x_3, x_4, x_5, x_6 \mid r_{-2}, r_{-1}, r_0, r_{+1}, r_{+2} \rangle$, where $r_0 = x_1 x_2 x_3 x_4 x_5 x_6$  is given by the middle horizontal line, and $r_{+1}, r_{+2}$ are given by the two horizontal lines in $\mathbb{H}_+$, $r_{-1}, r_{-2}$ are given by the two horizontal lines in $\mathbb{H}_-$.}
\end{figure}
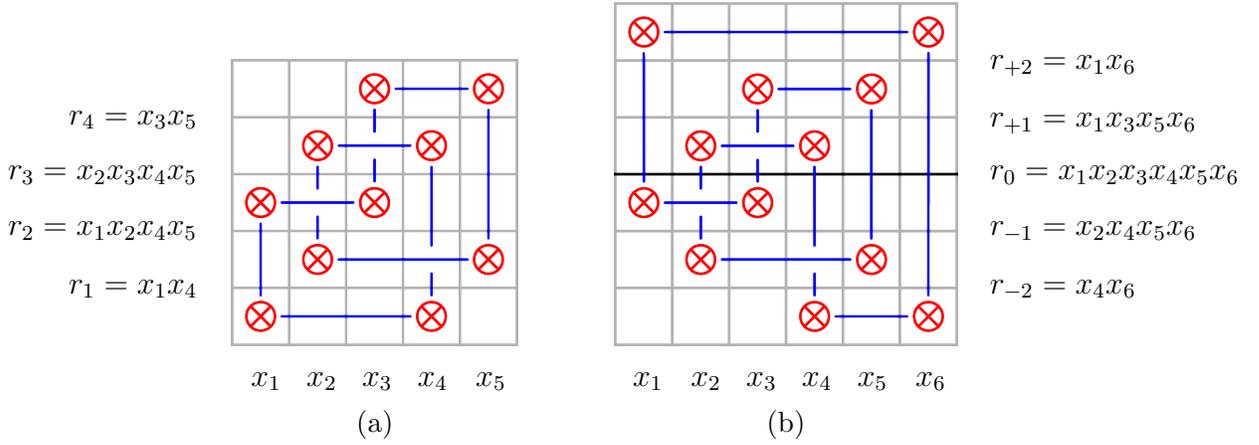

Now we are ready to prove Theorem \ref{thm: half grid presentation of link group}.  

\begin{proof}[Proof of Theorem \ref{thm: half grid presentation of link group}]
	Let's first prove that given any link, its link group has a half grid presentation (\ref{eq: half grid presentation}). Since the link group does not depend on the orientation of link, we can start with an unoriented link $L$. By Proposition \ref{prop: unoriented link is half grid presentable}, there exists a pair of $n \times 2n$ half grid diagrams $\mathbb{H}_+$ and $\mathbb{H}_-$ such that $\mathcal{L}(\mathbb{H}_+,\mathbb{H}_-)=L$.

	Then we claim that the presentation of link group of $\mathcal{L}(\mathbb{H}_+,\mathbb{H}_-)$ given by Lemma \ref{lem: grid presentation of link group} is the wanted half grid presentation. First, we index the horizontal line in the grid diagram $(\mathbb{H}_+,\mathbb{H}_-)$ from bottom to top with $-n$ to $n$ (so the middle lines is the $0^{th}$ one). Observe that the $0^{th}$ horizontal line intersects with all vertical segments (see Figure \ref{fig: half grid presentation of link group example} for an example), so we have
	$$
	r_0 = x_1 x_2 ... x_{2n} = X
	$$

	Other relations can be divided into two pieces $\{r_{+1},..,r_{+(n - 1)}\}$ and $\{r_{-1},..,r_{-(n - 1)}\}$ corresponding to $\mathbb{H}_+$ and $\mathbb{H}_-$ respectively. By the symmetric group representation given by Theorem \ref{thm: half grid and element of symmetric group}, $\mathbb{H}_{\pm}$ can be represented by some $\sigma_{\pm} \in \text{Sym}(2n)$, where the $X$ and $O$ in the $i^{th}$ row of $\mathbb{H}_{\pm}$ are exactly in the $\sigma_{\pm}(2i - 1)^{th}$ and $\sigma_{\pm}(2i)^{th}$ columns. Observe that in the unoriented grid diagram $(\mathbb{H}_+, \mathbb{H}_-)$, the $+1^{st}$ horizontal line intersects with all vertical segments except the $\sigma_+(1)^{th}$ one and the $\sigma_-(1)^{th}$ one, because the $X$ and the $O$ on the first row of $\mathbb{H}_+$ are below the $+1^{st}$ horizontal line of $(\mathbb{H}_+, \mathbb{H}_-)$. Then we take away $x_{\sigma_+(1)}$ and $x_{\sigma_+(2)}$ from $r_0 = x_1 x_2 ... x_{2n} =X$ to obtain $r_{+1} = X(x_{\sigma_+(1)}, x_{\sigma_+(2)})$. In general, we take away $x_{\sigma_+(2i - 1)}$ and $x_{\sigma_+(2i)}$ from $r_{+(i - 1)}$ to obtain $r_{+i}$. This argument works similarly for $r_{-1}, r_{-2},..., r_{-(n - 1)}$. In conclusion, we have
	\begin{align*}
		&r_{\pm 1} = X(x_{\sigma_{\pm}(1)}, x_{\sigma_{\pm}(2)}) \\
		&r_{\pm 2} = X(x_{\sigma_{\pm}(1)}, x_{\sigma_{\pm}(2)}, x_{\sigma_{\pm}(3)}, x_{\sigma_{\pm}(4)}) \\
		& \vdots \\
		&r_{\pm (n - 1)} = X(x_{\sigma_{\pm}(1)}, ..., x_{\sigma_{\pm}(2n-2)})
	\end{align*}
	
	We put relations $r_0, r_{\pm 1}, ..., r_{\pm(n - 1)}$ together to get the half grid presentation of link group.
	
	Conversely, suppose we have a group $G$ with a half grid presentation (\ref{eq: half grid presentation}), for some positive integer $n$ and two elements $\sigma_+$, $\sigma_-$ in $Sym(2n)$. By Theorem \ref{thm: half grid and element of symmetric group}, there is a pair of half grid diagrams $\mathbb{H}_+$ and $\mathbb{H}_-$ corresponding to $\sigma_+$ and $\sigma_-$ respectively. Notice that $(\mathbb{H}_+, \mathbb{H}_-)$ is not necessarily compatible, but we still have an unoriented grid diagram $(\mathbb{H}_+,\mathbb{H}_+)$. Let $L$ be the unoriented link $\mathcal{L}(\mathbb{H}_+,\mathbb{H}_+)$, then by previous discussion we conclude that the link group of $L = \mathcal{L}(\mathbb{H}_+,\mathbb{H}_+)$ has a presentation exactly same as the given half grid presentation of $G$. So $G$ is indeed a link group.
\end{proof}

\begin{example}
	The trefoil knot has a half grid representative $(\mathbb{H}_+,\mathbb{H}_-)$ shown in Figure \ref{fig: non compatible pair to unoriented grid example} and \ref{fig: half grid presentation of link group example}. In Example \ref{ex: sym representation} we saw that the corresponding symmetric group elements are $\sigma_+ = \begin{pmatrix} 1 \ 2 \ 3 \ 4 \ 5 \ 6 \\ 4 \ 2 \ 5 \ 3 \ 1 \ 6 \end{pmatrix}$ and $\sigma_- = \begin{pmatrix} 1 \ 2 \ 3 \ 4 \ 5 \ 6 \\ 3 \ 1 \ 5 \ 2 \ 6 \ 4 \end{pmatrix}$ respectively, so according to Theorem \ref{thm: half grid presentation of link group} the group
	$$
	\left< x_1,x_2,x_3,x_4,x_5,x_6\ \middle\vert \begin{array}{l}
		X=x_1x_2x_3x_4x_5x_6,\\
		X(x_{\sigma_+(1)},x_{\sigma_+(2)})=x_1x_3x_5x_6,\\
		X(x_{\sigma_+(1)},x_{\sigma_+(2)},x_{\sigma_+(3)},x_{\sigma_+(4)})=x_1x_6,\\
		X(x_{\sigma_-(1)},x_{\sigma_-(2)})=x_2x_4x_5x_6,\\
		X(x_{\sigma_-(1)},x_{\sigma_-(2)},x_{\sigma_-(3)},x_{\sigma_-(4)})=x_4x_6
	  \end{array} \right>
	$$
	is the knot group of trefoil.
\end{example}

\subsubsection{Half grid diagram and classical link invariants}

In this subsection we will first prove Theorem \ref{thm: Euler charactersitic of Thompson link}. The result follows directly from the construction assigning a canonical Seifert surface to an oriented Thompson link diagram introduced by Jones in \cite{MR3589908}. 

Given an s.d. partition pair $(\mathcal{I}, \mathcal{J})$, the signed regions in Figure \ref{fig: tangle orientation from signed regions example} give a surface $X_{(\mathcal{I}, \mathcal{J})}$ bounding $L_{(\mathcal{I}, \mathcal{J})}$. Furthermore, the $+$ and $-$ signs represent local orientations of the surface. Suppose that $(\mathcal{I}, \mathcal{J})$ represents an element in $\vec{F}$, then $\sigma_{\mathcal{I}} = \sigma_{\mathcal{J}}$ implies that $X_{(\mathcal{I}, \mathcal{J})}$ is orientable. Let it be the canonical Seifert surface of $L_{(\mathcal{I}, \mathcal{J})}$.

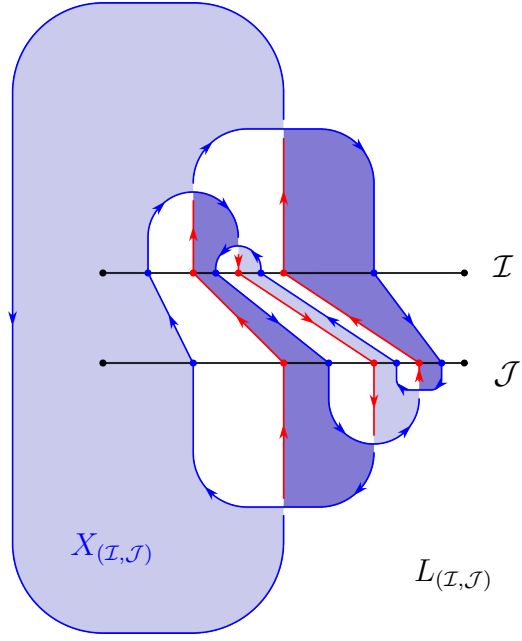
\begin{figure}
    \centering
    \includegraphics[width = 0.4\textwidth]{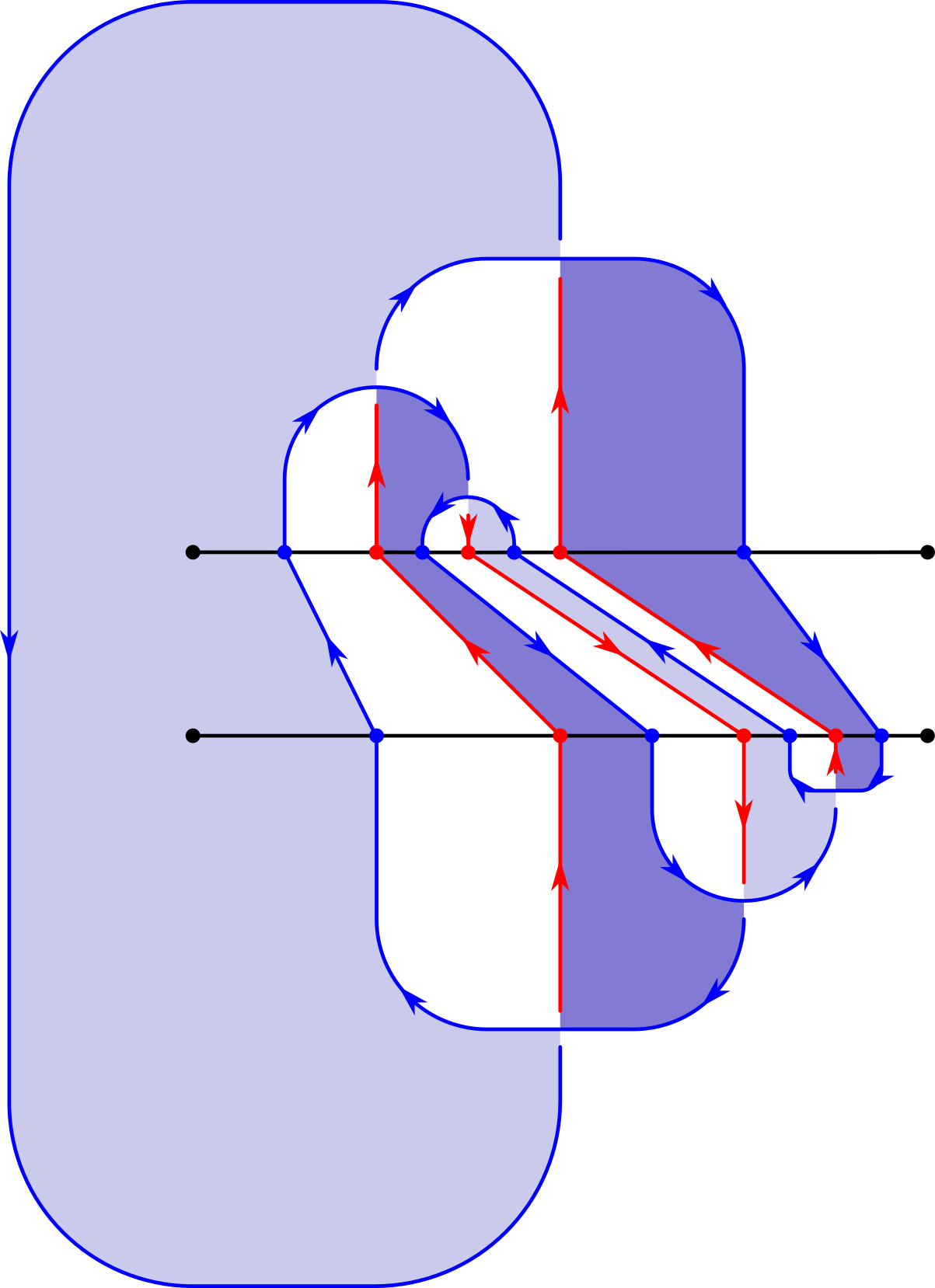}
    \put(10, 95){$\mathcal{J}$}
    \put(10, 135){$\mathcal{I}$}
    \put(-20, 20){$L_{(\mathcal{I}, \mathcal{J})}$}
    \put(-150, 30){$\color{blue} X_{(\mathcal{I}, \mathcal{J})}$}
    \caption{$L_{(\mathcal{I}, \mathcal{J})}$ with its canonical Seifert surface $X_{(\mathcal{I}, \mathcal{J})}$.}
    \label{fig:}
\end{figure}

Now we are ready to see the Proof of Theorem \ref{thm: Euler charactersitic of Thompson link}
\begin{proof}[Proof of Theorem \ref{thm: Euler charactersitic of Thompson link}]

Suppose $L$ has $ind_{\Vec{F}}=n$, let $(\mathcal{I}, \mathcal{J})$ be a compatible s.d. n-partition pair such that $\mathcal{L}(\mathbb{H}_{\mathcal{I}}, \mathbb{H}_{\mathcal{J}})= L_{(\mathcal{I},\mathcal{J})}=L$.Then there are total $2(n-1)$ crossing in such diagram, and by the construction of Seifert surface discussed above it's easy to see after we cut out the surface at the intersection points the surface decomposed into $n$ disks. In other words, the Seifert surface $S_{(\mathcal{I}, \mathcal{J})}$ is obtained by attaching $2(n-1)$ 1-handle to $n$ disks, so $\chi(S_{(\mathcal{I}, \mathcal{J})})=n-2(n-1)=-n+2$. Thus $\chi(L)\geq ind_{\Vec{F}}+2$, which implies when $L$ has one component $g(L)\leq \frac{ind_{\Vec{F}}-1}{2}$.
\end{proof}

Before we give proof of Theorem \ref{thm: grid number inequality}, let's first recall the definition of minimal grid number of oriented link. Suppose $\vec{L}$ is an oriented link. The minimal grid number of $\vec{L}$ is defined to be the minimal number $m$ such that there exists an $m \times m$ grid diagram $\vec{\mathbb{G}}$ satisfying $\mathcal{L}(\vec{\mathbb{G}}) = \vec{L}$. The minimal grid number of $\vec{L}$ is denoted as $grid(\vec{L})$.

Note that $grid(\vec{L})$ is independent of the orientation of $\vec{L}$, because we can always flip some $X$'s and $O$'s in $\vec{\mathbb{G}}$ to obtain any other orientation of $\mathcal{L}(\vec{\mathbb{G}})$. It is natural to define minimal grid number of unoriented link $L$.

\begin{definition}
	Let $L$ be an unoriented link. The minimal grid number of $L$ is defined to be the minimal number $m$ such that there exists an $m \times m$ unoriented grid diagram $\mathbb{G}$ satisfying $\mathcal{L}(\mathbb{G}) = L$. The minimal grid number of $L$ is denoted as $grid(L)$.
\end{definition}

It's easy to see that $grid(L) = grid(\vec{L})$, where $\vec{L}$ has an arbitrary orientation on $L$.

Next, let's see an unoriented version of Theorem \ref{thm: equivalent oriented links}.

\begin{proposition} \label{prop: equivalence of unoriented link}
	Suppose that $(\mathcal{I}, \mathcal{J})$ is any s.d. $n$-partition pair (not necessarily compatible), then $\mathcal{L}(\mathbb{H}_{\mathcal{I}}, \mathbb{H}_{\mathcal{J}})= L_{(\mathcal{I},\mathcal{J})}$ as unoriented links.
\end{proposition}

\begin{proof}
	Recall in Section \ref{sec: jones construction}, we had $L_{(\mathcal{I}, \mathcal{J})} = \overline{T_{\mathcal{J}}} T_{\mathcal{I}}$, where $T_{\mathcal{I}}, T_{\mathcal{J}}$ are unoriented $(0, 2n)$-tangles. Later, we assigned canonical orientations on $\vec{T}_{\mathcal{I}}$ and $\vec{T}_{\mathcal{J}}$ and showed $\mathcal{T}(\mathbb{H}_{\mathcal{I}}) = \widehat{\vec{T}_{\mathcal{I}}}$ and $\mathcal{T}(\mathbb{H}_{\mathcal{J}}) = \widehat{\vec{T}_{\mathcal{J}}}$ as oriented tangles (Proposition \ref{prop: equivalent oriented tangles}). Now we regard $\mathcal{T}(\mathbb{H}_{\mathcal{I}})$ and $\mathcal{T}(\mathbb{H}_{\mathcal{J}})$ as unoriented tangles, then we have
	$$
	\mathcal{L}(\mathbb{H}_{\mathcal{I}}, \mathbb{H}_{\mathcal{J}}) = \mathcal{T}(\mathbb{H}_{\mathcal{I}})\overline{\mathcal{T}(\mathbb{H}_{\mathcal{J}})} = \widehat{T_{\mathcal{I}}}\overline{\widehat{T_{\mathcal{J}}}} = {L}_{(\mathcal{I},\mathcal{J})}
	$$
\end{proof}

Now Theorem \ref{thm: grid number inequality} follows easily.

\begin{proof} [Proof of Theorem \ref{thm: grid number inequality}]
	Suppose $ind_F(L)=n$, then there exists an s.d. $n$-partition pair $(\mathcal{I},\mathcal{J})$ such that $L_{(\mathcal{I},\mathcal{J})}=L$. By Proposition \ref{prop: equivalence of unoriented link}, $L_{(\mathcal{I},\mathcal{J})}=\mathcal{L}(\mathbb{H}_{\mathcal{I}},\mathbb{H}_{\mathcal{J}})$ where $(\mathbb{H}_{\mathcal{I}},\mathbb{H}_{\mathcal{J}})$ is a $2n \times 2n$ unoriented grid diagram. Then we have
	$$
	grid(L) \leq 2n = 2 \cdot ind_F(L)
	$$
\end{proof}

\begin{appendices}

\section{Group of fractions}
\label{appx: group of fractions}

Let $\mathfrak{R}$ be a small category with the following 3 properties.

\begin{enumerate}
    \item (Unit) There is an element $1 \in Obj(\mathfrak{R})$ with $Mor_{\mathfrak{R}}(1, a) \neq \emptyset$ for all $a \in Ob(\mathfrak{R})$.
    \item (Stabilization) Let $\mathcal{D} = \bigcup_{a \in Obj(\mathfrak{R})} Mor_{\mathfrak{R}}(1, a)$. Then for each $f, g \in \mathcal{D}$ there are morphisms $p$ and $q$ with $pf = qg$.
    \item (Cancellation) If $pf = qf$ for $f \in \mathcal{D}$ then $p = q$.
\end{enumerate}

Then $\mathcal{D}$ becomes a directed set if we define $f \leq g$ given $g = pf$ for some morphism $p$.

Suppose that $\mathfrak{C}$ is another category and $\Phi: \mathfrak{R} \to \mathfrak{C}$ is a functor. For any $f \in \mathcal{D}$, we let $A(\Phi)_f = Mor_{\mathfrak{C}}(\Phi(1), \Phi(target(f)))$. If $f \leq g$ and $g = pf$, we define $\iota_f^g: A(\Phi)_f \to A(\Phi)_g$ sending $v$ to $\Phi(p) \circ v$. Then $(\{A(\Phi)_f\}, \{\iota_f^g\})$ is a direct system. Notice that the cancellation axiom implies that $p$ is unique, and the stabilization axiom implies that for any $f, g \in \mathcal{D}$ there exists some $h \in \mathcal{D}$ such that $f \leq h$ and $g \leq h$.

Let $\mathcal{Q} = \{(f, u): f \in \mathcal{D}, u \in A(\Phi)_f\}$. There is bijection from $\bigcup_{f \in \mathcal{D}} A(\Phi)_f$ to $\mathcal{Q}$, sending $u \in A(\Phi)_f$ to $(f, u)$, and its inverse sends $(f, u)$ back to $u \in A(\Phi)_f$.

The inverse limit of the direct system $(\{A(\Phi)_f\}, \{\iota_f^g\})$ is defined to be $\displaystyle\lim_{\begin{subarray}{l} \longrightarrow \\ f \in \mathcal{D} \end{subarray}} A(\Phi)_f = \bigcup_{f \in \mathcal{D}} A(\Phi)_f / \sim$, where $u  \in A(\Phi)_f$ and $v \in A(\Phi)_g$ are equivalent if there exists $h \geq f, g$ such that $\iota_f^h(u) = \iota_g^h(v)$. In other words, there exist morphisms $p, q$ in $\mathfrak{R}$ such that $p f = q g$ and $\Phi(p) \circ u = \Phi(q) \circ v$.

Correspondingly, there is an equivalence relation $\sim$ on $\mathcal{Q}$ such that $(f, u) \sim (g, v)$ if there exists morphisms $p, q$ in $\mathfrak{R}$ such that $(p f, \Phi(p) \circ u) = (q g, \Phi(q) \circ v)$. Then
$$
\displaystyle\lim_{\begin{subarray}{l} \longrightarrow \\ f \in \mathcal{D} \end{subarray}} A(\Phi)_f \cong \mathcal{Q} / \sim
$$
as sets. We don't distinguish them from now on.

Let $\Phi$ be the identity functor $I: \mathfrak{R} \to \mathfrak{R}$, then $\mathcal{Q} = \{(f, g): f, g \in \mathcal{D}, target(f) = target(g) \} $, and $(f_1, g_1) \sim (f_2, g_2)$ if there exists morphisms $p, q$ in $\mathfrak{R}$ such that $(p f_1, p g_1) = (q f_2, q g_2)$. Given $[f_1, g_1], [f_2, g_2] \in \displaystyle\lim_{\begin{subarray}{l} \longrightarrow \\ f \in \mathcal{D} \end{subarray}} A(I)_f$, we know there exist morphisms $p, q$ such that $p g_1 = q f_2$, then we define
$$
[f_1, g_1][f_2, g_2] = [p f_1, q g_2]
$$

This is a well-defined group structure on $\displaystyle\lim_{\begin{subarray}{l} \longrightarrow \\ f \in \mathcal{D} \end{subarray}} A(I)_f$, and this group is called the group of fractions of $\mathfrak{R}$, denoted as $G_{\mathfrak{R}}$.

Now we define the category of forests to recover Thompson group $F$ as a group of fractions.

\begin{definition}
    The category $\mathfrak{F}$ has objects positive numbers, with morphisms from $m$ to $n$ defined to be $(m, n)$-forests. Compositions are concatenations of forests.
\end{definition}

We can check that $\mathfrak{F}$ satisfies properties 1-3, and we have $G_{\mathfrak{F}} \cong F$ \cite[Proposition 2.2.1]{MR3764571}. To recover oriented Thompson group $\vec{F}$ as a group of fractions, we need to consider signed forests instead.

\begin{definition}
    The category $\vec{\mathfrak{F}}$ has objects $n$-signs, with morphisms from $\sigma$ to $\mu$ defined to be signed forests (forests with signing rules shown in Figure \ref{fig: region sign rule}) inducing $\sigma$ on the top and $\mu$ on bottom. Compositions are concatenations of signed forests.
\end{definition}

$\vec{\mathfrak{F}}$ also satisfies properties 1-3. Especially, let the unit of $\vec{\mathfrak{F}}$ be the $1$-sign $(+)$. Then we have $G_{\vec{\mathfrak{F}}} \cong \vec{F}$ \cite[Proposition 3.6]{MR3827807}. Notice that each element in $G_{\vec{\mathfrak{F}}}$ is a pair of signed trees $(f, g)$ from the $1$-sign $(+)$ to $target(f) = target(g)$, where the targets are the $n$-signs induced by $f$ and $g$ respectively. So the definition of $G_{\vec{\mathfrak{F}}}$ as a group of fractions coincides with Definition \ref{def: oriented Thompson group}.

\end{appendices}

\bibliographystyle{plain}
\bibliography{bib}

\begin{thebibliography}{10}

\bibitem{MR4071378}
Valeriano Aiello.
\newblock On the {A}lexander theorem for the oriented {T}hompson group {$\vec F$}.
\newblock {\em Algebr. Geom. Topol.}, 20(1):429--438, 2020.

\bibitem{MR3827807}
Valeriano Aiello, Roberto Conti, and Vaughan F.~R. Jones.
\newblock The {H}omflypt polynomial and the oriented {T}hompson group.
\newblock {\em Quantum Topol.}, 9(3):461--472, 2018.

\bibitem{MR1426438}
J.~W. Cannon, W.~J. Floyd, and W.~R. Parry.
\newblock Introductory notes on {R}ichard {T}hompson's groups.
\newblock {\em Enseign. Math. (2)}, 42(3-4):215--256, 1996.

\bibitem{MR2520400}
Jean-Marie Droz and Emmanuel Wagner.
\newblock Grid diagrams and {K}hovanov homology.
\newblock {\em Algebr. Geom. Topol.}, 9(3):1275--1297, 2009.

\bibitem{MR2179261}
John~B. Etnyre.
\newblock Legendrian and transversal knots.
\newblock In {\em Handbook of knot theory}, pages 105--185. Elsevier B. V., Amsterdam, 2005.

\bibitem{MR3589908}
Vaughan F.~R. Jones.
\newblock Some unitary representations of {T}hompson's groups {$F$} and {$T$}.
\newblock {\em J. Comb. Algebra}, 1(1):1--44, 2017.

\bibitem{MR3764571}
Vaughan F.~R. Jones.
\newblock A no-go theorem for the continuum limit of a periodic quantum spin chain.
\newblock {\em Comm. Math. Phys.}, 357(1):295--317, 2018.

\bibitem{MR3827056}
Robert Lipshitz, Peter~S. Ozsv\'{a}th, and Dylan~P. Thurston.
\newblock Bordered {H}eegaard {F}loer homology.
\newblock {\em Mem. Amer. Math. Soc.}, 254(1216):viii+279, 2018.

\bibitem{MR2480614}
Ciprian Manolescu, Peter Ozsv\'{a}th, and Sucharit Sarkar.
\newblock A combinatorial description of knot {F}loer homology.
\newblock {\em Ann. of Math. (2)}, 169(2):633--660, 2009.

\bibitem{MR2890458}
Lenhard Ng.
\newblock On arc index and maximal {T}hurston-{B}ennequin number.
\newblock {\em J. Knot Theory Ramifications}, 21(4):1250031, 11, 2012.

\bibitem{MR3381987}
Peter~S. Ozsv\'{a}th, Andr\'{a}s~I. Stipsicz, and Zolt\'{a}n Szab\'{o}.
\newblock {\em Grid homology for knots and links}, volume 208 of {\em Mathematical Surveys and Monographs}.
\newblock American Mathematical Society, Providence, RI, 2015.

\bibitem{szabo2019algebras}
Peter~S. Ozsv\'{a}th and Zolt\'{a}n Szab\'{o}.
\newblock Algebras with matchings and knot floer homology.
\newblock {\em preprint}, 2019.

\bibitem{ozsvath2022pong}
Peter~S. Ozsv\'{a}th and Zolt\'{a}n Szab\'{o}.
\newblock The pong algebra.
\newblock {\em preprint}, 2022.

\bibitem{MR2077671}
Olga Plamenevskaya.
\newblock Bounds for the {T}hurston-{B}ennequin number from {F}loer homology.
\newblock {\em Algebr. Geom. Topol.}, 4:399--406, 2004.

\bibitem{MR3584271}
Lawrence~P. Roberts.
\newblock A type {$A$} structure in {K}hovanov homology.
\newblock {\em Algebr. Geom. Topol.}, 16(6):3653--3719, 2016.

\bibitem{MR3474320}
Lawrence~P. Roberts.
\newblock A type {$D$} structure in {K}hovanov homology.
\newblock {\em Adv. Math.}, 293:81--145, 2016.

\end{thebibliography}

\end{document}